\documentclass[11pt,a4paper,oneside]{article} 
\usepackage[utf8]{inputenc}
\usepackage[T1]{fontenc}
\usepackage[affil-it]{authblk}

\usepackage{amsmath}
\usepackage{mathtools}
\usepackage{amsthm}
\usepackage{amssymb}
\usepackage{graphicx}
\usepackage{subfigure}
\usepackage[colorlinks=false]{hyperref}
\usepackage[usenames, dvipsnames]{color}
\usepackage{verbatim} 
\usepackage{setspace}
\usepackage{extsizes}
\usepackage{anysize} 
\usepackage{chngcntr}
\usepackage{fancyhdr} 

\usepackage{lastpage}
\usepackage[toc,title,page]{appendix}

\marginsize{1.5cm}{1.5cm}{0.8cm}{0.8cm}

\usepackage[normalem]{ulem} 

\newtheorem{proposition}{Proposition}

\newtheorem{lemma}{Lemma}
\newtheorem{definition}{Definition}

\newtheorem{remark}{Remark}
\usepackage{subfigure}
\usepackage{grffile}
\usepackage{caption}
\usepackage{latexsym}
\usepackage{graphpap}
\usepackage[english]{babel}
\usepackage{amsfonts}
\usepackage{graphicx}
\usepackage{amsmath,amsfonts,amssymb}
\numberwithin{equation}{section}

\newtheorem{assump}{Assumption}[section]

\begin{document}
\def\Box{ \framebox[5.5pt]}

\title{\centering \large \bf Decomposition of the Wave Manifold into Lax Admissible Regions and its 
Application to the  Solution of Riemann Problems}

\author{Cesar S. Eschenazi
  \thanks{email: \texttt{cseschenazi\emph{{@}}gmail.com}}}
\affil{Universidade Federal de Minas Gerais -- UFMG\\ Belo Horizonte, MG, CEP: 31270-901}
\author{Wanderson J. Lambert
  \thanks{email: \texttt{wanderson.lambert\emph{{@}}unifal-mg.edu.br}}}
\affil{Universidade Federal de Alfenas -- UNIFAL\\ Alfenas, MG, CEP: 37130-001}
\author{Marlon M. L\'{o}pez-Flores%
  \thanks{email: \texttt{mmlf\emph{{@}}impa.br}}}
\affil{Instituto Nacional de Matem\'{a}tica Pura e Aplicada -- IMPA\\ Rio de Janeiro, RJ, Brazil, CEP: 22460-320}
\author{Dan Marchesin%
  \thanks{email: \texttt{marchesi\emph{{@}}impa.br}}}
\affil{Instituto Nacional de Matem\'{a}tica Pura e Aplicada -- IMPA\\ Rio de Janeiro, RJ, Brazil, CEP: 22460-320}
\author{Carlos F.B. Palmeira
  \thanks{email: \texttt{fredpalm\emph{{@}}gmail.com} - Corresponding author.}}
\affil{Pontifícia Universidade Católica do Rio de Janeiro -- PUC-Rio\\
Rio de Janeiro, RJ - Brasil, CEP: 22451-900}
\date{}
\maketitle

\begin{abstract}

We utilize a three-dimensional manifold to solve Riemann Problems that
arise from a system of two conservation laws with quadratic flux functions. Points in this manifold represent potential shock waves, hence its name wave manifold. This manifold is subdivided into regions according to the Lax
admissibility inequalities for shocks. Finally, we present solutions for the Riemann 
Problems for various cases and exhibit continuity relative to $L$ and $R$ data, despite the fact that the system is not strictly hyperbolic. The usage of this manifold regularizes the solutions despite the presence of an elliptic region.
 

\end{abstract}
\tableofcontents

\section{Introduction\label{sec:introduction}}
We consider the system of two partial differential equations 
\begin{equation}
\left.
W_{_{t}}+F(W)_{_{x}}=0,
\label{eq:cons-law}
\right.
\end{equation}
with initial conditions
\begin{equation}
W(x,t=0)=\left\{
\begin{array}{ll}
W_{_{L}}&\hbox{\rm{if }}x < 0,\\
W_{_{R}} &\hbox{\rm{if }}x > 0.
\end{array}
\right.
\label{eq:initial-data}
\end{equation}

\noindent We will\ \ take\ \ $W=(u,v)\in \mathbb{R}^{^{2}}$\ \   and\ \  $F$\ \ a map\ \ $F:\mathbb{R}^{^{2}}\longrightarrow \mathbb{R}^{^{2}}$, the two components of which are called \emph{flux functions}. Equation \eqref{eq:cons-law} appears in fluid dynamics and together with initial conditions \eqref{eq:initial-data} is called a Riemann problem. 

We are mainly interested in the so called shock solutions, defined by
 \begin{equation*}
W=\left\{
\begin{array}{ll}
W &\hbox{\rm{for }}x < st,\\
W' &\hbox{\rm{for }}x > st,
\end{array}
\right.
\end{equation*}
where\ \ $s$\ \ is the  shock propagation speed,\ \ $W=(u,v)$\ \ and\ \ $W'=(u',v')$\ \ are the states to be connected by the shock.\\

In order to have physical meaning as a shock, the speed and the states must satisfy the so called \emph{ Rankine-Hugoniot} condition \cite{Smoller94}.
\begin{equation}\label{RH-condition}
F(W)-F(W')=s(W-W'),
\end{equation}
for some\ \ $s$. This leads to the definition of Hugoniot curve associated to a given state\ \ $W=(u_{_{0}},v_{_{0}})$,\ \ in Section  \ref{sec:wmhc}.

Furthermore, not all arcs of Hugoniot curves are useful to construct solutions. Shock curve arcs must satisfy some extra conditions called admissibility conditions. In \cite{AEMP10}, Liu's admissibility entropy criterion was introduced within the wave manifold context. There it was shown that under certain extra assumptions, which will be considered in the current paper, Liu's entropy criterion can be replaced by  the Lax entropy inequality conditions. These inequalities relate\ \ $s$\ \ and the eigenvalues of\ \ $DF(u,v)$. Hugoniot arc curves satisfying these inequalities are called \emph{ shock curve arcs} or \emph{ admissible arcs}.

We adopt here the topological point of view, as described in \cite{Isaacson92}. We consider the space\ \ $\mathbb{R}^{^{5}}$\ \ of coordinates\ \ $(u,v,u',v',s)$\ \ and in it the three-dimensional manifold defined by\ \ $F(u,v)-F(u',v')-s(u-u',v-v')=0$. In this manifold, we consider the curves defined by\ \ $(u,v)$\ \ constant, called Hugoniot curves, their projections onto the state space are the classical Rankine-Hugoniot curves, see \cite{AEMP10}. This manifold is called \emph{ wave manifold}.   

In a series of papers (see \cite{AEMP10, Eschenazi02, Marchesin94b, Palmeira88}), this manifold and its Hugoniot curves have been studied for the case where\ \ $F$\ \ is a polynomial of degree two. The wave manifold has been characterized, relevant surfaces ($characteristic$,\ \ $sonic$\ \ and\ \ $sonic'$) have been defined, the intersection of Hugoniot curves with these surfaces has been studied, and the Lax inequalities have been interpreted in this context. Here, also considering\ \ $F$\ \ as a polynomial of degree two, we decompose the\ \ $characteristic$\ \ and\ \ $sonic'$\ \ surfaces in their fast and slow components. And also decompose the wave manifold in regions which we call admissible or non-admissible regions. 

As stated at the beginning of Section \ref{sec:wmhc}, we will restrict ourselves to the symmetric Case IV of the
Sheaffer-Shearer classification, see \cite{Schaeffer87}.
A study of cases I, II, III and IV (non symmetric case) can be found in \cite{Lopez20}.

We will also consider rarefaction and composite solutions, to be defined in Section \ref{sec:rare-comp} and used in Sections \ref{subsec:WaveCurve} and  \ref{sec:RiemannSol}.

This paper is organized as follows: In Section \ref{sec:wmhc}, we review some basic facts and definitions, introduce new variables and describe the\ \ $characteristic$,\ \ $sonic$\ \ and\ \ $sonic'$\ \ surfaces, \cite{Isaacson92}. In Section \ref{sec:cscf}, we characterize the slow and fast components of the characteristic surface associated with the eigenvalues of\ \ $DF$. We also characterize the \emph{coincidence curve}, which is the boundary of these two components. In Section \ref{sec:waveman}, we describe how the\ \ $characteristic$,\ \ $sonic$\ \ and\ \ $sonic'$\ \ surfaces divide the wave manifold into twelve regions and characterize the surface formed by the Hugoniot curves through points of the \emph{ coincidence curve}.  This surface will be called { \emph{ saturated of the coincidence curve by Hugoniot curves}} and denoted by $SSC$. It is tangent to the\ \ $characteristic$\ \ and\ \ $sonic'$\ \ surfaces. In Section \ref{sec:sonlifs}, we characterize the slow and fast components of the\ \ $sonic'$\ \ surface, associated with the slow shock speed and the fast shock speed. We show that the boundaries of these two components are a straight line and the  \emph{hysteresis$'$ curve}, which is the curve where Hugoniot curves are tangent to the $sonic'$ surface. In Section \ref{sec:lax}, we identify some of regions in the wave manifold where Lax's inequalities are satisfied, indicating in which regions there are local shock curve arcs and in which of these regions  there are nonlocal shock curve arcs.
 
In Section \ref{sec:rare-comp}, we present the rarefaction and composite curves, derived from the rarefaction solutions mentioned before, and decompose the state space ($(u,v)$-plane) in elliptic and hyperbolic regions. In Section \ref{subsec:WaveCurve},  we introduce the wave curve in the wave manifold, such curves consist of arcs of admissible Hugoniot curves, rarefaction curves and composite curves of the same family. In Section \ref{sec:RiemannSol}, we use all the elements previously defined to solve Riemann problems.

It is known from \cite{Eschenazi02} that a Hugoniot curve through points of the secondary bifurcation has two components, a straight line and a curve. The lines generate a plane and the curves a surface, which we will call {\emph{ saturated of secondary bifurcation by Hugoniot curves}, denoted by $SSB$. In Appendix \ref{app:L2}, we characterize the regions of the wave manifold where condition \textbf{L2}, introduced in Section \ref{sec:lax}, is satisfied. In Appendix \ref{app:cho}, we justify the choices of the flux function, parameters and coordinates used in this paper.

\section{The Wave Manifold and Hugoniot Curves\label{sec:wmhc}}


In this section, we review some  basic facts and definitions and introduce new variables. We will define the $characteristic$, $sonic$ and $sonic'$ surfaces, since they will appear as boundaries of the admissible regions.

We refer the reader to Section 2 of \cite{Eschenazi02} for basic definitions. For the sake of completeness, we briefly present some necessary concepts needed here.

We consider the equation (\ref{eq:cons-law}) with\ \ $F$\ \ given by\ \ $F=(f,g)$,\ \ where
\begin{equation}
\left\{
\begin{array}{l}
f(u,v)=v^{^{2}}/2+(b_{_{1}}+1)u^{^{2}}/2+a_{_{1}}u+a_{_{2}}v,\\
\\
g(u,v)=uv+a_{_{3}}u+a_{_{4}}v,
\end{array}
\right.\label{fgeq}
\end{equation}
for\ \ $b_{_{1}}>1$. This is the symmetric Case IV  in the Schaeffer and Shearer classification, see \cite{Schaeffer87}.

Given a point\ \ $W=(u,v)$,\ \ in the (usually called) ``state space'', \emph{ the Hugoniot curve through this point is defined as the set of points\ \ $W'=(u',v')$,\ \ such that there exists\ \ $s$\ \ such that satisfies\ \ $F(W)-F(W')=s(W-W')$}. We clearly see that this curve passes through\ \ $W$.

To study Hugoniot curves, we consider $\mathbb{R}^{^{5}}=\{(u,v,u',v',s)\}$\ \ and in it the
three-dimensional manifold defined by\ \ $F(W)-F(W')=s(W-W')$. Since this manifold is singular along the diagonal\ \ $W=W'$,\ \ along this diagonal, we perform a blow up, which, in this simple case, is obtained using the coordinate transformations given by\ \ $U=(u+u')/2$,\ \ $V=(v+v')/2$,\ \ $X=u-u'$,\ \ $Y=v-v'$,\ \ and \ \ $Z=Y/X$\ \ and factoring\ \ $X^{^{2}}$. We also set\ \  $c=a_{_{3}}-a_{_{2}}>0$. Using these new coordinates, we get 
$$(1-Z^{^{2}})\left(V+a_{_{2}}\right)-Z\left(b_{_{1}}U+a_{_{1}}-a_{_{4}}\right)+c=0\ \ \ \ \ \ \ \text{and}\ \ \ \ \ \ \ Y=ZX.$$
The two above equations define a regular  three-dimensional manifold, which will be referred  as the \emph{wave manifold} and will be denoted by $\mathcal{W}$.

Since\ \ $Z$\ \ is a direction, we may think of the wave manifold\ \ $\mathcal{W}$,\ \ as contained in\ \ $\mathbb R^{^{4}} \times \mathbb RP^{^{1}}$,\ \ or, which is the same as,\ \ $\mathbb R^{^{4}}\times S^{^{1}}$. These coordinates are not valid at\ \ $Z=\infty$,\ \ but there are no special features at infinity, so, we can just use\ \ $X$,\ \ $U$,\ \ $V$,\ \ $Z$\ \ to study Hugoniot curves.

Defining\ \ $\widetilde{U}=b_{_{1}}U+a_{_{1}}-a_{_{4}}$\ \ and\ \ $\widetilde{V}=V+a_{_{2}}$ we obtain the equation $G_{_{1}}=(1-Z^{^{2}})\widetilde V-Z\widetilde{U}+c=0$,\ \ used in (\cite{Marchesin94b, Eschenazi02, Eschenazi13}).

Hugoniot curves in the wave manifold are defined by\ \ $u=\text{constant}$\ \ and\ \ $v=\text{constant}$,\ \ which, in these new coordinates, become 
\begin{equation}
\left \{
\begin{array}{l}
2\widetilde{U}+b_{_{1}}X=2b_{_{1}}u_{_{0}}+2(a_{_{1}}-a_{_{4}}),\\
2\widetilde{V}+ZX=2v_{_{0}}+2a_{_{2}}. 
\end{array}
\right.
\label{eq:KL}
\end{equation}

As shown in \cite{Marchesin94b}, Hugoniot curves are connected and foliate\ \ $\mathcal{W}$\ \ except along a straight line\ \ $B$,\ \ contained on the plane,\ \ $\Pi$,\ \ defined by\ \ $Z=0$. The straight line\ \ $B$\ \ is called \emph{ secondary bifurcation}. Hugoniot curves through points in\ \ $B$\ \ are formed by two arcs: a straight line\ \ $hugl$\ \ and a curve $hugc$. { These two arcs intersect at a point of \ \ $B$}. The lines \ \ $hugl$\ \ fill the plane $\Pi$.

\begin{definition}
 The set of Hugoniot curves through points of \ \ $hugc$\ \ is called  \emph{saturated of the secondary bifurcation by Hugoniot curves} and  denoted by $SSB$. 
\end{definition}

It is easy to see that $SSB$ is a two-dimensional sub-manifold of\ \ $\mathcal{W}$\ \ and that\ \ $SSB$\ \ and\ \ $\Pi$\ \ intersect transversely along\ \ $B$.

The same conclusions are valid for Hugoniot$'$ curves, defined by \ \ $u'=constant$\ \ and \ \  $v'=constant$. Hugoniot$'$ curves are connected and foliate \ \ $\mathcal{W}$\ \ except along a straight line \ \ $B'$,\ \ also contained on the plane  $\Pi$. Hugoniot$'$ curves through points in \ \ $B'$\ \ are formed by two arcs: a straight line \ \ $hugl'$\ \ and a curve  $hugc'$. { These two arcs intersect at a point of $B'$}. The lines \ \ $hugl'$\ \ also fill the plane $\Pi$. \ \ In a similar way the curves  $hugc'$ form a two-dimensional sub-manifold called \emph{ saturated of the secondary bifurcation by Hugoniot$'$ curves}, denoted by $SSB'$. The plane  $\Pi$ and intersect transversely along $B'$.

There are three important surfaces in the study of the wave manifold: Characteristic (denoted by \ \ $\mathcal{C}$), Sonic (denoted by\ \ $Son$) and Sonic$'$ (denoted by\ \ $Son'$). 



To define\ \ $Son$\ \ and\ \ $Son'$,\ \ one must look at the speed\ \ $s$\ \ as a real function in\ \ $\mathcal{W}$\ \ given by\break $[f(u,v)-f(u',v')]/(u-u')$\ \ or\ \ $[g(u,v)-g(u',v')]/(v-v')$. Using coordinates\ \ $X$,\ \ $\widetilde{U}$\ \ and\ \ $\widetilde{V}$,\ \ we have   
\begin{equation}
\left.  
s=\frac{\widetilde{U}}{b_{_{1}}}+ \frac{ \widetilde{V}}{Z},
\right.
\label{eq:sp1}
\end{equation}

To define\ \ $Son$,\ \ we restrict\ \ $s$\ \ to a Hugoniot curve, and look at its critical points. The set of these critical points for all Hugoniot curves is the \emph{ Sonic surface},\ \ $Son$. In the same way, we define the \emph{ $Sonic'$ surface},\ \ $Son'$,\ \ as the set of all critical points of\ \ $s$\ \ restricted to a Hugoniot$'$ curve. This was done in \cite{Marchesin94b}, obtaining equations
\begin{equation}\label{eq:eqson}
 2[Z^{^{2}}+b_{_{1}}+1]\widetilde{U}+2[Z^{^{3}}+(b_{_{1}}+3)Z]\widetilde{V}-[Z^{^{2}}-(b_{_{1}}+1)]X=0
\end{equation}
for $Son$ and 
\begin{equation}\label{eq:eqson'}
2[Z^{^{2}}+b_{_{1}}+1]\widetilde{U}+2[Z^{^{3}}+(b_{_{1}}+3)Z]\widetilde{V}+[Z^{^{2}}-(b_{_{1}}+1)]X=0 
\end{equation}
for\ \ $Son'$.\\

We list some important known facts about\ \ $\mathcal{C}$,\ \ $Son$\ \ and\ \ $Son'$,  for the quadratic flux functions considered here.

\noindent\textbf{Fact 1.-} $\mathcal{C}$\ \ is topologically a cylinder.\medskip

\noindent\textbf{Fact 2.-} Given a Hugoniot curve, either it intersects\ \ $\mathcal{C}$\ \ transversely in two points, or it is tangent or it does not intersect\ \ $\mathcal{C}$.\medskip

\noindent\textbf{Fact 3.-} The above mentioned set of tangency points form a simple closed curve. This curve does not bound a disk in\ \ $\mathcal{C}$\ \ and  divides it into two components, denoted by\ \ $\mathcal{C}_{_{s}}$\ \ and\ \ $\mathcal{C}_{_{f}}$. This simple closed curve is called \emph{coincidence curve in\ \ $\mathcal{C}$}, denoted by\ \ $\mathcal{E_C}$. \medskip

\noindent \textbf{Fact 4.-} Any Hugoniot curve through a point\ \ $\mathcal{U}$,\ \ $sh(\mathcal{U})$,\ \ which intersects\ \ $\mathcal{C}$\ \ transversely, does it at a point\ \ $\mathcal{U}_{_{s}}=sh(\mathcal{U})\cap \mathcal{C}_{_{s}}$\ \ and in another point\ \ $\mathcal{U}_{_{f}}=sh(\mathcal{U})\cap \mathcal{C}_{_{f}}$. The value of the speed\ \ $s$\ \  in\ \ $\mathcal{U}_{_{s}}$\ \ is smaller than the value of\ \ $s$\ \ in\ \ $\mathcal{U}_{_{f}}$. Proof of this fact will be presented in Section \ref{sec:cscf}, where we will characterize\ \ $\mathcal{C}_{_{s}}$\ \ and\ \ $\mathcal{C}_{_{f}}$.\medskip

\noindent \textbf{Fact 5.-} $Son$\ \ and\ \ $Son'$\ \ are topologically M\"{o}bius band and they intersect\ \ $\mathcal{C}$\ \ along a curve called \emph{ inflection locus}, denoted by\ \ $\mathcal{I}$,  which is diffeomorphic to\ \ $\mathbb{R}$, and contains the intersection point of\ \ $B$\ \ and\ \  $\mathcal{C}$. They also intersect along two straight lines, each of which intersects transversally\ \ $\mathcal{C}$\ \ at a point of the inflection locus.\medskip

\noindent \textbf{Fact 6.-} Generically, a Hugoniot curve through a point\ \ $\mathcal{U}$,\ \ $sh(\mathcal{U})$,\ \ intersects\ \ $Son'$\ \ in 0, 2 or 4 points. If\ \ $sh(\mathcal{U}) \cap Son'$\ \ has 2 points, then\ \ $s$\ \ has the same  value in these 2 points as the value of\ \ $s$\ \ either in\ \ $\mathcal{U}_{_{s}}$\ \ or in\ \ $\mathcal{U}_{_{f}}$. In the first case, the points are said to be in\ \ $Son'_{s}$\ \ and in the second case,\ \ in $Son'_{_{f}}$. If the intersection has 4 points, then\ \ $s$\ \ has the same value of\ \ $s$\ \ in\ \ $\mathcal{U}_{_{s}}$\ \ in two of them and the value of\ \ $s$\ \ in\ \ $\mathcal{U}_{_{f}}$\ \ in the other two. The first two are points in\ \ $Son'_{s}$\ \ and the other two in\ \ $Son'_{_{f}}$.

\subsection{New Coordinates \label{sec:Yzt-subsection}}

In this paper, we use coordinates\ \ $z=1/Z$,\ \ $t$\ \ and\ \ $Y$. The main advantage is that they are valid for\ \ $\mathcal{W}-\Pi$\ \  and\ \ $\Pi$\ \ becomes the plane\ \ $z=\infty$. To define\ \ $t$,\ \  we start by writing the equation of\ \ $\mathcal{W}$\ \ in term of the variables\ \ $\widetilde{U}$,\ \ $V_{_{1}}=V+a_{_{3}}=\widetilde{V}+c$\ \ and\ \ $z$. We get\ \ $G=0$,\ \  where $G=(z^{^{2}}-1)V_{_{1}}-z\widetilde{U}+c$. The parameter\ \ $t$\ \ is defined by
\begin{equation}\label{eq:tdef}
\widetilde{U}=\displaystyle \frac{2cz}{z^{^{2}}+1}+ct(z^{^{2}}-1)\ \ \ \ \ \ \ \text{and}\ \ \ \ \ \ \
 V_{_{1}}=\displaystyle \frac{c}{z^{^{2}}+1}+ctz.
\end{equation}

\noindent  For each fixed\ \ $z$,\ \ the parameter\ \ $t$\ \ measures as far as\ \ $\widetilde{U}$\ \ and\ \ $V_{_{1}}$\ \ are away from the { coincidence curve in $\mathcal{C}$} in the orthogonal direction to it.

In these new coordinates, the characteristic surface,\ \ $\mathcal{C}$,\ \  is the plane\ \ $Y=0$. The \emph{sonic}  and the $sonic'$ surfaces are given by the equations obtained from equations \eqref{eq:eqson} and \eqref{eq:eqson'} by transforming from variables\ \ $\widetilde{U}$,\ \ $\widetilde{V}$,\ \ $X$\ \ and\ \ $Z$\ \ to the new coordinates\ \ $z$,\ \ $t$\ \ and\ \ $Y$. We obtain\ \ $son=0$,\ \ where
\begin{equation}
son= -2c[(b_{_{1}}+1)z^{^{5}}+(b_{_{1}}+4)z^{^{3}}+3z]t-[(b_{_{1}}+1)z^{^{4}}+b_{_{1}}z^{^{2}}-1]Y-2c[(b_{_{1}}-1)z^{^{2}}+1],
\label{eq:son}
\end{equation}

\noindent and\ \ $son'=0$,\ \ where
\begin{equation}
son'=-2c[(b_{_{1}}+1)z^{^{5}}+(b_{_{1}}+4)z^{^{3}}+3z]t+[(b_{_{1}}+1)z^{^{4}}+b_{_{1}}z^{^{2}}-1]Y-2c[(b_{_{1}}-1)z^{^{2}}+1].
\label{eq:son'}
\end{equation}
\noindent
The shock  speed,\ \ $s$,\ \  is written as
\begin{equation}
s=\displaystyle\frac{c[(b_{_{1}}+1)tz^{^{4}}+b_{_{1}}z^{^{2}}t+(b_{_{1}}+2)z-t]}{b_{_{1}}(z^{^{2}}+1)}.
\label{eq:speed}
\end{equation}

Parametric equations for Hugoniot curves in\ \ $z$,\ \ $t$,\ \ $Y$\ \ coordinates are obtained by replacing\ \ $\widetilde{U}$,\ \ $\widetilde{V}$\ \ and\ \ $X$\ \ into equations in (\ref{eq:KL}), by their values in terms of\ \ $z$,\ \ $t$,\ \ $Y$\ \ and solving the system with respect to\ \ $t$\ \ and\ \ $Y$. We get:
\begin{equation}
\left\{\begin{array}{l}
z=z,\\	
t=\displaystyle \frac{(v_{_{0}}+a_{_{3}})b_{_{1}}z^{^{3}}+(b_{_{1}}u_{_{0}}+a_{_{1}}-a_{_{4}})z^{^{2}}-[b_{_{1}}(v_{_{0}}+a_{_{2}})+2c]z+b_{_{1}}u_{_{0}}+a_{_{1}}-a_{_{4}}}{c(z^{^{2}}+1)[(b_{_{1}}-1)z^{^{2}}+1]},\\	
Y=\dfrac{-2(v_{_{0}}+a_{_{3}})z^{^{2}}+2(b_{_{1}}u_{_{0}}+a_{_{1}}-a_{_{4}})z+2(v_{_{0}}+a_{_{2}})}{(b_{_{1}}-1)z^{^{2}}+1}.
\end{array}
\right.
\label{eq:parahug}
\end{equation}

 From the expression of\ \ $Y$\ \ in (\ref{eq:parahug}), we see that a Hugoniot curve intersects\ \ $\mathcal{C}$ depending on whether\ \ $\mathcal{E}_{_{ss}}=(b_{_{1}}u_{_{0}}+a_{_{1}}-a_{_{4}})^{^{2}}+4(v_{_{0}}+a_{_{3}})(v_{_{0}}+a_{_{2}})$\ \ is positive or negative: It intersects\ \  $\mathcal{C}$\ \ in two points if\ \ $\mathcal{E}_{_{ss}}>0$\ \ and does not intersect\ \ $\mathcal{C}$\ \ if\ \ $\mathcal{E}_{_{ss}}<0$. 

\begin{definition}
 The curve\ \ $\mathcal{E}_{_{ss}}=0$\ \ is called \emph{coincidence curve} in the state space, it is an ellipse. Its interior is called \emph{elliptic region} and its exterior is called \emph{hyperbolic region}.
\end{definition}

As we shall see in Subsection \ref{subsec:raref}, there are no rarefaction curves through points of the elliptic region and through every point of the hyperbolic region pass two transversal rarefaction curves.

Parametric equations for Hugoniot$'$ curves are obtained from equations (\ref{eq:parahug}) by interchanging\ \  $Y$\ \ with\ \ $-Y$. 

Substituting the second equation of (\ref{eq:parahug}) into equation (\ref{eq:speed}), we get the expression of the shock speed\ \  $s=s(u_{_{0}}, v_{_{0}},z)$\ \  along a Hugoniot curve.

In order to get the expression of a Hugoniot curve through a given point\ \  $(z_{_{0}},t_{_{0}},Y_{_{0}})$,\ \  we start by solving system  \eqref{eq:parahug} in terms of\ \  $u_{_{0}}$\ \  and\ \  $v_{_{0}}$\ \  obtaining
\begin{equation}\label{eq:tranf-u0v0}
\left\{
\begin{array}{l}
u_{_{0}}(z,t,Y)=\dfrac{b_{_{1}}z(z^{^{2}}+1)Y+2c(z^{^{4}}-1)t-2(a_{_{1}}-a_{_{4}})z^{^{2}}+4cz-2(a_{_{1}}-a_{_{4}})}{2b_{_{1}}(z^{^{2}}+1)},\\\\
v_{_{0}}(z,t,Y)=\dfrac{(z^{^{2}}+1)Y+2cz(z^{^{2}}+1)t-2a_{_{2}}(z^{^{2}}+1)-2cz^{^{2}}}{2(z^{^{2}}+1)}.
\end{array}
\right. 
\end{equation}

Substituting\ \ $u_{_{0}}(z_{_{0}},t_{_{0}},Y_{_{0}})$\ \  and\ \ $v_{_{0}}(z_{_{0}},t_{_{0}},Y_{_{0}})$\ \ from (\ref{eq:tranf-u0v0}) into equation  \eqref{eq:parahug}, we get the parametric equations for the Hugoniot curve through a point\ \  $(z_{_{0}},t_{_{0}},Y_{_{0}})$,

\begin{equation}
\left\{
\begin{array}{l}
z=z,\\
t=\displaystyle \frac{Dz^{^{3}}+Ez^{^{2}}+Fz+G}{c(z_{_{0}}^{^{2}}+1)[(b_{_{1}}-1)z^{^{4}}+b_{_{1}}z^{^{2}}+1]},\\	
Y=\displaystyle \frac{Az^{^{2}}+Bz +C}{(z_{_{0}}^{^{2}}+1)[(b_{_{1}}-1)z^{^{2}}+1]}.
\end{array}
\right.
\label{eq:hugponto}
\end{equation}

\noindent where
\begin{eqnarray*}
A&=&-[2c+2ct_{_{0}}z_{_{0}}(1+z_{_{0}}^{^{2}})+Y_{_{0}}(z_{_{0}}^{^{2}}+1)],\\
B&=&4cz_{_{0}}+2ct_{_{0}}(z_{_{0}}^{^{4}}-1)+b_{_{1}}Y_{_{0}}z_{_{0}}(z_{_{0}}^{^{2}}+1),\\
C&=&-2cz_{_{0}}^{^{2}}+2ct_{_{0}}z_{_{0}}(1+z_{_{0}}^{^{2}})+Y_{_{0}}(z_{_{0}}^{^{2}}+1),\\
D&=&2cb_{_{1}}+2cb_{_{1}}t_{_{0}}z_{_{0}}(z_{_{0}}^{^{2}}+1)+b_{_{1}}Y_{_{0}}(z_{_{0}}^{^{2}}+1),\\
E&=&2ct_{_{0}}(1-z_{_{0}}^{^{4}})-b_{_{1}}z_{_{0}}Y_{_{0}}(z_{_{0}}^{^{2}}+1)-4cz_{_{0}},\\
F&=&4c(z_{_{0}}^{^{2}}+1)+2cb_{_{1}}t_{_{0}}z_{_{0}}(1+z_{_{0}}^{^{2}})+b_{_{1}}Y_{_{0}}(z_{_{0}}^{^{2}}+1)-2cb_{_{1}}z_{_{0}}^{^{2}},\\
G&=&-4cz_{_{0}}+2ct_{_{0}}(1-z_{_{0}}^{^{4}})-b_{_{1}}z_{_{0}}Y_{_{0}}(1+z_{_{0}}^{^{2}}).
\end{eqnarray*}

It is clear from the above expressions that, given a point\ \  $(z_{_{0}},t_{_{0}},Y_{_{0}}) \in \mathcal{W}$, its Hugoniot curve will intersect\ \ $\mathcal{C}$\ \ or not, depending on whether\ \ $B^{^{2}} -4AC$\ \ is positive or negative.  If\ \ $B^{^{2}} -4AC =0$, the Hugoniot curve will be tangent to\ \ $\mathcal{C}$\ \ along the \emph{ coincidence curve in\ \ $\mathcal{C}$}, denoted by\ \ $\mathcal{E_C}$. Let us see that in these new coordinates the coincidence curve in $\mathcal{C}$ is given simply by\ \ $t=0$, and, of course,\ \ $Y=0$. Straightforward computation shows that for\ \ $Y=0$,\ \  $B^{^{2}} -4AC=  4c^{^{2}} t{_{_{_{0}}}}^{^{2}}(1+z_{_{0}}^{^{2}})^{^{4}}$,\ \ and also computing\ \ $dY/dz$\ \ from (\ref{eq:hugponto}) and setting\ \ $t_{_{0}}=0$, $Y_{_{0}}=0$,\ \ we get 0, which shows that along the curve\ \ $Y=0$, $t=0$. The Hugoniot curves are tangent to the characteristic surface (the plane\ \ $Y=0$\ \ in these coordinates). So, the line\ \ $(z,t=0,Y=0)$, is the coincidence curve in\ \ $\mathcal{C}$. Let us call\ \  $\mathcal{C}_{_{s}}$\ \ the\ \ $t<0$\ \ region and\ \ $\mathcal{C}_{_{f}}$\ \ the\ \ $t>0$\ \ region.

\section{Characterizing $\mathcal{C}_{_{s}} $ and $\mathcal{C}_{_{f}}$\label{sec:cscf}}
\vspace{.5cm}

As we will see in Section \ref{sec:lax}, shock curve arcs will either start on\ \ 
$\mathcal{C}_{_{s}}$\ \ or on\ \ $Son'_{s}$. So we must characterize\ \ $\mathcal{C}_{_{s}}$\ \ and\ \ $\mathcal{C}_{_{f}}$. In Section \ref{sec:sonlifs}, we will define and characterize $Son'_{s}$ and $Son'_{_{f}}$.

The speed\ \ $s$\ \ along a Hugoniot curve through a point\ \ $(z_{_{0}},t_{_{0}},Y_{_{0}})$\ \ is obtained by substituting\ \ $u_{_{0}}(z_{_{0}},t_{_{0}},Y_{_{0 }})$\ \ and\ \ $v_{_{0}}(z_{_{0}},t_{_{0}},Y_{_{0}})$\ \ into the expression\ \ \  $s=s(u_{_{0}}, v_{_{0}},z)$\ \ giving 

\begin{equation}
s_{_{hug}}= \frac{s_{_{p3}}z^{^{3}}-s_{_{p2}}z^{^{2}}-s_{_{p1}}z+s_{_{p0}}}{2b_{_{1}}(z_{_{0}}^{^{2}}+1)[(b_{_{1}}-1)z^{^{2}}+1]},
\label{eq:velhug}
\end{equation}
\noindent where
\begin{eqnarray*}
s_{_{p3}}&=&b_{_{1}}(b_{_{1}}+1)\left\{Y_{_{0}}(z_{_{0}}^{^{2}}+1)+2c[1+z_{_{0}}t_{_{0}}(z_{_{0}}^{^{2}}+1)]\right\},\\
s_{_{p2}}&=&(b_{_{1}}+1)[b_{_{1}}z_{_{0}}Y_{_{0}}(z_{_{0}}^{^{2}}+1)+2c(2z_{_{0}}+t_{_{0}}z_{_{0}}^{^{4}}-t_{_{0}})],\\
s_{_{p1}}&=&b_{_{1}}\left\{Y_{_{0}}(z_{_{0}}^{^{2}}+1)+2c[z_{_{0}}t_{_{0}}(z_{_{0}}^{^{2}}+1)-2z_{_{0}}^{^{2}}-1]\right\},\\
s_{_{p0}}&=&b_{_{1}}z_{_{0}}Y_{_{0}}(z^{^{2}}+1)+2c(2z_{_{0}}+t_{_{0}}z_{_{0}}^{^{4}}-t_{_{0}}).
\end{eqnarray*}
The equation of the shock speed\ \ $s$\ \  along a Hugoniot curve through a point\ \ $(z_{_{0}},t_{_{0}},0)$\ \ on the characteristic surface (characteristic speed) is obtained by substituting\ \ $Y_{_{0}}=0$\ \ into equation \eqref{eq:velhug}

\begin{equation}
s_{_{ch}}=\displaystyle \frac{s_{_{ch3}}z^{^{3}}- s_{_{ch2}}z^{^{2}}- s_{_{ch1}} z+s_{_{ch0}}}{b_{_{1}}(z_{_{0}}^{^{2}}+1)[(b_{_{1}}-1)z^{^{2}}+1]},
\label{eq:velcar}
\end{equation}
\noindent where
\begin{eqnarray*}
s_{_{ch3}}&=&b_{_{1}}c[t_{_{0}}z_{_{0}}(1+z_{_{0}}^{^{2}})+1](1+b_{_{1}}),\\
s_{_{ch2}}&=&c(t_{_{0}}z_{_{0}}^{^{4}} +2z_{_{0}}-t_{_{0}})(1+b_{_{1}}),\\
s_{_{ch1}}&=&b_{_{1}}c(t_{_{0}}z_{_{0}}^{^{3}}-2z_{_{0}}^{^{2}}+t_{_{0}}z_{_{0}}-1),\\
s_{_{ch0}}&=&c(t_{_{0}}z_{_{0}}^{^{4}} +2z_{_{0}}-t_{_{0}}).
\end{eqnarray*}

\begin{lemma}

If\ \ $(z_{_{0}} ,t_{_{0}},0)$\ \ and\ \ $(z_{_{1}} ,t_{_{1}},0)$\ \ are the intersections of a Hugoniot curve with\ \ $\mathcal{C}$,\ \  then, one of them is on\ \ $\mathcal{C}_{_{s}}$\ \ and the other on\ \ $\mathcal{C}_{_{f}}$. Furthermore, the value of\ \ $s$\ \ at the point on\ \ $\mathcal{C}_{_{s}}$\ \ is always smaller then the value of\ \ $s$\ \ at the point on\ \ $\mathcal{C}_{_{f}}$ (thus the subscripts $s$ stand for \emph{slow} and $f$ for \emph{fast}).

\label{lemma:cf-cs}
\end{lemma}
\begin{proof}
A Hugoniot curve through a point on the characteristic  has parametric equations
\begin{equation}
\left\{
\begin{array}{l}
z=z,\\	
t=\displaystyle \frac{D_{_{0}}z^{^{3}}+E_{_{0}}z^{^{2}}+F_{_{0}}z+G_{_{0}}}{c\left(z_{_{0}}^{^{2}}+1\right)[(b_{_{1}}-1)z^{^{4}}+b_{_{1}}z^{^{2}}+1]},\\	
Y=\displaystyle \frac{A_{_{0}}z^{^{2}}+B_{_{0}}z+C_{_{0}}}{(z_{_{0}}^{^{2}}+1)[(b_{_{1}}-1)z^{^{2}}+1]}.
\end{array}
\right.
\label{eq:hugch}
\end{equation}
\noindent where\ \ $A_{_{0}}$,\ $B_{_{0}}$,\ $C_{_{0}}$,\ $D_{_{0}}$,\  $E_{_{0}}$,\ $F_{_{0}}$\ \ and\ \ $G_{_{0}}$\ \ are obtained from (\ref{eq:hugponto}). It is just a matter of making\ \ $Y_{_{0}}=0$\ \ in  equations (\ref{eq:hugponto}). 

Solving equation\ \ $Y=0$\ \ we get\ \  $z_{_{0}}$\ \ and\ \ $z_{_{1}}=-[t_{_{0}}(1+z_{_{0}}^{^{2}})-z_{_{0}}]/[t_{_{0}}z_{_{0}}(1+z_{_{0}}^{^{2}})+1]$,\ \  values of\ \  $z$\ \ at the intersection points with\ \ $\mathcal{C}$.\ \ It is easy to see that\ \ $z_{_{0}}=z_{_{1}}$\ \ if and only if \ \ $t_{_{0}}=0$,\ \ characterizing  the \emph{coincidence curve}. 

Replacing\ \ $z$\ \ by\ \ $z_{_{1}}$\ \ in the expression of\ \ $t$\ \ in equations (\ref{eq:hugch}), we get the expression of\ \ $t_{_{1}}$. A straightforward computation shows that\ \ $t_{_{0}}\cdot t_{_{1}}<0$, so, an intersection point is on\ \ $\mathcal{C}_{_{s}}$\ \ and the other one on\ \ $\mathcal{C}_{_{f}}$.

Let\ \ $sch0$\ \ and\ \ $sch1$\ \  be the characteristic speeds for\ \ $z=z_{_{0}}$\ \ and\ \  $z=z_{_{1}}$. Straightforward computations give\ \ $sch0-sch1=c(z_{_{0}}^{^{2}}+1)t_{_{0}}$. It follows that\ \ $ sch0>sch1$\ \ if and only if\ \ $t_{_{0}}>0$. So, the value of\ \ $s$\ \ at the point on\ \ $\mathcal{C}_{_{f}}$\ \ is larger  than the value of\ \ $s$\ \ at the point on\ \ $\mathcal{C}_{_{s}}$.
\end{proof}

\section{Decomposition of the Wave Manifold\label{sec:waveman}}

In this section, we describe how the \emph{characteristic}, \emph{sonic} and $sonic'$ surfaces divide\ \ 
$\mathcal{W}-\Pi$.

Recall that, here,\ \ $\mathcal{W}-\Pi$\ \ is just the\ \ $(z,t,Y)$ space.  Let\ \ $\Pi_{_{z}}$\ \  be the fixed\ \ $z$\ \ plane. In it, we will consider the\ \ $t$\ \ and\ \ $Y$\ \ axis inherited from the\ \ $(z,t,Y)$ space. We will consider\ \ $t$\ \ to be the horizontal axis and\ \ $Y$\ \ to be the vertical axis.

It follows from equations \eqref{eq:son} and \eqref{eq:son'} that \ \ $Son$\ \ and\ \ $Son'$\ \ are ruled surfaces. Their intersections with\ \ $\Pi_{_{z}}$\ \ are straight lines. We will see how these intersections vary with\ \ $z$. The intersection\ \ $\mathcal{C} \cap \Pi_{_{z}}$\ \  is just the\ \ $t$-axis, \emph{i.e.}, the line\ \ $Y=0$.
 
Let\ \ $Sz = Son \cap \Pi_{_{z}}$\ \   and\ \ $S’z=Son’ \cap \Pi_{_{z}}$. For a fixed\ \ $z$,\  $Sz$  and  $S'z$ intersect at the point\ \ $(t=t_{_{0}}, Y=0)$,\ $t_{_{0}}$\ \  given in equation \eqref{eqinfle}. Letting\ \ $z$\ \ vary, these points in\ \ $\mathcal{C}$\ \ are on the inflection locus curve, $\mathcal{I}$, whose expression is given by:
\begin{equation}
t=\frac{(b_{_{1}}-1)z^{^{2}}+1}{z[(b_{_{1}}+1)z^{^{4}}+(b_{_{1}}+4)z^{^{2}}+3]}.\label{eqinfle}
\end{equation}
 For\ \  $z=0$,\ $Sz$\ \ and\ \ $S'z$, become horizontal lines  with equations\ \ $Y=2c$\ \  and\ \  $Y=-2c$, respectively.  Figure \ref{fig:01} illustrates typical relative positions of\ \ $Sz$\ \ and\ \ $S'z$\ \ in the\ \ $(t,Y)$-plane for some fixed values of\ \ $z$,\ $z<0$.
	

{ Besides intersecting at the inflection locus,\ \ $Son$\ \ and\ \ $Son'$\ \ intersect for\ \ 
\begin{equation}
z_{_{crit_{_{1}}}}= \frac{1}{ \sqrt{b_{_{1}}+1}} \quad\text{ and }\quad 
 z_{_{crit_{_{2}}}}=-\frac{1}{\sqrt{b_{_{1}}+1}}, \label{doublecon}
\end{equation}
\ \ values of\ \ $z$\ \ that cancel out the coefficient of\ \ $Y$\ \ in equations \eqref{eq:son} and \eqref{eq:son'}.  Straightforward computations show that for\ \  $z= z_{_{crit_{_{1}}}}$, we get\ \ $t_{_{1}}=-b_{_{1}} \sqrt{b_{_{1}}+1}/[2(b_{_{1}}+2)]$\ \ and for\ \ $z=z_{_{crit_{_{2}}}}$, we get\break $t_{_{2}}= b_{_{1}} \sqrt{b_{_{1}}+1} / [2(b_{_{1}}+2)]$. So, for these values of\ \  $z$,\ $Sz$\ \ and\ \ $S'z$, coincide and become the vertical lines\ \ $t=t_{_{1}}$\ \ and\ \ $t=t_{_{2}}$}. In the wave manifold\ \ $\mathcal{W}$\ \ we have two straight  lines\ \ $(z=z_{_{crit_{_{1}}}},t=t_{_{1}},Y)$\ \ and\ \ $(z=z_{_{crit_{_{2}}}},t=t_{_{2}},Y)$.

\begin{definition}
The straight lines\ \ $(z=z_{_{crit_{_{1}}}},t=t_{_{1}},Y)$\ \ and\ \ $(z=z_{_{crit_{_{2}}}},t=t_{_{2}},Y)$\ \ are called \emph{double sonic locus}.
 \end{definition}
 
 In this way,\ \  $Son \cap Son'$\ \ is formed by the  inflection locus and  the double sonic locus, \cite{Isaacson92}. 
 
 \begin{remark}
 In papers with classical approach the double sonic locus is called \emph{double contact}. 
 \end{remark}

\newpage 

\noindent
We can state,
\begin{proposition}
The characteristic, sonic and $sonic'$ surfaces divide\ \ $\mathcal{W}-\Pi$\ \ into twelve regions. 
\label{proposition:divideM}

\end{proposition}
\begin{proof}
It is enough to consider the half-space\ \ $Y>0$,\ \ since a symmetric division will appear in the\ \ $Y<0$\ \ half-space. For each fixed\ \ $z$\ \ we will look at the two-dimensional regions in which the plane\ \ $\Pi_{_{z}}$\ \ is divided and see how these two-dimensional regions form 3-dimensional regions as\ \ $z$\ \ moves. As described above, there are 3 critical values in the\ \ $z$ movement:\ \ $z=z_{_{crit_{_{2}}}}$,\ \ $z=0$\ \ and\ \ $z=z_{_{crit_{_{1}}}}$. The equation of\ \ $Sz$\ \ is obtained by solving \eqref{eq:son} for\ \ $Y$.





Let\ \ $(t_{_{0}},0)$\ \ be the intersection point of\ \ $Sz$\ \ and the\ \ $t-$axis. Taking\ \ $Y=0$\ \ in equation  \eqref{eq:son} and solving for\ \ $t$\ \ we have
\begin{equation}
t_{_{0}}= -\frac{[(b_{_{1}}-1)z^{^{2}}+1]}{z[(b_{_{1}}+1)z^{^{4}}+(b_{_{1}}+4)z^{^{2}}+3]}.
\label{eq:t0}
\end{equation}
\noindent
We see that\ \ $|t_{_{0}}|$\ \ increases as\ \ $z \to 0$. 

The straight line\ \  $S’z$\ \ obtained from  \eqref{eq:son'}  also goes through\ \ $(t_{_{0}},0)$\ \ and its slope is the slope of\ \ $Sz$\ \ with reversed sign.

Let us draw\ \  $Sz$\ \ and\ \ $S'z$\ \ for values of\ \ $z$\ \ going from\ \ $z<z_{_{crit_{_{2}}}}$\ \ 
to\ \ $z>z_{_{crit_{_{1}}}}$\ \ in the half-plane\ \ $Y>0$.

See Figure \ref{fig:01}. We start with\ \ $z<z_{_{crit_{_{2}}}}$. For a fixed\ \ $z$\ \ in this region,\ \ $Sz$,\ $S'z$\ \ and the\ \ $t-$axis define three regions in the half plane\ \ $Y>0$:\ \ $tS'z$ bounded by the\ \ $t-$axis and\ \ $S'z$;\ \ $S'zSz$\ \ bounded by\ \ $S'z$\ \ and\ \ $Sz$\ \ and\ \ $tSz$\ \ bounded by\ \ $Sz$\ \ and the\ \ $t-$axis. As\ \ $z$\ \  approaches the critical value,\ \ $Sz$\ \ and\ \ $S'z$\ \ become vertical, so, the region \ \ $SzS'z$\ \ for\ \  $z<z_{_{crit_{_{2}}}}$\ \ disappears. When\ \ $z$\ \ crosses the critical value,\ \ $Sz$\ \ and\ \ $S'z$\ \ just change their relative positions. So, in each side of the vertical line, we have three regions\ \ $tSz$,\ $SzS'z$\ \  and\ \ $tS'z$\ \ in the half plane\ \ $Y>0$. In Figure \ref{fig:01}, we show\ \ $Sz$\ \ and\ \ $S'z$\ \ for three different values of $z$, one value smaller than\ \ $z_{_{crit_{_{2}}}}$,\ $z= z_{_{crit_{_{2}}}}$\ \ and one bigger than\ \  $z_{_{crit_{_{2}}}}$, but still smaller than zero. To draw Figure \ref{fig:01}, we take\ \ $b_{_{1}}=8$,\ $c=1$. 
\begin{figure}[htpb]
	\begin{center}
		\includegraphics[scale=0.8,width=0.5\linewidth]{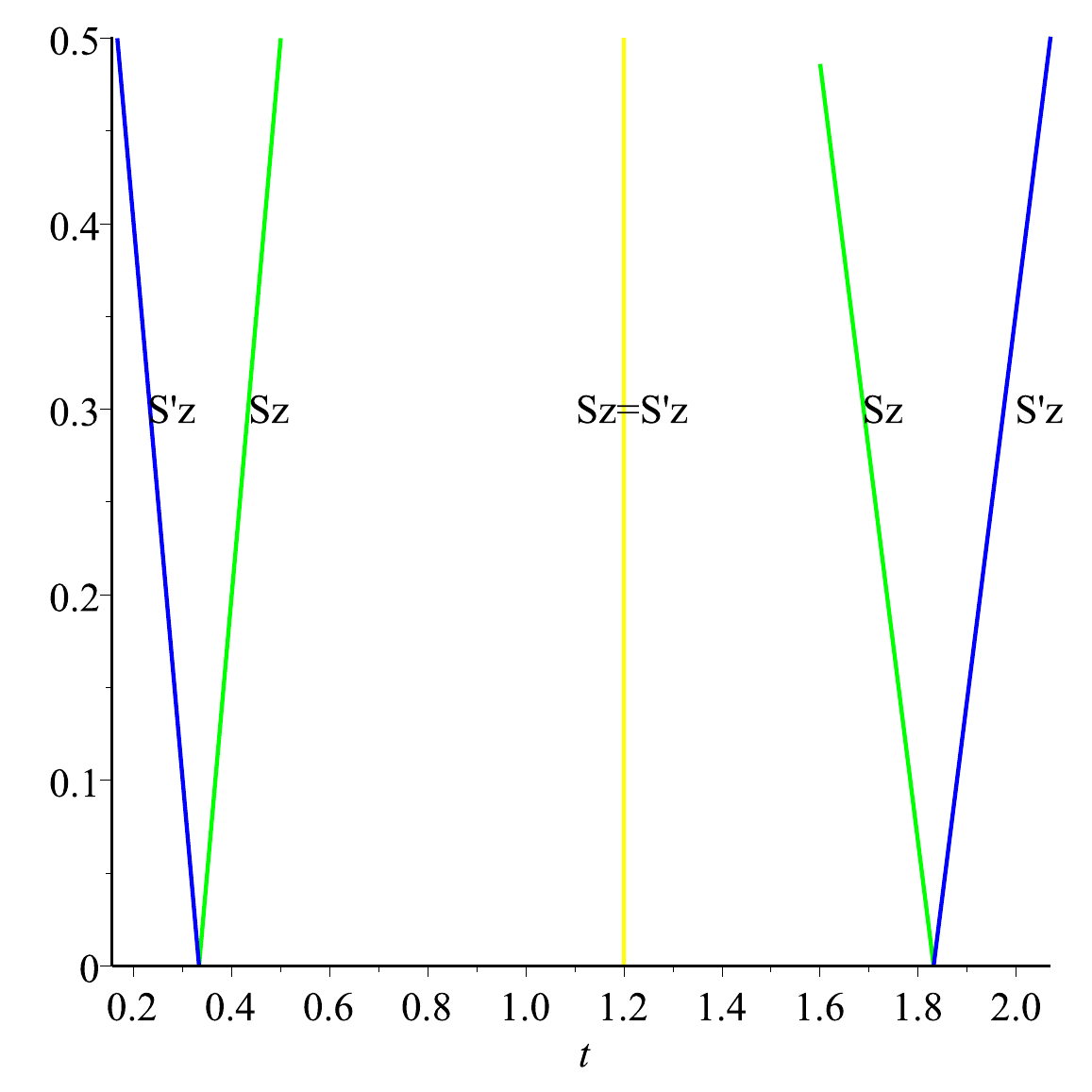}
		\caption[]{Typical relative positions of $Sz$(green rays) and $S'z$(blue rays) for three different values of $z<0$: $z<z_{_{crit_{_{2}}}}$, $z=z_{_{crit_{_{2}}}}$, $z_{_{crit_{_{2}}}}<z<0$, yellow ray represents $Sz=S'z$. Here $b_{_{1}}=8$, $c=1$} \label{fig:01}
	\end{center}
\end{figure}


Letting\ \ $z$\ \ vary, in the half space\ \ $(z<0,t, Y>0)$\ \ there are four regions: one bounded by\ \ $Son$\ \  and\ \  $Son'$\ \ contained in\ \ $z<z_{_{crit_{_{2}}}}$; one bounded by\ \ $Son$\ \ and\ \ $Son'$\ \ but contained in\ \ $z_{_{crit_{_{2}}}}<z<0$. The 3-dimensional regions generated by\ \ $tSz$\ \ and\ \ $tS'z$\ \ in\ \ $z<z_{_{crit_{_{2}}}}$\ \ connect with its corresponding region in\ \  $z_{_{crit_{_{2}}}}<z<0$\ \ generating two regions bounded by\ \ $Son$,\ \ $Son'$\ \ and\ \ $\mathcal{C}$\ \ contained in the half space\ \ $(z<0, t, Y>0)$. Symmetrically, there are four regions in the half space\ \ $(z>0,t,Y>0)$, see Figure \ref{fig:02} which illustrates how\ \ $\mathcal{C}$ (plane),\ \ $Son$ (green surface) and\ \ $Son'$ (blue surface) decompose the wave manifold\ \ $\mathcal{W}$. To draw Figure \ref{fig:02}, we take\ \ $b_{_{1}}=8$\ \ and\ \ $c=1$.

Let us describe how the three regions in\ \ $z_{_{crit_{_{2}}}}<z<0$\ \ connect with the three ones in\ \ $0<z<z_{_{crit_{_{1}}}}$. As\ \ $z \to 0^+$, the intersection point of\ \ $Sz$\ \ and\ \ $S'z$\ \ goes to\ \ $-\infty$,\ \ $Sz$\ \ and\ \  $S'z$\ \ become the lines\ \ $Y=2c$\ \ and\ \ $Y=-2c$, respectively. So, the region under\ \ $S'z$\   is pushed to\ \ $-\infty$\ \ and does not connect to the region on\ \ $z<0$, generating two regions: one above\ \ $Sz$\ \ and one limited by\ \ $Sz$\ \ and the\ \ $t-$axis. In this way we have six regions in the half space\ \ $(z,t,Y>0)$: two regions for\ \ $z_{_{crit_{_{2}}}}<z< z_{_{crit_{_{1}}}}$; two for\ \ $z<z_{_{crit_{_{2}}}}$\ \ and two for\ \ $z>z_{_{crit_{_{1}}}}$. Symmetrically there are six more regions in the half-space\ \ $(z,t,Y<0)$.
\end{proof}

\vspace{.5cm}


We will refer to Figure \ref{fig:02} in the description of the twelve regions in\ \ $\mathcal{W}$. \emph{Sonic} and $sonic'$ surfaces intersect along the inflection locus (the hyperbola-like curve in the characteristic surface) and also along two vertical lines transversal to the characteristic. We call\ \  $SS'$\ \ the  3-dimensional regions contained in\ \  $z^{^{2}}> 1/(b_{_{1}}+1)$\ \ and bounded by\ \ $Son$,\ \ $Son'$\ \ and\ \ $IL$. There are four such regions, which we can specify by the signs of\ \ $z$\ \ and\ \ $Y$. We call \emph{lateral regions}, the regions bounded by\ \ $Son$,\ \ $\mathcal{C}$\ \ and\ \ $Son'$\ \ contained, respectively, in\ \ $z>0$\ \ and in\ \ $z<0$. There are four such regions, again we can specify them by the signs of\ \ $z$\ \ and\ \ $Y$. Let us describe the four regions in\ \ $z^{^{2}}< 1/(b_{_{1}}+1)$. In subspace\ \ $Y>0$,\ \ the sonic surface is connected, we call it \emph{bridge}, and the sonic$'$ surface is not connected. We have two regions: the \emph{over bridge} region bounded by\ \ $Son$\ \ and\ \ $Son'$\ \ and the \emph{under bridge} region bounded by\ \ $\mathcal{C}$,\ \ $Son$\ \ and\ \ $Son'$. In subspace\ \ $Y<0$,\ \ the $sonic'$ surface is connected, we call it \emph{tunnel}, and the sonic surface is not connected. We have two regions: one bounded by\ \ $Son$,\ \ $Son'$\ \ and\ \ $\mathcal{C}$,\ \ called \emph{over tunnel} and other bounded by\ \ $Son$\ \ and\ \ $Son'$,\ \ called \emph{under tunnel}.
\begin{figure}[htpb]
\begin{center}
\includegraphics[scale=0.8,width=0.5\linewidth]{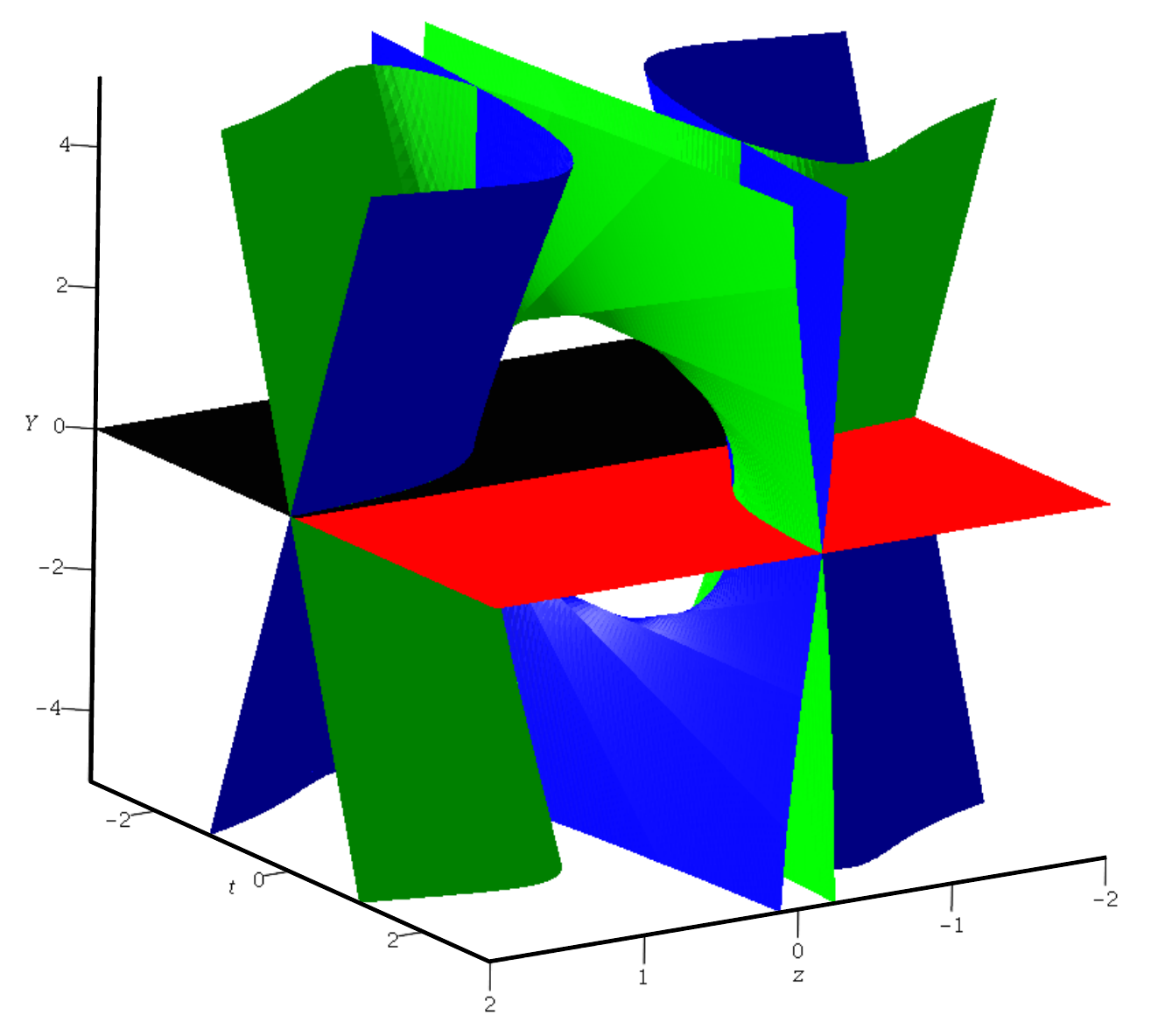}
\caption[]{ Decomposition of the wave manifold $\mathcal{W}$ in twelve regions defined by $\mathcal{C}$( plane), $Son$( green surface) and $Son'$( blue surface). Here $b_{_{1}}=8$ and $c=1$}\label{fig:02}
\end{center}
\end{figure}


We recall that the coincidence curve on\ \ $\mathcal{C}$\ \ is the common boundary of\ \ $\mathcal{C}_{_{f}}$\ \ and\ \ $\mathcal{C}_{_{s}}$. Hugoniot curves through points of the coincidence curve are tangent to\ \ $\mathcal{C}$\ \ and\ \ $Son'$, see \cite{Eschenazi02}. { In order to prove this fact in these new coordinates we introduce:}
	
\begin{definition}
The two-dimensional sub-manifold of\ \ $\mathcal{W}$\ \ generated by Hugoniot curves through points of the coincidence curve is called \emph{saturated of the coincidence curve} by Hugoniot curves, denoted by\ \  $SCC$.
\end{definition}
 
\begin{lemma}
The surface\ \ $SCC$\ \ is tangent to\ \ $\mathcal{C}$\ \ and to\ \ $Son'$.
\label{lemma:Tf}
\end{lemma}
\begin{proof}
A Hugoniot curve through a point of the coincidence curve is given by\ \ $2\widetilde{U}+b_{_{1}}zY={2(b_{_{1}}u_{_{0}}+a_{_{1}}-a_{_{4}})}$,\ $2(V-c)+Y=2{(v_{_{0}}+a_{_{2}})}$, where\ \ $(u_{_{0}},v_{_{0}})$\ \ is a point of the ellipse\ 
\begin{equation}
{\mathcal{E}=(b_{_{1}}u_{_{0}}+a_{_{1}}-a_{_{4}})^{^{2}}+4(v_{_{0}}+a_{_{3}})(v_{_{0}}+a_{_{2}})=0}.
\label{eq:ellipse}
\end{equation}

Using equations \eqref{eq:tdef} in the expressions of\ \ $u_{_{0}}$\ \ and\ \ $v_{_{0}}$\ \ and substituting into the ellipse equation \ref{eq:ellipse}, we get that the surface\ \ $SCC$\ \ is given by:
\begin{equation}
\frac{tf_{_{1}}Y^{^{2}}+tf_{_{2}}Y+tf_{_{3}}}{z^{^{2}}+1}=0,
\label{eq:tfequ}
\end{equation}
where
\begin{eqnarray*}
tf_{_{1}}&=&(z^{^{2}}+1)(b_{_{1}}^{^{2}}z^{^{2}}+4),\\
tf_{_{2}}&=&4cz(z^{^{2}}+1)(b_{_{1}}z^{^{2}}-b_{_{1}}+4)t+2[(b_{_{1}}-1)z^{^{2}}+1],\\
tf_{_{3}}&=&4c^{^{2}}t^{^{2}}(z^{^{2}}+1)^{^{3}}.
\end{eqnarray*}
The intersection of\ \ $SCC$ \ \ and \ \ $\mathcal{C}$\ \ is obtained putting\ \ $Y=0$\ \ in equation \eqref{eq:tfequ}. Doing so, we get\ \ $4c^{^{2}}t^{^{2}}(z^{^{2}}+1)^{^{2}}=0$. It follows that\ \ $SCC$\ \ is tangent to\ \ $\mathcal{C}$\ \ along the straight line\ \ $t=0$,\ \ the coincidence curve.

Hugoniot curves are tangent to\ \ $Son'$\ \ along a curve called {\emph{{hysteresis' curve}}}, somewhat improperly referred as sonic fold in \cite{Eschenazi02}.
To show that\ \ $SCC$\ \ surface is tangent to\ \ $Son'$\ \ we solve the equation of \eqref{eq:son'} in\ \ $Y$, obtaining
\begin{equation}
 Y= \frac{2c[(b_{_{1}}+1)tz^{^{5}}+(b_{_{1}}+4)tz^{^{3}}+(b_{_{1}}-1)z^{^{2}}+3tz+1]}{(b_{_{1}}+1)z^{^{4}}+b_{_{1}}z^{^{2}}-1}.
\label{eq:sonlizt}
\end{equation}

By Substituting\ \ $Y$\ \ from equation \eqref{eq:sonlizt} into equation \eqref{eq:tfequ} we obtain
\begin{equation}
4c^{^{2}} \left\{
 \frac{(1+z^{^{2}})[(b_{_{1}}+1)^{^{2}}z^{^{4}}+2(b_{_{1}}+3)z^{^{2}}+1]t+(2+b_{_{1}})z[(b_{_{1}}-1)
 	\label{eq:Tfzt}z^{^{2}}+1]}{(1+z^{^{2}})[(b_{_{1}}+1)z^{^{2}}-1]}
\right\}^{^{2}}=0.
\end{equation} 

So\ \ $SCC$\ \ is tangent to\ \ $Son'$.
\end{proof} 

\begin{remark}
The surface\ \ $SCC$\ \  is topologically a cylinder, each section\ \ $z=constant$\ \ is an ellipse. It is easy to see that all ellipses are tangent to the\ \ $t-axis$\ \ at\ \  $(0,0)$  and contained in the\ \  $Y<0$\ \ half plane. In this way\ \ $SCC$\ \ is contained in the\ \ $Y<0$\ \ half space, in fact, in the above tunnel region. The interior of\ \ $SCC$\ \ is the saturated of the elliptic region by Hugoniot curves and its exterior is the saturated of the hyperbolic region by Hugoniot curves.
\end{remark}

 We emphasize that equation \eqref{eq:Tfzt} is the projection of\ \ $Son'\cap { SCC}$,\ \ \emph{hysteresis$'$ curve} on the plane\ \ $(z,t)$.
 
Parametric equations for {\emph{{hysteresis' curve}}} curve are obtained solving  equation \eqref{eq:Tfzt} in\ \ $t$,\ \ and substituting\ \ $t(z)$\ \ { into} equation \eqref{eq:sonlizt}. Straightforward calculations give the parametric equations as,
\begin{equation}
\left \{
\begin{array}{l}
z=z, \\
t=-\dfrac{(b_{_{1}}+2)z[(b_{_{1}}-1)z^{^{2}}+1)]}{(z^{^{2}}+1)[(b_{_{1}}+1)^{^{2}}z^{^{4}}+2(b_{_{1}}+3)z^{^{2}}+1]},\\ 
Y=-\dfrac{2c[(b_{_{1}}-1)z^{^{2}}+1]}{(b_{_{1}}+1)^{^{2}}z^{^{4}}+2(b_{_{1}}+3)z^{^{2}}+1}.
\end{array}
\right.
\label{eq:parsonli}
\end{equation}

\section{Finding $Son'_{s}$ and $Son'_{_{f}}$ \label{sec:sonlifs}}

In this section, we describe how\ \ $Son'$\ \ splits into \emph{slow $sonic'$ surface},\ \ $Son'_{s}$,\ \ and \emph{fast $sonic'$ surface},\ \ $Son'_{_{f}}$. In order to do so, we will take a point in\ \  $Son'$,\ \ the Hugoniot curve through this point, find the intersection with\ \ $\mathcal{C}$,\ \ calculate the shock speed in each of these 2 points and see which point of\ \ $\mathcal{C}$\ \ has the same speed as our initial point in\ \ $Son'$.

Given a point\ \ $(t'_{_{0}},z_{_{0}},Y_{_{0}})$  on\ \ $Son'$,\ \ we obtain from equation \eqref{eq:son'} that

$$t'_{_{0}}= \frac{Y_{_{0}}(z_{_{0}}^{^{2}}+1)[(b_{_{1}}+1)z_{_{0}}^{^{2}}-1]-2c[(b_{_{1}}-1)z_{_{0}}^{^{2}}+1]}{2cz_{_{0}}(z_{_{0}}^{^{2}}+1)[(b_{_{1}}+1)z_{_{0}}^{^{2}}+3]}.$$

\noindent The shock speed at this point,\ \  $s_{_{son'}}(t_{_{0}}',z_{_{0}},Y_{_{0}})$,\ \ is obtained putting\ \  $t_{_{0}}=t_{_{0}}'$\ \ into equation \eqref{eq:speed},

\begin{equation}
s_{_{son'}}(t_{_{0}}',z_{_{0}},Y_{_{0}})=\frac{[(b_{_{1}}+1)z_{_{0}}^{^{2}}-1]^{^{2}}Y_{_{0}}+6c(b_{_{1}}+1)z_{_{0}}^{^{2}}+2c}{2b_{_{1}}z_{_{0}}[(b_{_{1}}+1)z_{_{0}}^{^{2}}+3]}.
\label{eq:speedson'}
\end{equation}



%

Parametric equations for the Hugoniot curve through\ \  $(t_{_{0}}',z_{_{0}},Y_{_{0}})$\ \ are obtained changing \ \ $t_{_{0}}$\ \ into\ \ $t_{_{0}}'$\ \ in equations \eqref{eq:hugponto}, giving\ \ $(t_{_{hu}}(z),z,Y_{_{hu}}(z))$. Solving equation\ \ $Y_{_{hu}}=0$,\ \ we obtain the\ \ $z$\ \ coordinates of  the intersection points of Hugoniot curve with\ \ $\mathcal{C}$,\
$$z_{_{\mathcal{C}_{_{1}}}}= \frac{(b_{_{1}}+1)z_{_{0}}^{^{2}}+1}{2z_{_{0}}}\ \ \ \text{and}\  \ \ z_{_{\mathcal{C}_{2}}}= -\frac{2(Y_{_{0}}-c)z_{_{0}}}{(b_{_{1}}+1)Y_{_{0}}z_{_{0}}^{^{2}}+Y_{_{0}}+2c}.$$
Substituting\ \ $z_{_{\mathcal{C}_{_{1}}}}$\ \ and\ \ $z_{_{\mathcal{C}_{2}}}$\ \ in the expression of\ \ $t_{_{hu}}(z)$ we obtain, respectively, the\ \ $t$\ \ coordinates of the intersection points, 
$$t_{_{\mathcal{C}_{_{1}}}}=\frac{2z_{_{0}}\left\{[(b_{_{1}}+1)^{^{2}}z_{_{0}}^{^{4}}+2(b_{_{1}}+3)z_{_{0}}^{^{2}}+1]Y_{_{0}}+2c[(b_{_{1}}-1)z_{_{0}}^{^{2}}+1]\right\}}{c[(b_{_{1}}+1)z_{_{0}}^{^{2}}+3][(b_{_{1}}+1)^{^{2}}z_{_{0}}^{^{4}}+2(b_{_{1}}+3)z_{_{0}}^{^{2}}+1]}$$
and
$$t_{_{\mathcal{C}_{2}}}=-\frac{(A_{_{t_{_{\mathcal{C}_{2}}}}}Y_{_{0}}+B_{_{t_{_{\mathcal{C}_{2}}}}})([(b_{_{1}}+1)z_{_{0}}^{^{2}}+1]Y_{_{0}}+2c)^{^{2}}}{2cz_{_{0}}[(b_{_{1}}+1)z_{_{0}}^{^{2}}+3](C_{_{t_{_{\mathcal{C}_{2}}}}}Y_{_{0}}^{^{2}}+D_{_{t_{_{\mathcal{C}_{2}}}}}Y_{_{0}}+E_{_{t_{_{\mathcal{C}_{2}}}}})},$$
where
\begin{eqnarray*}
A_{_{t_{_{\mathcal{C}_{2}}}}}&=&(b_{_{1}}+1)^{^{2}}z_{_{0}}^{^{4}}+2(b_{_{1}}+3)z_{_{0}}^{^{2}}+1,\\
B_{_{t_{_{\mathcal{C}_{2}}}}}&=&2c[(b_{_{1}}-1)z_{_{0}}^{^{2}}+1)],\\
C_{_{t_{_{\mathcal{C}_{2}}}}}&=&[(b_{_{1}}+1)^{^{2}}z_{_{0}}^{^{4}}+2(b_{_{1}}+3)z_{_{0}}^{^{2}}+1],\\
D_{_{t_{_{\mathcal{C}_{2}}}}}&=&4c[(b_{_{1}}-1)z_{_{0}}^{^{2}}+1],\\
E_{_{t_{_{\mathcal{C}_{2}}}}}&=&4c^{^{2}}(z_{_{0}}^{^{2}}+1).
\end{eqnarray*}
\noindent
So, the intersection points of Hugoniot curve with\ \ $\mathcal{C}$\ \ are\ \ $\mathcal{C}_{_{1}}= (t_{_{\mathcal{C}_{_{1}}}}, z_{_{\mathcal{C}_{_{1}}}},0)$\ \ and\ \ $\mathcal{C}_{_{2}}= (t_{_{\mathcal{C}_{2}}}, z_{_{\mathcal{C}_{2}}},0)$. 

Substituting\ \ $\mathcal{C}_{_{1}}$\ \ and\ \ $\mathcal{C}_{_{2}}$\ \ into equation \eqref{eq:speed} we obtain, respectively
$$s_{_{\mathcal{C}_{_{1}}}}=\dfrac{\left\{(b_{_{1}}+1)^{^{3}}z_{_{0}}^{^{4}}+2[(b_{_{1}}+1)^{^{2}}-2]z_{_{0}}^{^{2}}+b_{_{1}}+1\right\}Y_{_{0}}+2c[(b_{_{1}}+1)^{^{2}}z_{_{0}}^{^{2}}+2z_{_{0}}^{^{2}}+b_{_{1}}+1]}{2z_{_{0}}b_{_{1}}[(b_{_{1}}+1)z_{_{0}}^{^{2}}+3]}$$
and
$$s_{_{\mathcal{C}_{2}}}=\dfrac{[(b_{_{1}}+1)z_{_{0}}^{^{2}}-1]^{^{2}}Y_{_{0}}+6c(b_{_{1}}+1)z_{_{0}}^{^{2}}+2c}{2b_{_{1}}z_{_{0}}[(b_{_{1}}+1)z_{_{0}}^{^{2}}+3]}.$$
A simple inspection shows that\ \ $s_{_{\mathcal{C}_{2}}}=s_{_{son'}}(t_{_{0}}',z_{_{0}},Y_{_{0}})$. 

According to Lemma \ref{lemma:cf-cs}, we must to study the sign of\ \  $t_{_{\mathcal{C}_{2}}}$. If\ \ $t_{_{\mathcal{C}_{2}}}>0$, the point\ \ $(t_{_{0}}',z_{_{0}},Y_{_{0}})$\ \ is in\ \ $Son'_{_{f}}$,\ \ otherwise it is in\ \ $Son'_{s}$.

\begin{proposition}  
The hysteresis$'$ curve and the straight line\ \ $(t,z=0,Y=-2c)$\ \ are the boundaries of\ \ $Son'_{s}$\ \ and\ \ $Son'_{_{f}}$.
\label{proposition:son'sson'f}
\end{proposition}

\begin{proof} We have to study the sign of\ \ $t_{_{\mathcal{C}_{2}}}$. The denominator is the product 
of\ \ $2cz_{_{0}}[(b_{_{1}}+1)z_{_{0}}^{^{2}}+3]$\ \ and a polynomial of degree 2 in\ \ $Y_{_{0}}$\ \ with negative discriminant. Since the coefficient of\ \ $Y_{_{0}}^{^{2}}$\ \ is positive the denominator has the same sign as\ \ $z_{_{0}}$. The numerator is a term of the form\ \ $A(z_{_{0}})Y_{_{0}}+B(z_{_{0}})$,\ \ $A$ \ \ and\ \ $B$\ \ positives, multiplied by a positive term, everything preceded by a minus sign.

It follows that the curves\ \ $ z = 0 $\ \ and\ \ $ AY + B = 0,$\ \ in\ \ $Son'$, are the boundaries of\ \ $ Son'_{s} $\ \ and\ \ $ Son'_{_{f}}$. If\ \ $Y>-B/A$\ \ and\ \ $z>0$\ \ the point is in\ \ $Son'_{s}$, switching according to sign change.
 
 The curve\ \ $AY+B=0$\ \ is the projection of hysteresis$'$ curve on  $zY$ plane (the third equation in \eqref{eq:parsonli}). In\ \ $Son'$, $z=0$\ \ is the straight line\ \ $(t,z=0,Y=-2c)$. It is the intersection of\ \ $Son'$\ \ with the surface\ \ $SSB$, the saturated of the secondary bifurcation by Hugoniot curves, (see Appendix 1). The straight line and the hysteresis$'$ curve intersect transversely forming a saddle point in\ \ $Son'$.
\end{proof}


\section{Lax's Shocks in $\mathcal{W}$\label{sec:lax}}


In this section, we introduce Admissibility condition for Hugoniot arcs in\ \ $\mathcal{W}$. By arc we mean the image of a closed interval. There will be two types of conditions, which are called simply type 1 and type 2.

We begin by introducing two simplifying assumptions, \cite{AEMP10}: 

\noindent
Giving a point\ \ $\mathcal{U} \in \mathcal{W}$, we use\ \ $sh(\mathcal{U})$\ \ ($sh'(\mathcal{U})$)\ \ to denote the Hugoniot (Hugoniot$'$) curve through\ \ $\mathcal{U}$. We define\ \ $\mathcal{U}_{_{s}}$,\ \ $\mathcal{U}_{_{f}}$,\ \ $\mathcal{U'}_{_{s}}$\ \ and\ \  $\mathcal{U}_{_{f}}'$\ \ to be the points
\begin{equation}
	\mathcal{U}_{_{s}}=sh (\mathcal{U}) \cap \mathcal{C}_{_{s}}, \hspace{.2cm}
	\mathcal{U}_{_{f}}=sh (\mathcal{U}) \cap \mathcal{C}_{_{f}},\hspace{.2cm}
	\mathcal{U}_{_{s}}^{\prime} = sh'(\mathcal{U}) \cap \mathcal{C}_{_{s}}, \hspace{.2cm} \text { and } \hspace{.2cm}
	\mathcal{U}_{_{f}}^{\prime} = sh'(\mathcal{U}) \cap \mathcal{C}_{_{f}}.
	\label{usuli}
\end{equation}

These points are fundamental to obtain the characteristic speed, $\lambda_{_{s}}$ (or $\lambda_{_{f}}$),\ \ associated to $\mathcal{U}$, see Fig. \ref{figureshock}.


\begin{assump}
	\label{assu:assumption1}
	We will deal only with points\ \ $\mathcal{U}$\ \ on\ \ $\mathcal{W}$\ \ such that the points
	$\mathcal{U}_{_{s}}$\ \ and\ \ $\mathcal{U}_{_{f}}$\ \ always exist, \emph{i.e.}, in the hyperbolic region. 
\end{assump}
\begin{assump}
	\label{assu:bifsec}
	We consider only Hugoniot and Hugoniot$'$ curves that are 
	diffeomorphic to $\mathbb{R}$, \emph{i.e.}, they do not 
	contain points in the secondary bifurcation locus.
\end{assump}

\noindent
{{Let us consider a fixed\ \ $\mathcal{U}\in\mathcal{W}$\ \ and 
we use\ \ $(\ref{usuli})$\ \ to define\ \ $\mathcal{U}'_{_{s}}$,\ $\mathcal{U}'_{_{f}}$,\medskip 


\noindent
\textbf{1.-} The Lax's shock of type 1: 
\begin{equation}
s(\tilde{\mathcal{U}}) < s({\mathcal{U}}_{_{s}}) = \lambda_{_{s}}(U)\ \ \ 
\lambda_{_{s}}(\tilde U') = s(\tilde{\mathcal{U}}_{_{s}}') < s(\tilde{\mathcal{U}}) < 
=\lambda_{_{f}} (\tilde U').\label{lax1}
\end{equation}
\noindent
\textbf{2.-} The Lax's shock of type 2:
\begin{equation}
s(\tilde{\mathcal{U}}) > s(\mathcal{U}_{_{f}}) = \lambda_{_{f}}(U)\ \ \ 
\lambda_{_{s}}(U') = s({\mathcal{U}}_{_{s}}') < s(\tilde{\mathcal{U}}) < s({U}_{_{f}}')
=\lambda_{_{f}} (U').\label{lax2}
\end{equation}

{Following the physics, Lax defines one- and two- shocks. Any of these shocks is a pair of states that satisfies the Rankine-Hugoniot equation and certain speed inequalities. Therefore, the definition \emph{shock}, \emph{1-shock}, \emph{2-shock} applies to points in the wave manifold, not to curves in the wave manifold. For the purpose of the topological construction of Riemann solutions, which is the object of this paper, it is may be convenient to use this nomenclature for certain curves of shock points in the wave manifold consisting of either 1-shocks or 2-shocks. It is also convenient to provide an orientation to these curves. So we use the nomenclature \emph{forward shock curve} to a curve oriented with decreasing speed consisting of 1-shocks. Similarly, \emph{backward shock curve} to a curve oriented with increasing speed consisting of 2-shocks. Other combination of orientations or Lax type are not necessary in this paper and will not be used. However, if any of these shocks are needed, they will receive their full name, for instance, \emph{1-backward shock}; or \emph{2-forward shock}.}



A Hugoniot arc curve is a \emph{ forward shock curve} if is oriented with decreasing speed and  satisfies the following conditions, \cite{AEMP10} : 

\begin{figure}[htpb]
	\begin{center}
		\includegraphics[scale=1.0,width=1.0\linewidth]{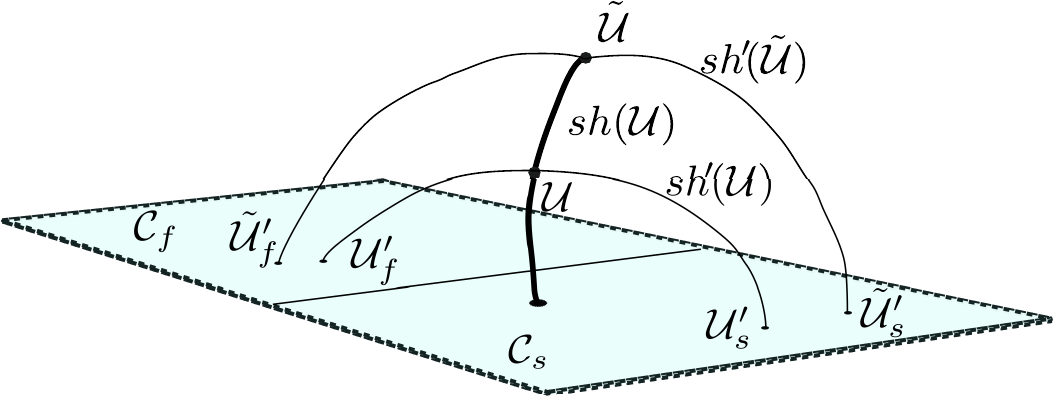}
		\caption[]{\label{figureshock}Consider $\mathcal{U}\in\mathcal{W}$. From $\mathcal{U}$, we take a $\tilde{\mathcal{U}}\in sh(\mathcal{U})$. We obtain the projections of $\mathcal{U}$  and $\tilde{\mathcal{U}}$ in $\mathcal{C}_{_{s}}$ and $\mathcal{C}_{_{f}}$ using Hugoniot$'$ curves, denoted as $\mathcal{U}'_{_{s}}$, $\tilde{\mathcal{U}}'_{_{s}}$, $\mathcal{U}'_{_{f}}$ and  $\tilde{\mathcal{U}}'_{_{f}}$. } 
	\end{center}
\end{figure}

\newpage


\noindent
\textbf{L1.1-} The speed\ \ $s$\ \ at any point\ \ $\mathcal{U}$\ \ in the arc satisfies\ \ $s(\mathcal{U})\leq s( \mathcal{U}_{_{s}})$,\medskip


\noindent
\textbf{L1.2-} The speed\ \ $s$\ \ at any point\ \ $\mathcal{U}$\ \ in the arc satisfies\ \ $s(\mathcal{U}_{_{s}}^{\prime})\leq s(\mathcal{U})\leq s(\mathcal{U}_{_{f}}')$.\\




It follows that:

\noindent\textbf{1.-} Since a forward shock curve is oriented with decreasing speed, it stops if it reaches $Son$, otherwise it goes to infinity(as $s$ decreases and $z$ goes to plus or minus infinity)


\noindent\textbf{2.-} From condition\ \ \textbf{L1.1}\ \  a  { forward shock curve} begins at\ \ $C_{_{s}}$\ \ or\ \ $Son'_{_{s}}$.\\



Forward shock curves starting at\ \ $\mathcal{C}_{_{s}}$\ \ are called \emph{ local} and forward shock curves starting at\ \ $Son'_{_{s}}$\ \ are called \emph{nonlocal}.

By continuity of\ \ $s$, given a Hugoniot arc with decreasing\ \ $s$, starting at\ \   $\mathcal{C}_{_{s}}$\ \ or in\ \  $Son'_{_{s}}$\ \ and ending in\ \ $Son$\ \ or at infinity, it is sufficient to check conditions \textbf{L1.1} or \textbf{L1.2} at the initial point, to verify whether it is a forward shock curve or not.

Since both conditions are trivially satisfied in\ \ $\mathcal{C}_{_{s}}$\ \ and \textbf{L1.1} is also trivially in\ \ $Son'_{_{s}}$, it is sufficient to verify \textbf{L1.2} at\ \ $Son'_{_{s}}$\ \ for the arc to be a forward shock curve.\\

In the same way, a Hugoniot arc curve is said a { (a backward shock curve)} if is oriented with increasing speed and satisfies the following conditions, \cite{AEMP10}:\medskip

\noindent

\noindent
\textbf{L2.1-} The speed\ \ $s$\ \ at any point\ \ $\mathcal{U}$\ \ in the arc satisfies\ \ $s(\mathcal{U})\geq s( \mathcal{U}_{_{f}})$;\medskip

\noindent
\textbf{L2.2-} The speed\ \ $s$\ \ at any point\ \ $\mathcal{U}$\ \ in the arc satisfies\ \ $s(\mathcal{U}_{_{s}}^{\prime})\leq s(\mathcal{U})\leq s(\mathcal{U}_{_{f}}')$.\\




It follows that:\\

\noindent\textbf{1.-}  Since a backward shock curve is oriented with increasing speed, it stops if it reaches\ \ $Son$, otherwise it goes to infinity (as\ \ $s$\ \  increases and\ \ $z$\ \ goes to plus or minus infinity.) 

\noindent\textbf{2.-} From condition\ \ \textbf{L2.1}\ \  a { backward shock curve} begins at\ \ $\mathcal{C}_{_{f}}$\ \ or\ \ $Son'_{_{f}}$. 
Backward shock curves starting at\ \ $\mathcal{C}_{_{f}}$\ \ are called \emph{local} and backward shock curves starting at\ \ $Son'_{_{f}}$\ \ are called \emph{nonlocal}.\\

By continuity of\ \ $s$, giving a Hugoniot arc with increasing\ \ $s$, starting at\ \ $\mathcal{C}_{_{f}}$\ \  or in\ \  $Son'_{_{f}}$\ \ and ending in\ \ $Son$\ \ or at infinity, it is sufficient to check conditions \textbf{L2.1} and \textbf{L2.2} at the initial point to verify whether it is a backward sock curve or not. Since both conditions are trivially satisfied in\ \  $\mathcal{C}_{_{f}}$, and \textbf{L2.1} is also trivially satisfied in\ \ $Son'_{_{s}}$, it is sufficient to verify \textbf{L2.2} at \ \ $Son'_{_{s}}$\ \ for the arc to be a backward shock curve.

\begin{remark}
	Since conditions \textbf{L1.2} and \textbf{L2.2} are the same, we refer to them simply as \textbf{L2}.
\end{remark}

\section{Rarefactions and Composites\label{sec:rare-comp}}

\subsection{Rarefactions\label{subsec:raref}}

Going back to equation\ \ $W_{_{t}}+F(W)_{_{x}}=0$, we will consider solutions of the form\ \ $\widetilde{W}(x/t)$,\ \ which will be called \emph{rarefaction waves}.  So, our equation becomes 
$$ (-x/t^{^{2}})\widetilde{W}’+(1/t)DF(\widetilde{W}).\widetilde{W}’=0, $$
where\ \ $\widetilde{W}= \widetilde{W}(\lambda)$,\ $\lambda=x/t$\ \ and\ \ $'$\ \ indicates differentiation with respect to\ \ $\lambda$.

We see that the curves parametrically defined in the\ \ $(u,v)$-plane by the rarefaction waves are the integral curves of the line fields defined by the eigenvectors of\ \ $DF(u,v)$. We call these curves \emph{rarefaction curves}.

For\ \ $F$\ \ given by (\ref{fgeq}) we have 
$$DF=\begin{bmatrix}
(b_{_{1}}+1)u+a_{_{1}}& v+a_{_{2}}\\
v+a_{_{3}}&u+a_{_{4}}\\
\end{bmatrix}$$

Let\ \ $r=(r_{_{1}},r_{_{2}})$\ \ be the eigenvector associated to the eigenvalue\ \ $\lambda$. Eliminating\ \ $\lambda$\ \ in the system
\begin{equation*}
\left\{\begin{array}{rl}
[(b_{_{1}}+1)u+a_{_{1}})]r_{_{1}}+(v+a_{_{2}})r_{_{2}}  &= \lambda r_{_{1}}\\
(a_{_{3}}+v)r_{_{1}}        + (a_{_{4}}+u)r_{_{2}} &= \lambda r_{_{2}}.
\end{array} \right.
\end{equation*}
and replacing\ \ $r_{_{1}}/r_{_{2}}$\ \  by\ \  $du/dv$  we get the differential equation: 

$$(v+a_{_{3}})\left(\frac{du}{dv}\right)^{^{2}} - (b_{_{1}}u+a_{_{1}}-a_{_{4}})\frac{du}{dv} - (v+a_{_{2}}) =0.$$ 

Let\ \ $\Delta$\ \ be the discriminant of this 2nd degree equation in\ \ $du/dv$,
$$\Delta= (b_{_{1}}u+a_{_{1}}-a_{_{4}})^{^{2}} +4(v+a_{_{2}})(v+a_{_{3}}).$$

When\ \ $\Delta=0$, the equation of the \emph{coincidence curve} defined in Subsection \ref{sec:Yzt-subsection}, it is an ellipse. As stated before, its interior is called elliptic region, and there are no rarefaction curves there. The exterior of ellipse is called hyperbolic region  and here, through every point pass two rarefaction curves transversal to each other.

To study this differential equation, we introduce a new variable\ \ $z$, and replace the differential equation by the pair of equations
\begin{equation*}
\left\{\begin{array}{rl}
G_{_{2}}=(v+a_{_{3}})z^{^{2}} -(b_{_{1}}u+a_{_{1}}-a_{_{4}})z -(v+a_{_{2}}) &=0\\
zdv-du&=0.
\end{array} \right.
\end{equation*}
which we see as a differential equation in a surface in\ \ $(u,v,z)$-space. Since\ \ $z$\ \ is a direction, $(u,v,z)$-space is actually\ \ $\mathbb{R}^{^{2}} \times \mathbb{R}P^{^{1}}$.

The integral curves of\ \ $zdv-du=0$\ \ in this surface are also called \emph{rarefaction curves}. Their projections onto the\ \ $(u,v)$-plane by\ \ $(u,v,z)\rightarrow (u,v)$\ \ are the rarefaction curves defined above.

By introducing new variables:\ \ $\widetilde{U}=b_{_{1}}u +a_{_{1}}-a_{_{4}}$\ \ and\ \ $V_{_{1}}= v+a_{_{2}}$, the expression of\ \ $G_{_{2}}$\ \ becomes\ \ $G$, \emph{i.e.}, the equation of surface is written as\ \ $G=0$, so, if we consider\ \ $(\widetilde{U},V_{_{1}},z)$-space as naturally embedded in\ \  $(\widetilde{U},V_{_{1}},X,Y,z)$-space defined in Section \ref{sec:wmhc}, we see that the surface introduced here is precisely the characteristic surface defined there, by\ \ $G=0$,\ $Y=0$.

Let us use the second  equation of the system
\begin{equation*}
\left\{\begin{array}{rl}
[(b_{_{1}}+1)u+a_{_{1}})]r_{_{1}}+(v+a_{_{2}})r_{_{2}}  &= \lambda r_{_{1}},\\
(a_{_{3}}+v)r_{_{1}}       + (a_{_{4}}+u)r_{_{2}} &= \lambda r_{_{2}}
\end{array} \right.
\end{equation*}
to define\ \ $\lambda$\ \ as a function of\ \ $u$,\ $v$\ \ and\ \ $z$, using that\ \ 
$z=r_{_{1}}/r_{_{2}}$. If we replace\ \ $u$\ \ and\ \ $v$\ \ by the new coordinates\ \  $\widetilde{U},V_{_{1}}$, we see that\ \  $\lambda$\ \  coincides with\ \ $s$\ \ defined in Section \ref{sec:wmhc} as a function of\ \ $(\widetilde{U},V_{_{1}},Y,z)$, when\ \ $z=0$, and the inflection locus defined there as the intersection of\ \ $Son$\ \ with\ \ $\mathcal{C}$\ \ can also be defined as the set of points where\ \ $\lambda$\ \ is critical along the rarefaction curves.

The curve defined in the characteristic surface by\ \ $G=0$\ \ and\ \ $\partial G/\partial z =0$\ \  is where the tangent plane to the characteristic surface is vertical, and projects onto the coincidence curve, so, it is also called \emph{coincidence curve}. It separates\ \ $\mathcal{C}$\ \ into two componentes, which are\ \ $\mathcal{C}_{_{s}}$\ \ and\ \ $\mathcal{C}_{_{f}}$\ \  defined before.

A complete study of rarefactions can be found in \cite{Palmeira88} and in \cite{Bastos05} and a formal definition of rarefactions in the wave manifold can be found in \cite{Isaacson92} and in \cite{Bastos05}.

Let us obtain the differential equation of the rarefaction curves in\ \ $t$\ \ and\ \ $z$\ \ coordinates. Replacing\ \ $u$\ \ and\ \ $v$\ \ by\ \ $\widetilde{U}$\ \ and\ \  $V_{_{1}}$, the equation\ \ $zdv-du=0$\ \ becomes\ \ $b_{_{1}}zdV_{_{1}}-d\widetilde{U}=0$. Take the expressions of\ \ $\widetilde{U}$\ \ and\ \  $V_{_{1}}$\ \ in terms of\ \ $t$\ \  and\ \  $z$,\  $\widetilde{U}=2cz/(z^{^{2}}+1) +ct(z^{^{2}}-1)$,\  
 $V_{_{1}}= c/(z^{^{2}}+1)+ctz$\ \ differentiate, and replace\ \ $d\widetilde{U}$\ \ and\ \ $dV_{_{1}}$\ \ by their expressions in\ \ $dt$\ \ and\ \ $dz$. Collecting terms, we arrive at 
\begin{equation}
\frac{dt}{dz}=\frac{-t(b_{_{1}}-2)z^5-2t(b_{_{1}}-2)z^{^{3}}+2(b_{_{1}}-1)z^{^{2}}-t(b_{_{1}}-2)z+2}{(z^{^{2}}+1)^{^{2}}(1+(b_{_{1}}-1)z^{^{2}})}.
\label{eq:edoraref}
\end{equation}

In this context it is important to orient the rarefaction curves according to the growth of\ \ $s$ (or, which is the same\ \ $\lambda$).

For the purpose of this paper we, will define a \emph{forward rarefaction} to be a rarefaction arc in\ \ $\mathcal{C}_{_{s}}$\ \ with increasing\ \ $s$. We will define a \emph{backward rarefaction} to be a rarefaction arc in\ \ $\mathcal{C}_{_{f}}$\ \ with decreasing\ \ $s$.

The expression of\ \ $s$\ \ in terms of\ \ $z$\ \ and\ \ $t$ (as we have seem in Section \ref{sec:wmhc}, $s$\ \ does not depend on\ \ $Y$) is given by equation \eqref{eq:speed}. The variation of\ \ $s$\ \ along a rarefaction curve is given by\ \ $ds/dz=\partial s/\partial z+(\partial s/\partial t) (dt/dz)$, where\ \ $dt/dz$\ \ is the differential equation \eqref{eq:edoraref}.

After simplification, we arrive at
$$\frac{ds}{dz}=\frac{z(z^{^{2}}+1)[(b_{_{1}}+1)z^{^{2}}+3]t+(b_{_{1}}-1)z^{^{2}}
+1}{[(b_{_{1}}-1)z^{^{2}}+1](z^{^{2}}+1)}.$$

Since the denominator is positive, we only need to look at the numerator, which is of the form\ \ $A(z)t+B(z)$. The equation\ \ $A(z)t+B(z)=0$\ \ defines the curve where\ \ $s$\ \ is critical along the rarefaction curves. This is exactly the inflection locus, defined in Section \ref{sec:wmhc} as the intersection of the sonic surface\ \ $Son$\ \ with the characteristic surface\ \ $\mathcal{C}$. Just put\ \ $Y=0$\ \ in the equation of\ \ $Son$.

In\ \ $z$,\ $t$\ \ coordinates, the inflection locus looks like the hyperbola\ \ $t=-1/z$. The derivative\ \ $ds/dz$ is clearly positive between the two branches and negative in the two other regions (``exterior'' to the two branches), so, a rarefaction curve will be oriented in the direction of increasing $z$ between the branches and in the direction of decreasing\ \ $z$\ \ out of the branches. As shown in figure ???? a rarefaction curve arrives at the inflection locus, in\ \ $\mathcal{C}_{_{s}}$ ($t<0$) and leaves the inflection locus in\ \ $\mathcal{C}_{_{f}}$ ($t>0$) { and this remains true even if the rarefaction curve is tangent to the inflection locus. As a consequence, at such a tangency point, the rarefaction curve necessarily crosses the inflection locus. As will be seen in Section \ref{subsec:WaveCurve}, there are two such points. We denote by\ \ $\mathcal {I}_{_{s}}$\ \ and\ \ $\mathcal {I}_{_{f}}$\ \ the branches of the inflection locus in\ \ $\mathcal{C}_{_{s}}$\ \ and\ \ $\mathcal{C}_{_{f}}.$

\subsection{Composites\label{subsec:Compos}}

Given a point\ \ $\mathcal{U}$\ \ in\ \ $Son’$, the Hugoniot curve through\ \  $\mathcal{U}$\ \ generically intersects\ \ $\mathcal{C}$\ \ in two points,  $\mathcal{U}_{_{s}}$\ \ and\ \ $\mathcal{U}_{_{f}}$\ \ such that either\ \ $s(\mathcal{U}_{_{s}})=s(\mathcal{U})$\ \ or\ \ $s(\mathcal{U}_{_{f}})=s(\mathcal{U})$. This allows us to define a projection\ \ $T$\ \ from\ \ $Son’$\ \ to\ \ $\mathcal{C}$\ \ which takes\ \ $\mathcal{U}$\ \ into\ \ $\mathcal{U}_{_{s}}$\ \ or\ \ $\mathcal{U}_{_{f}}$ such  that\ \ $s(T(\mathcal{U}))= s(\mathcal{U})$. The pullback of rarefaction curves from\ \ $\mathcal{C}$\ \ to\ \ $Son’$\ \ are called \emph{composite curves}. They are described in \cite{Eschenazi02}. 

We can obtain the differential equation of the composites curves in coordinates\ \ $z$\ \ and\ \ $Y$. As seen in Section \ref{sec:sonlifs}, given a point\ \ $(z_{_{1}},Y_{_{1}},t(z_{_{1}},Y_{_{1}}))$\ \ in\ \ $Son'$\ \ the map which associates the point in\ \ $\mathcal{C}$\ \ in the same Hugoniot  curve through\ \ $(z_{_{1}},Y_{_{1}},t(z_{_{1}},Y_{_{1}}))$\ \ and with the same value of the speed\ \ $s(z_{_{1}},Y_{_{1}},t(z_{_{1}},Y_{_{1}}))$, is given by\ \
$(z_{_{1}},Y_{_{1}}) \longrightarrow (z_{_{\mathcal{C}_{_{2}}}}, t_{_{\mathcal{C}_{_{2}}}})$, where\ \ $$z_{_{\mathcal{C}_{_{2}}}}=-\displaystyle\frac{2(Y_{_{1}}-c)z_{_{1}}}{2b_{_{1}}z_{_{1}}[(b_{_{1}}+1)z_{_{1}}^{^{2}}+3]}\ \ \ \ \ \text{and}\ \ \ \ \
t_{_{\mathcal{C}_{2}}}=\displaystyle\frac{(A_{_{t_{_{\mathcal{C}_{2}}}}}Y_{_{1}}+B_{_{t_{_{\mathcal{C}_{2}}}}})([(b_{_{1}}+1)z_{_{1}}^{^{2}}+1]Y_{_{1}}+2c)^{^{2}}}{2cz_{_{1}}[(b_{_{1}}+1)z_{_{1}}^{^{2}}+3](C_{_{t_{_{\mathcal{C}_{2}}}}}Y_{_{1}}^{^{2}}+D_{_{t_{_{\mathcal{C}_{2}}}}}Y_{_{1}}+E_{_{t_{_{\mathcal{C}_{2}}}}})},$$
where
\begin{eqnarray*}
	A_{_{t_{_{\mathcal{C}_{2}}}}}&=&(b_{_{1}}+1)z_{_{1}}^{^{4}}+2(b_{_{1}}+3)z_{_{1}}^{^{2}}+1,\\
	B_{_{t_{_{\mathcal{C}_{2}}}}}&=&-2c[(b_{_{1}}-1)z_{_{1}}^{^{2}}+1)],\\
	C_{_{t_{_{\mathcal{C}_{2}}}}}&=&Y_{_{1}}^{^{2}}[b_{_{1}}+1z_{_{1}}^{^{4}}+2(b_{_{1}}+3)z_{_{1}}^{^{2}}+1],\\
	D_{_{t_{_{\mathcal{C}_{2}}}}}&=&4cY_{_{1}}[(b_{_{1}}-1)z_{_{1}}^{^{2}}+1],\\
	E_{_{t_{_{\mathcal{C}_{2}}}}}&=&4cY_{_{1}}[(b_{_{1}}-1)z_{_{1}}^{^{2}}+1]+4c^{^{2}}(z_{_{1}}^{^{2}}+1).
\end{eqnarray*}           
To avoid confusion, we use\ \ $z_{_{1}}$\ \ and\ \  $Y_{_{1}}$\ \ as coordinates in\ \ $Son'$.
To get the composite differential equation, we write equation \eqref{eq:edoraref}  in the form\ \ $m(z,t)dz+n(z,t)dt=0$. Replacing\ \ $m(z,t)$\ \ by\ \ $m(z_{_{\mathcal{C}_{_{2}}}}(z_{_{1}},Y_{_{1}}),t_{_{\mathcal{C}_{_{2}}}}(z_{_{1}},Y_{_{1}})),$  
$n(z,t)$\ \ by\ \ $n(z_{_{\mathcal{C}_{_{2}}}}(z_{_{1}},Y_{_{1}}),t_{_{\mathcal{C}_{_{2}}}}(z_{_{1}},Y_{_{1}}))$,\ $dz$\ \ by\ \ $(\partial z_{_{\mathcal{C}_{_{2}}}}/\partial z_{_{1}})dz_{_{1}}+ (\partial z_{_{\mathcal{C}_{_{2}}}}/\partial Y_{_{1}})dY_{_{1}}$\ \ and\ \ $dt$\ \ by\ \ $(\partial t_{_{\mathcal{C}_{_{2}}}}/\partial z_{_{1}})dz_{_{1}}+ (\partial t_{_{\mathcal{C}_{_{2}}}}/\partial Y_{_{1}})dY_{_{1}}$,\ 
after simplification, we get the composite differential equation given by 
\begin{equation}
\frac{[(\mu_{_{2}}(z_{_{1}})Y_{_{1}}^{^{2}}+\mu_{_{1}}(z_{_{1}})Y_{_{1}} +\mu_{_{0}}(z_{_{1}})]}{z_{_{1}}((b_{_{1}}+1)z_{_{1}}^{^{2}}+3)}dz_{_{1}}+[\nu_{_{1}}(z_{_{1}})Y_{_{1}}+\nu_{_{0}}(z_{_{1}})]dY_{_{1}},
\label{eq:edocomp}
\end{equation}
 where
\begin{eqnarray*}
\mu_{_{2}}(z_{_{1}})&=&-(b_{_{1}}+1)^{^{4}}z_{_{1}}^{^{8}}-4(2b_{_{1}}+3)(b_{_{1}}+1)^{^{2}}z_{_{1}}^{^{6}}-2(b_{_{1}}+1)(13b_{_{1}}+1)z_{_{1}}^{^{4}}+12z_{_{1}}^{^{2}}+3,\\
\mu_{_{1}}(z_{_{1}})&=&4c[(b_{_{1}}+1)^{^{2}}z_{_{1}}^{^{6}}+(b_{_{1}}+1)(5b_{_{1}}-7)z_{_{1}}^{^{4}}+3(z_{_{1}}^{^{2}}+1)],\\
\mu_{_{0}}(z_{_{1}})&=&-12c^{^{2}}[(b_{_{1}}+1)z_{_{1}}^{^{2}}-1)][(b_{_{1}}-1)z_{_{1}}^{^{2}}+1],\\
\nu_{_{1}}(z_{_{1}})&=&-[(b_{_{1}}+1)^{^{3}}z_{_{1}}^{^{6}}+(b_{_{1}}+1)(7b_{_{1}}-1)z_{_{1}}^{^{4}}+(7b_{_{1}}-1)z_{_{1}}^{^{2}}+1],\\
\nu_{_{0}}(z_{_{1}})&=&2c [(b_{_{1}}+1)z_{_{1}}^{^{2}}-1][(b_{_{1}}-1)z_{_{1}}^{^{2}}+1].
\end{eqnarray*}

The equation \eqref{eq:edocomp} is singular at points\ \ $(z_{_{2}}=-1/\sqrt{b_{_{1}}+1},Y_{_{1}}=0)$\ \ and\ \ $(z_{_{1}}=1/\sqrt{b_{_{1}}+1},Y_{_{1}}=0)$, such points are, respectively, the intersection points of the double sonic locus\ \ $(z=z_{_{crit_{_{2}}}}, t=t_{_{2}},Y)$\ \ and\ \
$(z=z_{_{crit_{_{1}}}},t=t_{_{1}},Y))$\ \ with the Characteristic surface. The differential equation is not defined at\ \ $z_{_{1}}=0$. In fact, besides these two singularities there are two other ones at\ \ $z_{_{1}}=0$\ \ and\ \ $z_{_{1}}=\infty$, both of type saddle, see \cite{Eschenazi02}.

In the same way as for rarefaction curves, it is important to orient the composite curves according to the growth of\ \ $s$.

For the purpose of this paper, we will define \emph{local forward composite} to be a composite arc in\ \ $ Son'_{_{s}}$\ \ starting at\ \ $\mathcal {I}_{_{s}}$\ \ with decreasing\ \ $s$. We will define \emph{local backward composite} to be a composite arc in\ \ $Son'_{_{f}}$\ \ starting at\ \ $\mathcal {I}_{_{f}}$\ \ with increasing\ \ $s$.

\section{Wave Curves \label{subsec:WaveCurve}}

In this section, we construct wave curves in the wave manifold\ \ $\mathcal{W}$. A wave curve in\ \ $\mathcal{W}$\ \ is a succession of rarefaction, shock and composite waves (forward and backward) satisfying certain conditions. Our list is not complete, not all possibilities are described, we focus only wave curves used here. We follow Section 4 of \cite{AEMP10}.\\



\subsection{Local Wave Curves \label{sec:lwc}}



 
		


 We begin by constructing a local forward wave curve as follows:
start with a point\ \ $P_{_{0}}$\ \ in\ \ $\mathcal{C}_{_{s}}$. Let\ \ $s_{_{0}}$\ \ be the value of the speed at\ \ $P_{_{0}}$. The first arc is a forward rarefaction, contained in\ \ $\mathcal{C}_{_{s}}$, starting at\ \ $P_{_{0}}$, with speed increasing\ \ $s$. It is easy to see that there are only two possibilities: either the arc stops at the inflection locus ($s$\ \ stops increasing) or the arc stops at the coincidence locus (the arc leaves\ \ $\mathcal{C}_{_{s}}$). Let\ \ $P_{_{1}}$\ \ be the end point at this forward rarefaction arc. The second arc is a forward shock arc starting at\ \ $P_{_{0}}$, with decreasing speed\ \ $s$. Again there are two possibilities. Either the arc goes to infinity, or it stops at a point\ \ $P_{_{2}}$ in the sonic surface ($s$\ \ stops decreasing).

If\ \ $P_{_{1}}$\ \ is on the coincidence locus, the wave curve stops there. If there is no point\ \ $P_{_{2}}$, the wave curve is an infinite forward shock arc starting at\ \ $P_{_{0}}$, or $Sh_{P_{0\infty}}$, and the forward rarefaction from\ \ $P_{_{0}}$\ \ to\ \ $P_{_{1}}$,\ $R{P_{_{0}}P_{_{1}}}$.

If\ \ $P_{_{1}}$\ \ is on the inflection locus, we add a third arc: the forward composite arc starting at\ \ $P_{_{1}}$\ \ with  decreasing\ \ $s$, which will stop at a point\ \ $P_{_{3}}$\ \ on\ \ $Son'$, with speed\ \ $s_{_{3}}$, namely\ \ $Co_{P_{_{1}}P_{_{3}}}$. The latter is either at the double contact (double sonic locus), or\ \ $s_{_{3}}=s_{_{1}}$, whichever comes first.

In case\ \ $s_{_{3}}=s_{_{1}}$, we add the non-local forward shock arc starting at\ \ $P_{_{3}}$\ \ with decreasing\ \ $s$\ \  until\ \ $P_{_{6}}$, $NLSh_{P_{_{3}}P_{_{6}}}$. If\ \ $P_{_{3}}$\ \ is on the double contact, we consider the auxiliary Hugoniot$'$ arc starting at\ \ $P_{_{3}}$,\ $hug'_{P_{_{3}}P_{_{4}}}$. The arc will hit\ \ $\mathcal{C}_{_{s}}$\ \ at a point\ \ $P_{_{4}}$. From\ \ $P_{_{4}}$\ \ we take the forward rarefaction until\ \ $P_{_{5}}$, contained in\ \ $\mathcal{C}_{_{s}}$,\ $R_{P_{_{4}}P_{_{5}}}$ with increasing\ \ $s$.\\

\noindent
We may summarize as:


\noindent	
  $\bullet$ {\bf Case 1}- For\ \ $P_{_{1}}$\ \ on the coincidence locus. The wave curve is\ \ $Sh_{P_{0\infty}}\cup R_{P_{_{0}}P_{_{1}}}$. The speed\ \ $s$\ \  decreases along shock arc and increases along rarefaction arc.\medskip

	
\noindent		
$\bullet$ {\bf Case 2}- For\ \ $P_{_{1}}$\ \ on the inflection locus and\ \ $s_{_{3}}=s_{_{1}}$. The wave curve is\ \ $Sh_{P_{0\infty}}\cup R_{P_{_{0}}P_{_{1}}}\cup Co_{P_{_{1}}P_{_{3}}}\cup NLSh_{P_{_{3}}P_{_{6}}}$. The speed\ \ $s$\ \ decreases along shock and composite arcs and increases along the rarefaction arc.\medskip





\noindent
$\bullet$ {\bf Case 3}- For\ \ $P_{_{3}}$\ \ on the double contact. The wave curve is\ \ $Sh_{P_{0\infty}}\cup R_{P_{_{0}}P_{_{1}}}\cup Co_{P_{_{1}}P_{_{3}}}\cup hug'_{P_{_{3}}P_{_{4}}}\cup R_{P_{_{4}}P_{_{5}}}$. The speed\ \ $s$\ \  decreases along the shock and composite arcs and increases along the rarefaction arcs.\\


\noindent
We have 3 more cases, replacing infinity by a point\ \ $P_{_{2}}$\ \ on\ \ $Son$.



We define a local backward wave curves from a point\ \ $P_{_{0}}\in \mathcal{C}_{_{f}}$\ \ as a succession of a backward rarefaction, a local backward shock and a local backward composite. We can also describe similar wave curves for backward arcs.

We remark that wave curves in the wave manifold are not smooth objects, in case 3 they are not even continuous. The intermediate surfaces, introduced in \cite{AEMP10}, are the better correspondent in the wave manifold for the wave curves in the state space.\\

\subsection{Decomposing $C_{_{s}}$ according to wave curves structures}




 To obtain local forward wave curves we need to divide\ \ $\mathcal{C}_{_{s}}$\ \ into regions such that wave curves starting at points of each region have similar properties.

We know that along rarefaction curves in\ \ $\mathcal{C}_{_{s}}$\ \ the speed\ \ $s$\ \ increases towards\ \ $\mathcal{I}_{_{s}}$, \emph{i.e.}, for a given state in\ \ $(z,t)\in \mathcal{C}_{_{s}}$:\medskip

	
\noindent
\textbf{1.-} If\ \ $(z,t)$\ \ lies on the left side of\ \ $\mathcal{I}_{_{s}}$, then the forward rarefaction through it, is constructed with increasing values of\ \ $z$.
	 
\noindent
\textbf{2.-} If\ \ $(z,t)$\ \ lies on the right side of\ \ $\mathcal{I}_{_{s}}$, then the forward rarefaction through it, is constructed with decreasing values of\ \ $z$.\medskip

In our model, there are states\ \ $(z,t)$\ \ for which the forward rarefaction through these points reaches\ \ $\mathcal{I}_{_{s}}$, as well as there are states for which the forward rarefaction through it,  does not reach\ \ $\mathcal{I}_{_{s}}$. In the second case, the forward rarefaction stops at the coincidence curve, the boundary of\ \ $C_{_{s}}$\ \ and\ \ $C_{_{f}}$. So, it is necessary to decompose\ \
$\mathcal{C}_{_{s}}$\ \ in regions for which these different behaviors occur.

To obtain this decomposition, we look for points\ \ $(z,t)$\ \ for which a rarefaction curve and the inflection curve\ \ $\mathcal{I}$\ \ are tangent.
Differentiating the equation of the inflection curve\ \ $\eqref{eqinfle}$\ \ with respect to\ \ $z$, we get:

\begin{equation}
\frac{dt}{dz}=\frac{3z^{^{2}}b_{_{1}}+b_{_{1}}+4}{(z^{^{2}}+1)^{^{2}}(3+(b_{_{1}}+1)z^{^{2}})}.\label{difto}
\end{equation}
Now, equating $(\ref{difto})$ to the rarefaction field given by $(\ref{eq:edoraref})$, with\ \ $t$\ \ given by\ \ $(\ref{eqinfle})$,
 we get that the\ \ $z$\ \  coordinates of the tangency points are the roots of:

\begin{equation}
	\frac{-3+(3b_{_{1}}+3)z^{^{2}}}{(3+(b_{_{1}}+1)z^{^{2}})^{^{2}}z^{^{2}}}=0,
	\label{eqdiff}
\end{equation}
  $z_{_{1}}=1/\sqrt{b_{_{1}}+1}$\ \ and\ \  $z_{_{2}}=-1/\sqrt{b_{_{1}}+1}$, \emph{i.e.}, the rarefaction curve and the inflection curve, $\mathcal{I}$, are tangent at the intersection points of $\mathcal{C}$ with the double contact,\ $(z_{_{crit_{_{1}}}}, t_{_{1}},Y)$\ \ and\ \ $(z_{_{crit_{_{2}}}},t_{_{2}},Y)$. 

We have proved the following:


\begin{proposition}
 The rarefaction and inflection curves are tangent at the states\ \  $(z_{_{crit_{_{1}}}},t_{_{1}})\in \mathcal{C}_{_{s}}$\ \ and\ \   $(z_{_{crit_{_{2}}}},t_{_{2}})\in \mathcal{C}_{_{f}}$.
\end{proposition}

\noindent
 Rarefaction curves through\ \ $(z_{crit_{_{1}}}, t_{_{1}})$\ \ and\ \ $(z_{crit_{_{2}}}, t_{_{2}})$\ \ subdivide\ \ $\mathcal{C}_{_{s}}$\ \ into three regions:

The rarefaction curve through\ \ $(z_{_{crit_{_{2}}}},t_{_{2}})$\ \ with\ \ $z\leq z_{_{crit_{_{2}}}}$\ \ is denoted by\ \ $\mathcal{R}_{_{f}}$. This curve reaches\ \ $t=0$\ \ at a point\ \ $\hat{z}$. 
Let\ \ $L_{_{1}}$\ \ and\ \ $L_{_{2}}$\ \ be the straight lines\ \ $(z,0)$\ \ with\ \ $z\leq \hat{z}$\ \ and\ \ $z>\hat{z}$, respectively. The part of\ \ $\mathcal{R}_{_{f}}$\ \ for\ \ $z\leq \hat{z}$\ \ is denoted as\ \ $\mathcal{R}_{_{fs}}$. 
The region\ \ $I_{_{a}}$\ \ is formed by states\ \ $(z,t)$\ \ between\ \ $L_{_{1}}$\ \ and\ \ $\mathcal{R}_{_{fs}}$. The region\ \ $I_{_{b}}$\ \ is formed by states\ \ $(z,t)$\ \ between\ \ $\mathcal{R}_{_{fs}}\cup L_{_{2}}$\ \ and\ \ $\mathcal{R}_{_{s}} \cup \mathcal{I}_{_{s}}$ ($z>z_{_{crit_{_{1}}}}$). We define the region\ \ $I$\ \ as\ \  $I_{_{a}}\cup I_{_{b}}$. The region\ \ $II$\ \ is formed by states\ \ $(z,t)$ below\ \ $\mathcal{R}_{_{s}}$\ \ and on the left side of\ \ $\mathcal{I}_{_{s}}$($z<z_{_{crit_{_{1}}}}$). The region\ \ $III$\ \ is formed by states\ \ $(z,t)$\ \ on the right side of\ \ $\mathcal{I}_{_{s}}$, see Fig. \ref{inflecfig}.

\newpage


\begin{figure}[htpb]
	\begin{center}	\includegraphics[scale=1.0,width=0.8\linewidth]{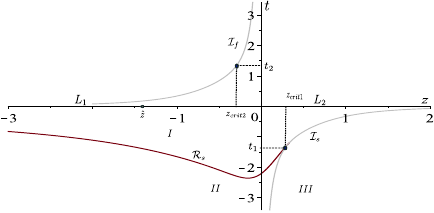}
		\caption[]{The three regions defined in\ \ $\mathcal{C}_{_{s}}$. The branch of inflection in\ \ $\mathcal{C}_{_{s}}$\ \ is denoted by\ \ $\mathcal{I}_{_{s}}$\ \ and the branch of inflection in\ \ $\mathcal{C}_{_{f}}$\ \ is denoted by\ \ $\mathcal{I}_{_{f}}$. \label{inflecfig}} 
	\end{center}
\end{figure}


\section{Riemann Solution \label{sec:RiemannSol}}

The Riemann solution is constructed by using the wave curves defined in the Section \ref{subsec:WaveCurve}. For a pair\ \ $(W_{_{L}},W_{_{R}})$\ \ given in the hyperbolic region, we utilize the following algorithm to obtain the solution.\\

\noindent
\textbf{Algorithm RS:}\\

\noindent
$\bullet$ {\it For a pair\ \ $(W_{_{L}},W_{_{R}})$\ \ both given in hyperbolic region  in the phase space, by substituting\ \ $W_{_{L}}=(u_{_{L}},v_{_{L}})$\ \ into\ \  $(\ref{eq:parahug})$\ \ and setting\ \ $Y=0$, we obtain two states in the wave manifold\ \ $\mathcal{W}$: $\mathcal{U}_{_{L}}=(z_{_{1}},t_{_{1}}<0,0)_{_{L}}\in\mathcal{C}_{_{s}}$\ \ and\ \ $\displaystyle{\tilde{\mathcal{U}}_{_{L}}=(z_{_{2}},t_{_{2}}>0,0)_{_{L}}\in\mathcal{C}_{_{f}}}$}.\medskip

\noindent
$\bullet$ {\it Similarly, by substituting\ \ $W_{_{R}}=(u_{_{R}},v_{_{R}})$\ \ into\ \  $(\ref{eq:parahug})$\ \ and setting\ \ $Y=0$, we obtain two states in the wave manifold\ \ $\mathcal{W}$: $\tilde{\mathcal{U}}_{_{R}}=(z_{_{1}},t_{_{1}}<0,0)_{_{R}}\in\mathcal{C}_{_{s}}$\ \ and\ \ ${\mathcal{U}}_{_{R}}=(z_{_{2}},t_{_{2}}>0,0)_{_{R}}\in\mathcal{C}_{_{f}}$}.\medskip

\noindent
$\bullet$ {\it To construct the solution, we utilize the states\ \ $\mathcal{U}_{_{L}}$\ \ and\ \ $\mathcal{U}_{_{R}}$}.\medskip

\noindent
$\bullet$ {\it From\ \ $\mathcal{U}_{_{L}}$\ \ we draw the waves of 1-family, described in Subsection \ref{subsec:WaveCurve}}.\medskip

\noindent
$\bullet$ {\it After we constructed the 1-family, we saturated the curves as described in Subsection \ref{saturatedsurface}}.\medskip

\noindent
$\bullet$ {\it From the state\ \ $\mathcal{U}_{_{R}}$, we construct the wave of 2-reverse wave sequence, described in Section $\ref{2reverse}$}.\medskip

\noindent
$\bullet$ {\it The Riemann solution consists of in the sequence of waves and states obtained through the intersection between the saturated surface through\ \ $\mathcal{U}_{_{L}}$\ \ and the 2-reverse wave   through\ \ $\mathcal{U}_{_{L}}$}.\medskip

\noindent
$\bullet$ {\it After, we obtain the solution on the plane\ \ $uv$}.\medskip

\begin{remark}
\label{remarkdef}
The Riemann solution is a sequence of shocks, rarefactions and constant states. 
In the wave manifold\ \ $\mathcal{W}$, we denote:\\

\noindent
$\bullet$ Hugoniot curve arcs:\ \ $\mathcal{H}_{_{1}}$\ \ for curves starting at\ \ $t<0$\ \ and\ \ $\mathcal{H}_{_{2}}$\ \ starting at\ \ $t>0$. \medskip

\noindent
$\bullet$ Rarefactions curves:\ \  $\mathcal{R}_{_{1}}$\ \ for curves starting at\ \ $t<0$\ \ and\ \
$\mathcal{H}_{_{2}}$\ \ starting at\ \ $t>0$. \medskip

\noindent
$\bullet$ Composite curves:\ \  $\mathcal{C}_{_{1}}$\ \ for curves starting at\ \ $t<0$\ \ and\ \ $\mathcal{C}_{_{2}}$\ \ starting at\ \ $t>0$.\\

Given a state\ \ $\mathcal{U}\in\mathcal{W}$, if the wave through this state\ \ $z$\ \ is positive, then identify this wave on right of\ \ $\mathcal{U}$. On the other hand, if\ \ $z$\ \ is negative the wave is identified on left of\ \ $\mathcal{U}$. For instance, from state\ \ $\mathcal{U}$\ \ we have a\ \ $\mathcal{R}_{_{1}}$\ \ followed by\ \ $\mathcal{C}_{_{1}}$\ \ for increasing\ \ $z$\ \ and\ \ $\mathcal{H}_{_{1}}$\ \ for decreasing\ \ $z$, we denote this sequence through\ \ $\mathcal{U}$\ \ as\ \ $\mathcal{H}_{_{1}}\mathcal{U}\mathcal{R}_{_{1}}\mathcal{C}_{_{1}}$. The saturated surfaces follow the same sequence of waves.   
\end{remark}

\begin{remark}
It is necessary to verify that the wave sequence satisfies the geometrical compatibility, \emph{i.e.}, the waves in the sequence have increasing speed.
\end{remark}

In the following sections, we divide the\ \ $\mathcal{C}_{_{s}}$\ \ into different parts for which the solution is equal in each region. With these structures, we can obtain the Riemann solution.

\subsection{The wave curves in each region}
\label{subsec:sectwave}\medskip

In this section, we describe the construction of local forward wave curves in each of the four regions
\ \ $I$\ \ to\ \ $III$, in\ \ $\mathcal{M}-\Pi$, the interpretation of such wave curve in the\ \ $(u,v)$-plane will be seen in Subsection \ref{subsec:secwaveuv}.\\ 


\noindent
We consider a state\ \ $(z_{_{i}},t_{_{i}})\in I$, we refer to Fig. $\ref{secondcase}$.\\

\noindent
$\bullet$ \textit{Rarefaction}: In this case a rarefaction curve does not cross the inflection\ \
$\mathcal{I}$ (neither\ \ $\mathcal{I}_{_{s}}$\ \ nor\ \ $\mathcal{I}_{_{f}}$). We construct the forward rarefaction\ \ $\mathcal{R}_{_{1}}$\ \ starting at\ \ $(z_{_{i}},t_{_{i}})$.\ \ $\mathcal{R}_{_{1}}$\ \ to
\ \ $L_{_{1}}$\ \ at\ \ $(\tilde{z},0)$. We stop the construction of rarefaction because from this point we will have a rarefaction in\ \ $\mathcal{C}_{_{f}}$.\medskip
	 
	
	
\noindent
$\bullet$ \textit{Hugoniot}: From\ \ $(z_{_{i}},t_{_{i}})$\ \ we construct the local forward shock arc\ \
$\mathcal{H}_{_{1}}$, with decreasing\ \ $s$ (decreasing values of\ \ $z$).

\begin{figure}[htpb]
	\begin{center}	\includegraphics[scale=0.4,width=0.4\linewidth]{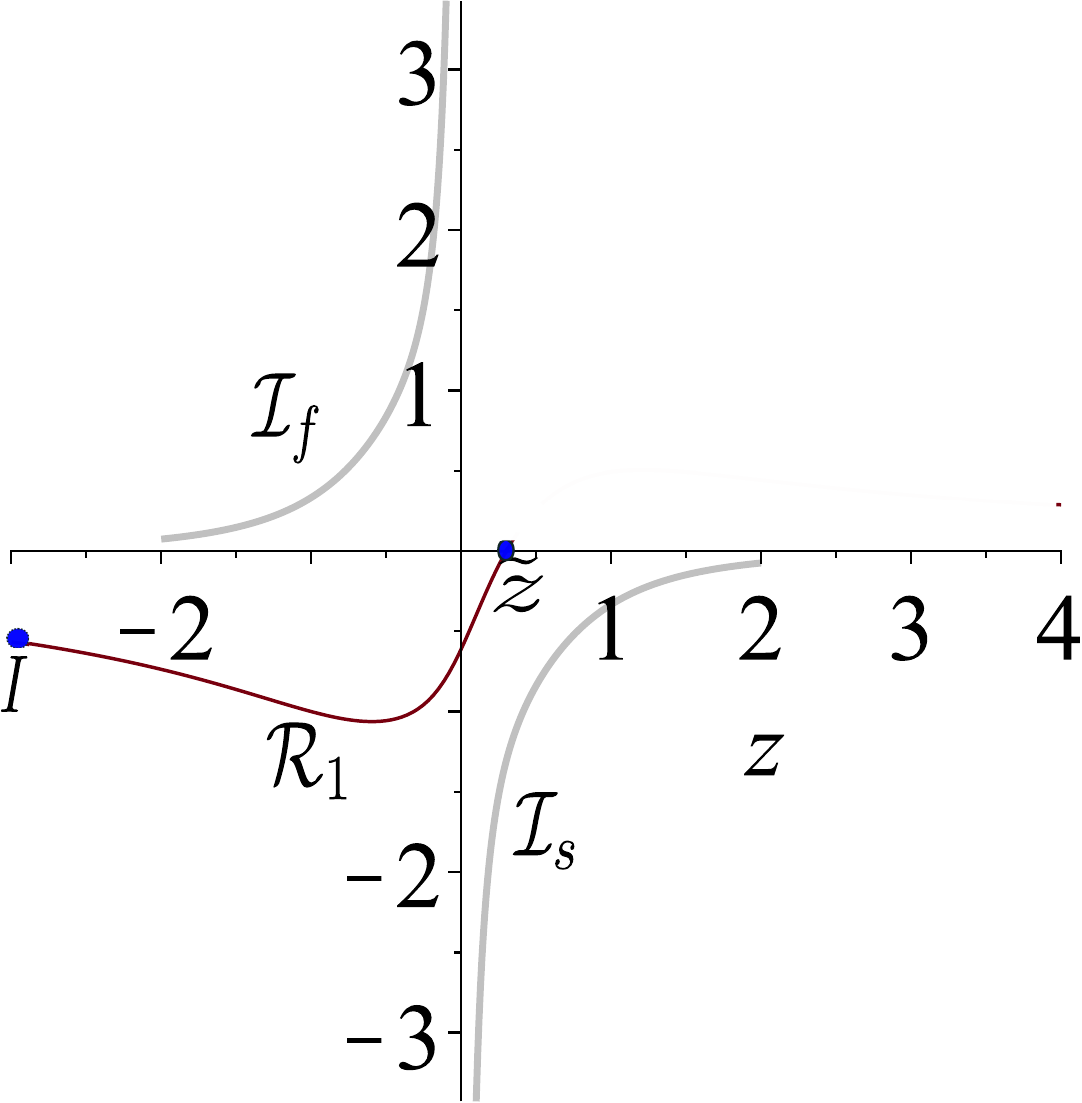}\hspace{1cm}	\includegraphics[scale=0.4,width=0.4\linewidth]{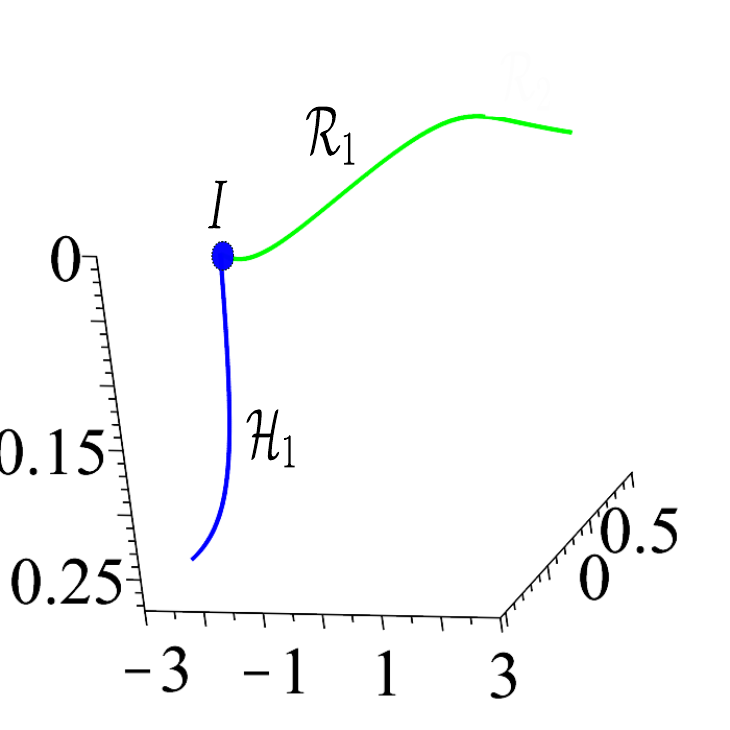}	
		\caption[]{Waves sequence for initial data in region $II$. \textit{Left a:)-} The rarefaction $\mathcal{R}_{_{1}}$ in the characteristic plane $\mathcal{C}$.  \textit{Right b:)-} The wave sequence $\mathcal{H}_{_{1}}$,  $\mathcal{R}_{_{1}}$,   in $\mathcal{W}-\Pi$ . \label{secondcase}} 
	\end{center}
\end{figure}

\newpage

\noindent
 For a state $(z_{_{i}},t_{_{i}})\in II$, we refer to Fig. $\ref{thirdcase}$.\\
 
\noindent
$\bullet$ \textit{Rarefaction:} From\ \ $(z_{_{i}},t_{_{i}})$\ \ we construct the forward rarefaction
\ \ $\mathcal{R}_{_{1}}$. This rarefaction intersects the branch\ \ $\mathcal{I}_{_{s}}$\ \ of the inflection curve at\ \ $(z^\dag,t^\dag)$,\ $z^\dag<z_{_{crit_{_{1}}}}$.\medskip
	
\noindent
$\bullet$ \textit{Composite:}  From\ \ $({z}^\dag,{t}^\dag)$\ \ we construct the forward composite\ \
$\mathcal{C}_{_{1}}$\ \ parameterized by\ \ $z$,\  $z^\dag \leq z \leq z_{_{crit_{_{1}}}}$.\medskip 
	
\noindent
$\bullet$ \textit{Hugoniot:} The forward shock curve arc\ \ $\mathcal{H}_{_{1}}$\ \ is constructed for decreasing\ \ $s$ (decreasing values of\ \ $z$).

\begin{figure}[htpb]
	\begin{center}	\includegraphics[scale=0.4,width=0.4\linewidth]{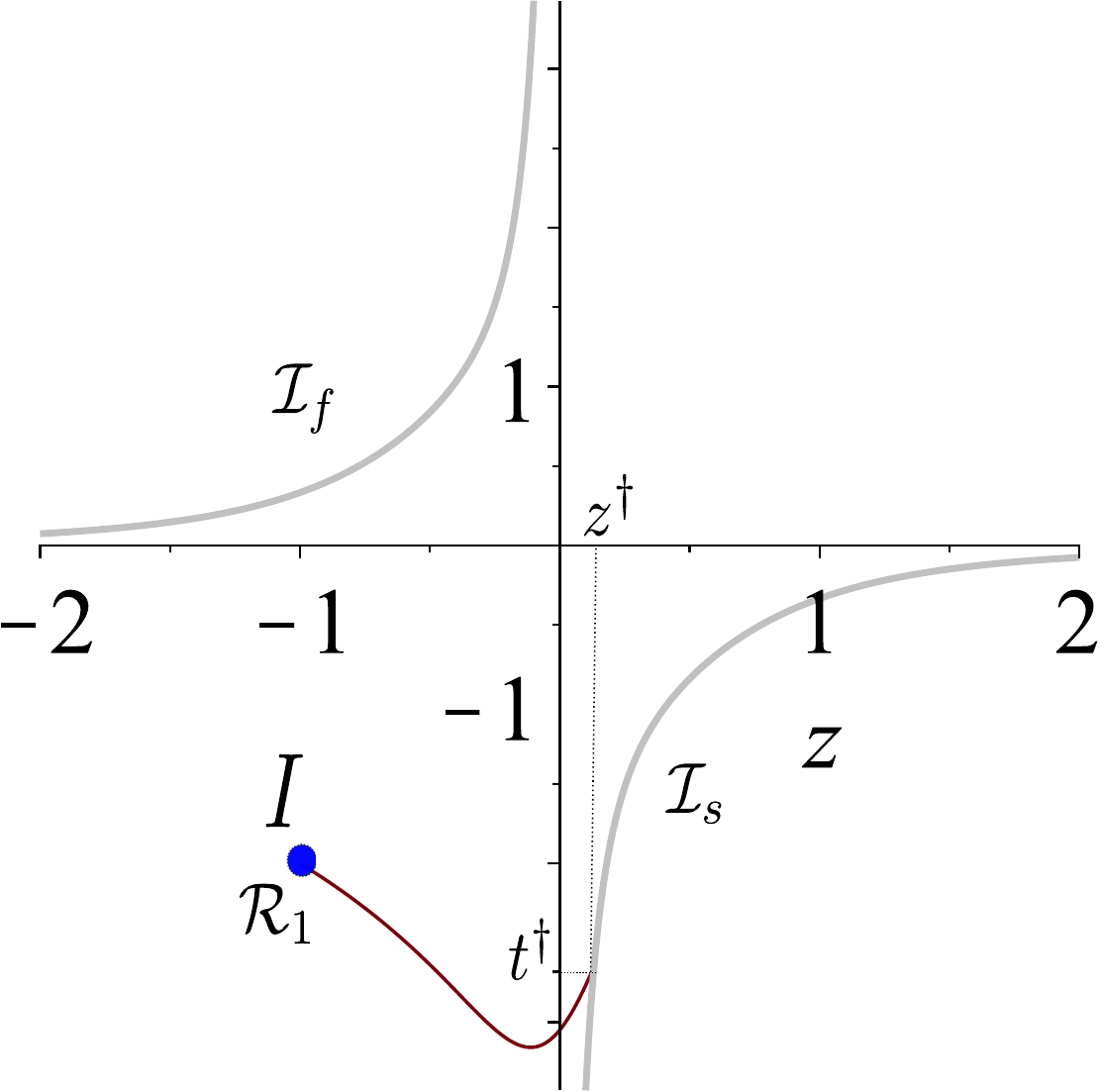}\hspace{1cm}	\includegraphics[scale=0.4,width=0.4\linewidth]{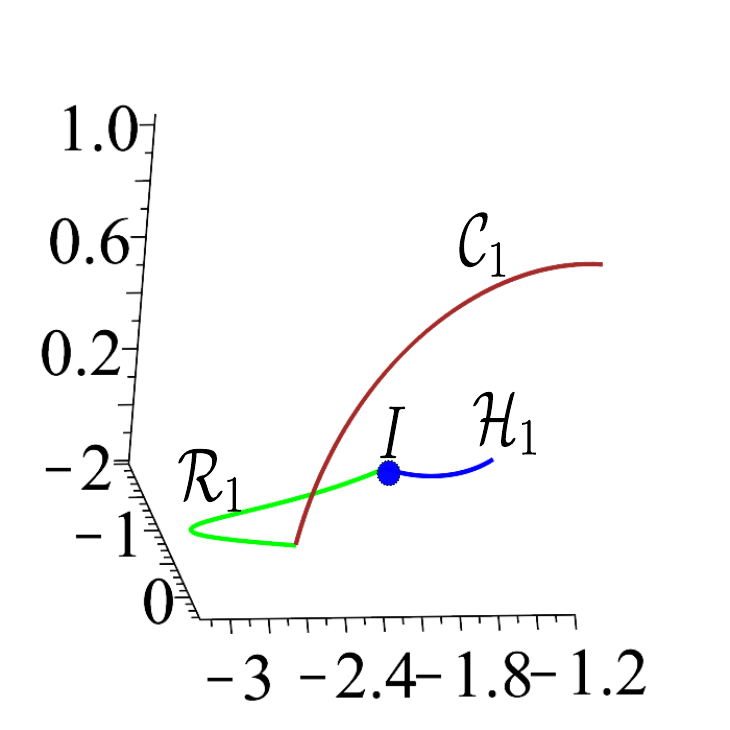}	
		\caption[]{{Waves sequence for initial data in region $II$. \textit{Left a:)-} The rarefactions $\mathcal{R}_{_{1}}$ in $\mathcal{C}$ reaches the branch of inflection $\mathcal{I}_{_{s}}$.  \textit{Right b:)-} The wave sequence $\mathcal{H}_{_{1}}$,  $\mathcal{R}_{_{1}}$ and $\mathcal{C}_{_{1}}$ in $\mathcal{W}-\Pi$ . \label{thirdcase}}} 
	\end{center}
\end{figure}

\noindent
For a state $(z_{_{i}},t_{_{i}})\in III$, we refer to Fig. $\ref{fourthcase}$.\\

\noindent
$\bullet$  \textit{Rarefaction:} From\ \ $(z_{_{i}},t_{_{i}})$, we construct the forward rarefaction\ \
$\mathcal{R}_{_{1}}$\ \ up to\ \ $({z}^\dag,{t}^\dag)$, the intersection\ \ $\mathcal{R}_{_{1}}  \cap \mathcal{I}_{_{s}}$,\ $z_{_{crit_{_{2}}}}<z^\dag<z_{_{crit_{_{1}}}}$.\medskip
	
\noindent
$\bullet$  \textit{Composite:} From\ \ $({z}^\dag,{t}^\dag)$\ \  we construct the forward composite
\ \ $\mathcal{C}_{_{1}}$\ \ parameterized by\ \ $z$,\  $z^\dag \leq z \leq z_{_{crit_{_{1}}}}$.\medskip
	
\noindent
$\bullet$   \textit{Hugoniot:} From\ \ $(z_{_{i}},t_{_{i}})$\ \ we construct the forward shock curve arc
\ \ $\mathcal{H}_{_{1}}$\ \ for decreasing\ \ $s$ (increasing values of\ \ $z$).

\begin{figure}[htpb]
	\begin{center}	\includegraphics[scale=0.4,width=0.4\linewidth]{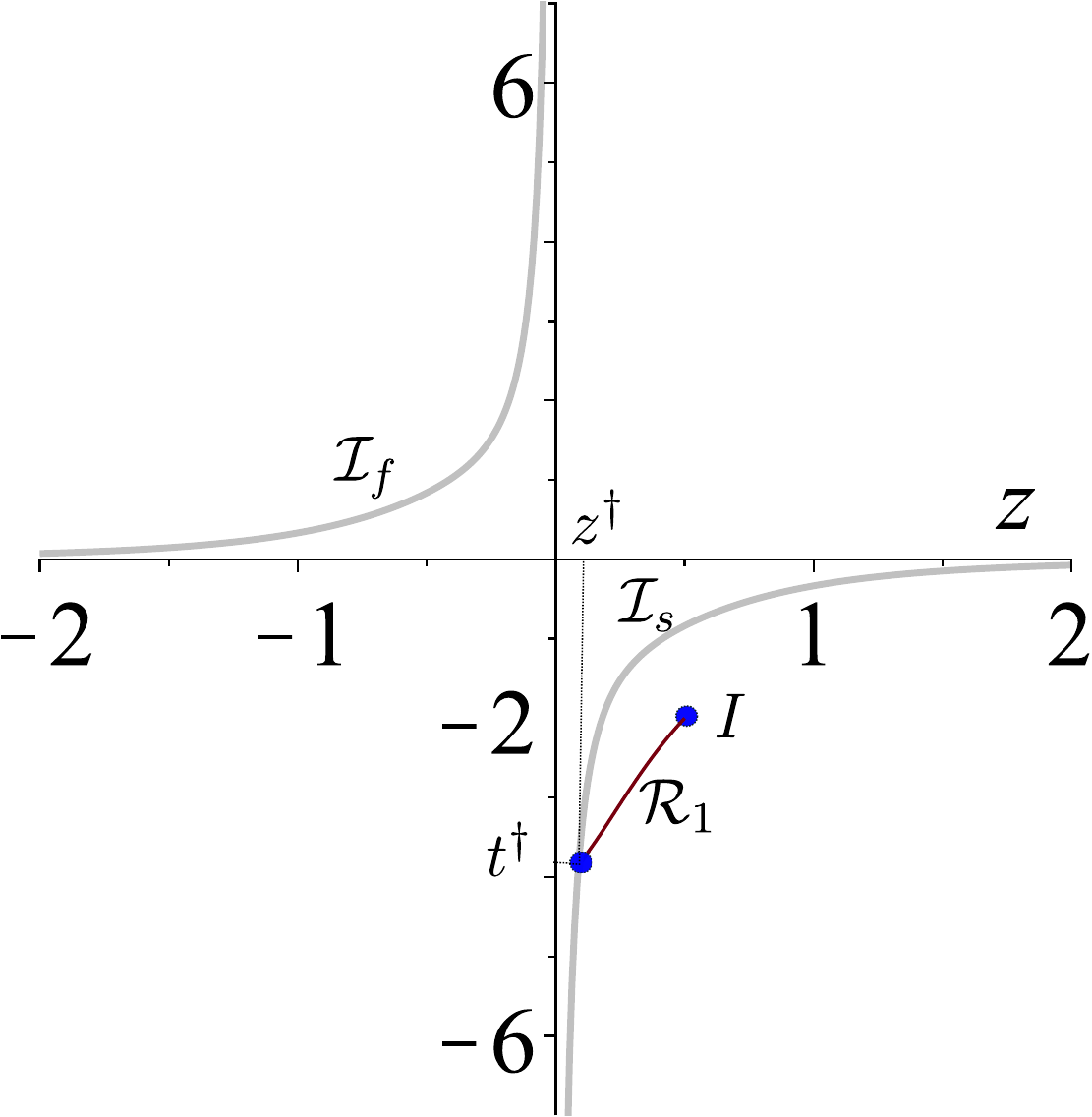}\hspace{1cm}	\includegraphics[scale=0.4,width=0.4\linewidth]{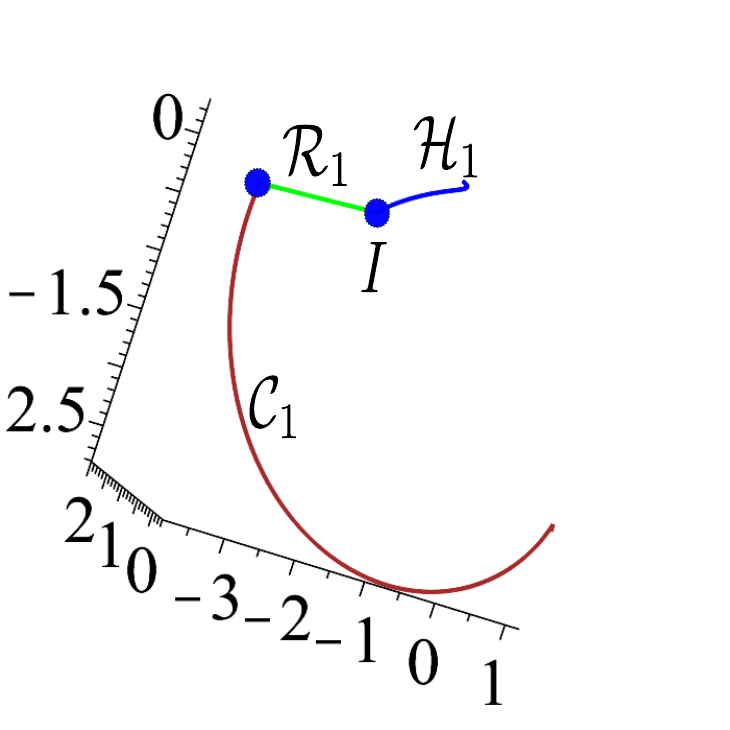}	
		\caption[]{Waves sequence for initial data in region\ \ $III$. \textit{Left a:)-} The 
		rarefactions\ \ $\mathcal{R}_{_{1}}$\ \ in\ \ $\mathcal{C}$\ \ reaches the branch of inflection
		\ \ $\mathcal{I}_{_{s}}$\ \ for decreasing\ \ $z$.  \textit{Right b:)-} The wave sequence\ \
		$\mathcal{H}_{_{1}}$,\  $\mathcal{R}_{_{1}}$\ \ and\ \ $\mathcal{C}_{_{1}}$\ \ in\ \ $\mathcal{W}-\Pi$. \label{fourthcase}} 
	\end{center}
\end{figure}

\subsection{The wave curves on the plane uv}
\label{subsec:secwaveuv}

 The construction of the local forward wave curve in\ \ $\mathcal{W}-\Pi$\ \ allow us to obtain the sequence of the 1-waves curves in the original state space, the\ \ $uv$-plane.

\begin{remark}
	The meaning of colors used in this paper by software TORS to represent rarefactions, shocks, composites etc... do not have the same meaning of the colors used by the software ELI.
\end{remark}

From the wave sequence defined in Subsection \ref{subsec:sectwave}, we utilize the mapping between  the wave manifold\ \ $\mathcal{W}-\Pi$\ \ and the plane\ \ $uv$\ \ given by the equations $(\ref{u})$ -- $(\ref{v})$.


There is a mapping between the elliptic boundary and the\ \ $SCC$. An interesting feature is that the straight line\ \ $(z,t=0)$\ \ in the\ \ $\mathcal{W}$\ \ represents the elliptic boundary on the plane\ \ $uv$.

Here, we obtain the same sequence of waves in each region\ \ $I$\ \ to\ \ $III$ (of wave manifold\ \
$\mathcal{W}$\ \ in\ \ $\mathcal{C}_{_{s}}$)  on the plane\ \ $uv$.

For\ \ $I=(z_{_{i}},t_{_{i}})\in I$, the wave sequence is as in Fig. $\ref{doistrescaseuv}.a$. 
There is a shock curve\ \ $\mathcal{S}_{_{1}}$\ \ and a rarefaction\ \ $\mathcal{R}_{_{1}}$. Since in region\ \ $I$\ \ the rarefaction crosses from slow region to fast region (changing from slow rarefaction to fast rarefaction)  the rarefaction\ \ $\mathcal{R}_{_{1}}$\ \ touches elliptic boundary and then the rarefaction changes to\ \ $\mathcal{R}_{_{2}}$. Since the rarefaction\ \ $\mathcal{R}_{_{2}}$\ \ does not touch the inflection\ \ $\mathcal{I}_{_{f}}$, there is no composite wave here,   see Fig. \ref{doistrescaseuv}.{\it a}.

For\ \ $I=(z_{_{i}},t_{_{i}})\in II$, the wave sequence is as in Fig. \ref{doistrescaseuv}.{\it b}. 
There are a shock curve\ \ $\mathcal{S}_{_{1}}$\ \ and a rarefaction\ \ $\mathcal{R}_{_{1}}$. The rarefaction\ \ $\mathcal{R}_{_{1}}$\ \ is drawn to the inflection\ \ $\mathcal{I}_{_{s}}$. From\ \ $\mathcal{I}_{_{s}}$\ \ we draw the composite wave\ \ $\mathcal{C}_{_{1}}$ to the straight line that represents the double contact, see Fig. \ref{doistrescaseuv}.{\it a}

\begin{figure}[htpb]
	\begin{center}	\includegraphics[scale=0.4,width=0.4\linewidth]{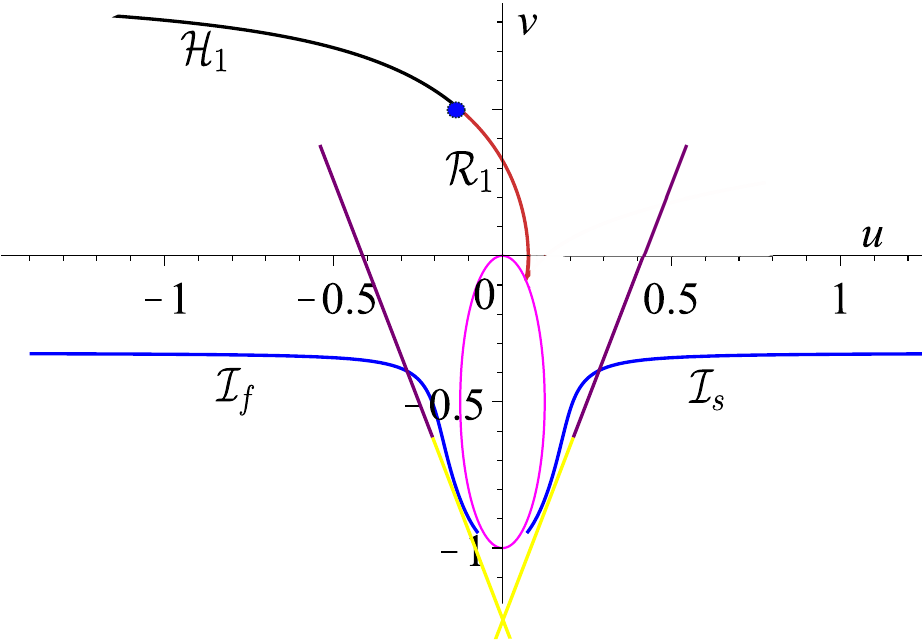}\hspace{1cm}	\includegraphics[scale=0.4,width=0.4\linewidth]{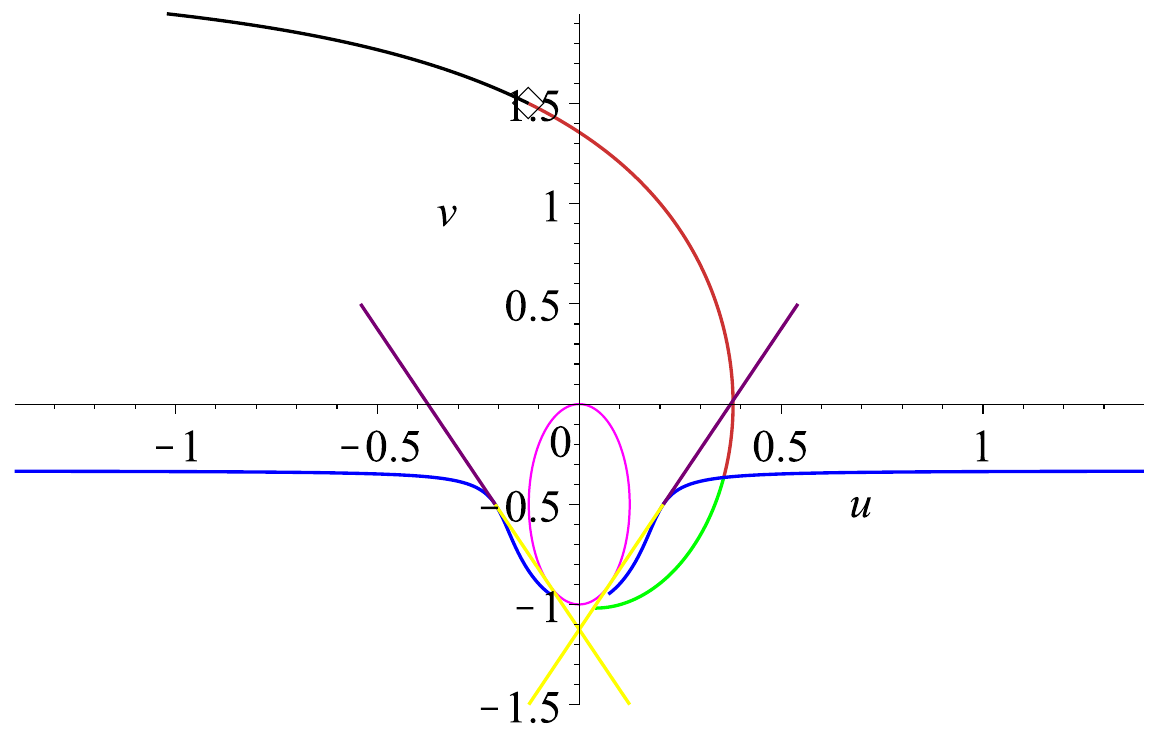}	
		\caption[]{\textit{ Left a:)} The wave sequence for\ \ $I=(z_{_{i}},t_{_{i}})\in I$\ \ on the plane\ \ $uv$. \textit{ Right b:) The wave sequence for\ \ $I=(z_{_{i}},t_{_{i}})\in III$\ \ on the plane\ \ $uv$.} \label{doistrescaseuv}} 
	\end{center}
\end{figure}

For\ \ $I=(z_{_{i}},t_{_{i}})\in III$, the wave sequence is as in Fig. \ref{quatrocaseuv}. 
There are a shock curve\ \ $\mathcal{S}_{_{1}}$\ \ and a rarefaction\ \ $\mathcal{R}_{_{1}}$. The rarefaction\ \ $\mathcal{R}_{_{1}}$\ \ is drawn to the inflection\ \ $\mathcal{I}_{_{s}}$. From\ \
$\mathcal{I}_{_{s}}$\ \ we draw the composite wave\ \ $\mathcal{C}_{_{1}}$\ \ to the straight line that represents the double contact, see Fig. \ref{quatrocaseuv}.

\begin{figure}[htpb]
	\begin{center}	\includegraphics[scale=0.4,width=0.4\linewidth]{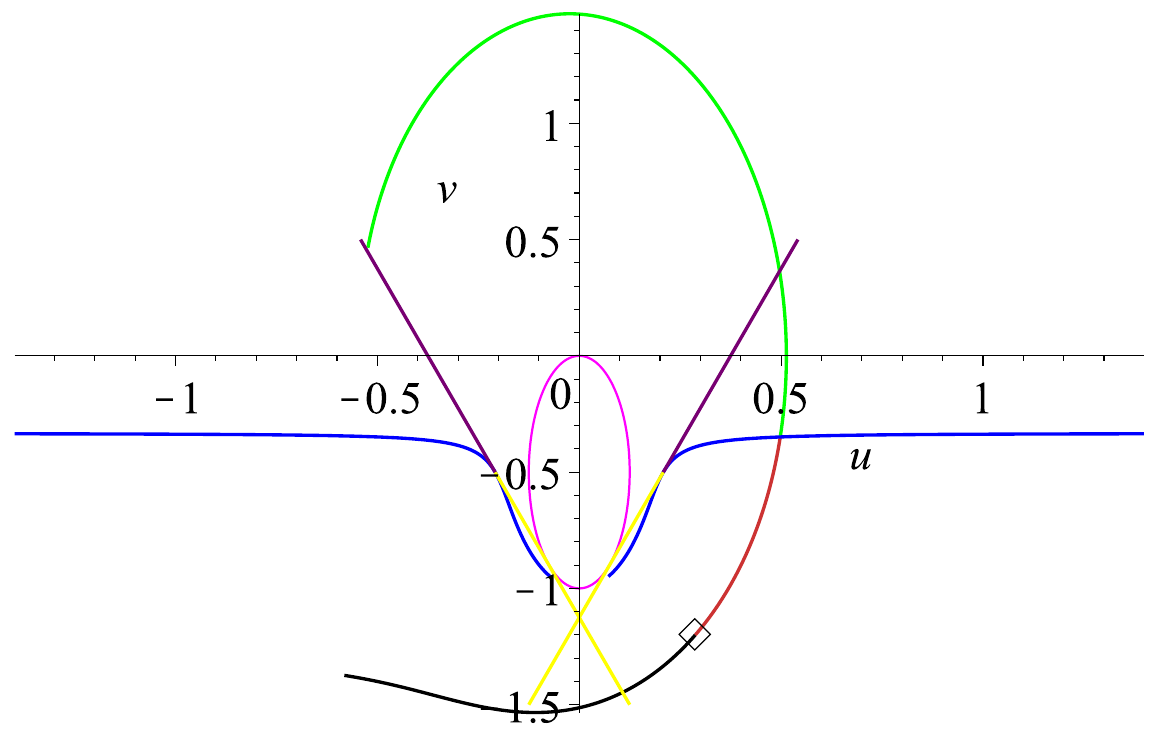}\hspace{1cm}
		\caption[]{The wave sequence for\ \ $I=(z_{_{i}},t_{_{i}})\in III$\ \ on the plane\ \ $uv$. \label{quatrocaseuv}} 
	\end{center}
\end{figure}

\subsection{The saturated surfaces}
\label{saturatedsurface}

The wave curves in slow region\ \ $(t<0)$\ \ represents on the plane\ \ $uv$ (physical planes) waves of family 1 (shock, $\mathcal{S}_{_{1}}$, rarefaction, $\mathcal{R}_{_{1}}$ and composite, $\mathcal{C}_{_{1}}$) .

It is necessary to study the interaction of this family 1 with the family 2. Here in the wave manifold, the curve from  the slow region\ \ $(t<0)$\ \ should access  their corresponding projection in\ \ $C_{_{f}}$.

We know that give a state\ \ $U$\ \ in wave manifold\ \ $\mathcal{W}$, we can obtain a unique Hugoniot$'$ through this point. Here, we are assuming  that\ \ $U$\ \ lies in the Hyperbolic region, \emph{i.e.}, the Hugoniot$'$ (and Hugoniot) through\ \ $U$\ \ crosses\ \ $\mathcal{C}$\ \ in\ \ $\mathcal{C}_{_{s}}$\ \
and\ \ $\mathcal{C}_{_{f}}$. 

Using the saturation of each wave,  we obtain a surface that crosses\ \ $\mathcal{C}_{_{s}}$\ \ to\ \
$\mathcal{C}_{_{f}}$, as we explain as follows.

First, we consider a state\ \ $U\in\mathcal{C}_{_{s}}$. From\ \ $U$, we draw the wave sequence (as described in Section \ref{sectwave}). These wave sequence defines a connected curve as we denote as
\ \ $C$:
\begin{equation}
C=\{(z,t(z),Y(z))\quad\quad\text{ such that } \quad \quad z\in \Theta\},
\end{equation}
where is\ \ $\Theta$\ \ is a interval.

From each point of each curve of\ \ $C$, we draw the Hugoniot$'$ that crosses\ \ $\mathcal{C}_{_{s}}$ (and\ \ $\mathcal{C}_{_{f}}$). So, we define the saturated of a curve\ \ $C$, denoted as\ \ $SAT_C$, as:
\begin{equation}
SAT_C=\left\{\bigcup_{z\in\Theta}sh'(U)\quad \text{ such that }\quad U\in C\right\},
\end{equation}
here\ \ $sh'(U)$\ \ is the Hugoniot$'$ through\ \ $U$.

In Fig. \ref{wavesolutions}, we give an example of a sequence of waves, starting at state\ \ 
$U\in\mathcal{C}_{_{s}}$.  In Fig. \ref{wavesolutions}.{\it a}, the yellow curve is a rarefaction to the inflection curve. From the inflection there is a blue curve, that is the composite wave.  In Fig. \ref{wavesolutions}.{\it b}, in other view, we can see the shock curve from the same state\ \ $U$.

\begin{figure}[htpb]
	\begin{center}	\includegraphics[scale=0.5,width=0.45\linewidth]{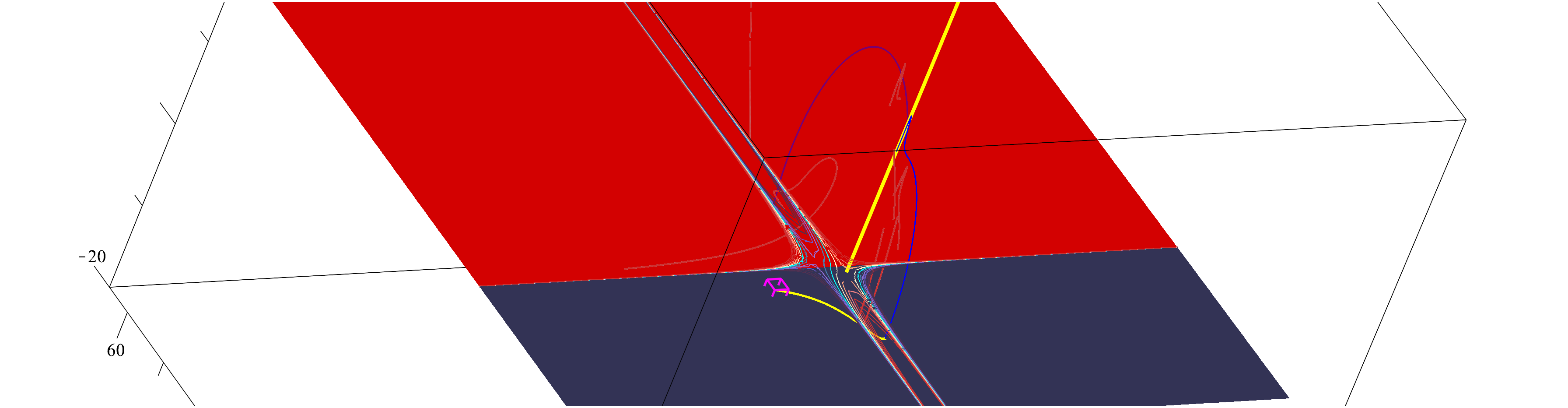}
\includegraphics[scale=0.5,width=0.45\linewidth]{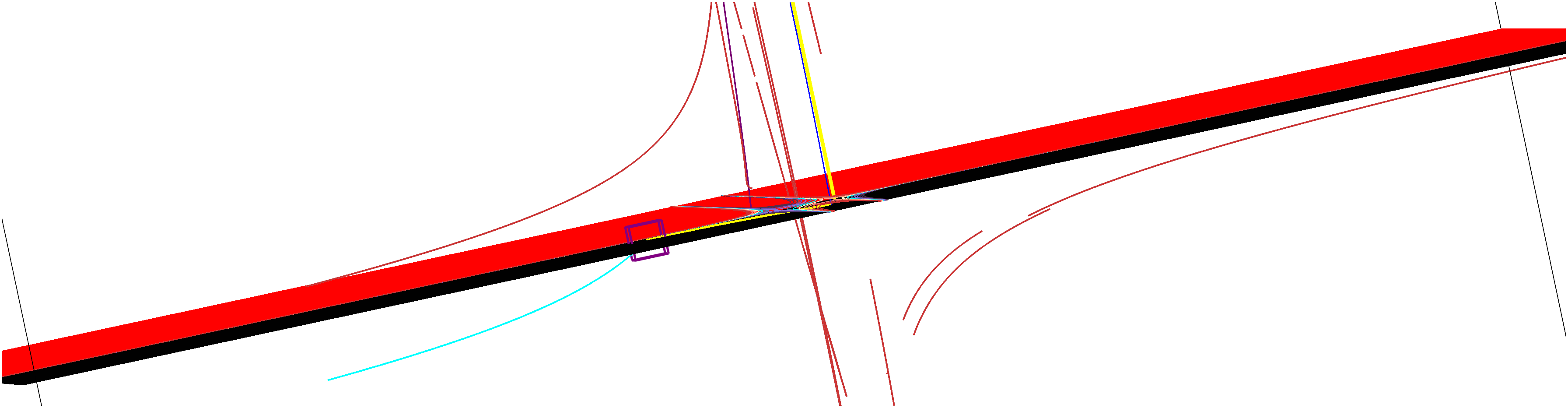}	
		\caption[]{\textit{ Left a:)}The wave sequence from a state\ \ $U\in\mathcal{C}_{_{s}}$\ \ the yellow curve represents a rarefaction and the blue is a composite wave.   \textit{ Right b:) From other view we can see the shock curve from\ \ $U$.} \label{wavesolutions}} 
	\end{center}
\end{figure}

To access\ \ $\mathcal{C}_{_{s}}$, we saturated each wave. The saturated of rarefaction is the brown surface drawn in Fig. \ref{saturated1}.{\it a}; the saturated of shock is the magenta surface  drawn in Fig. \ref{saturated1}.{\it b}; the saturated of composite is the green surface  drawn in Fig. \ref{saturated2}.{\it a}. In Fig. \ref{saturated2}.{\it b}, we drawn are saturated together.  

\begin{figure}[htpb]
	\begin{center}	\includegraphics[scale=0.5,width=0.45\linewidth]{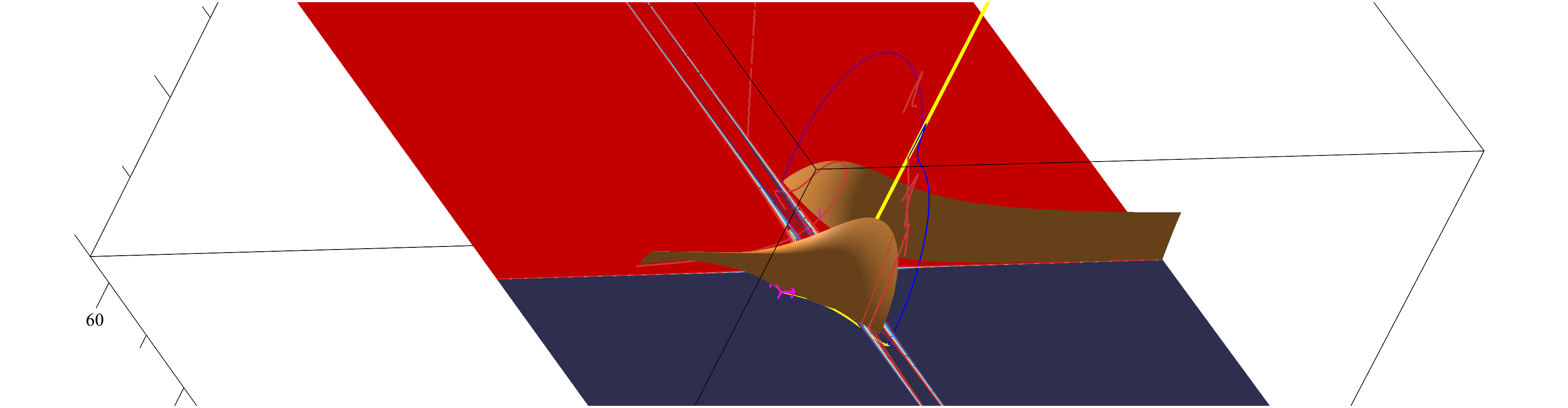}
\includegraphics[scale=0.5,width=0.45\linewidth]{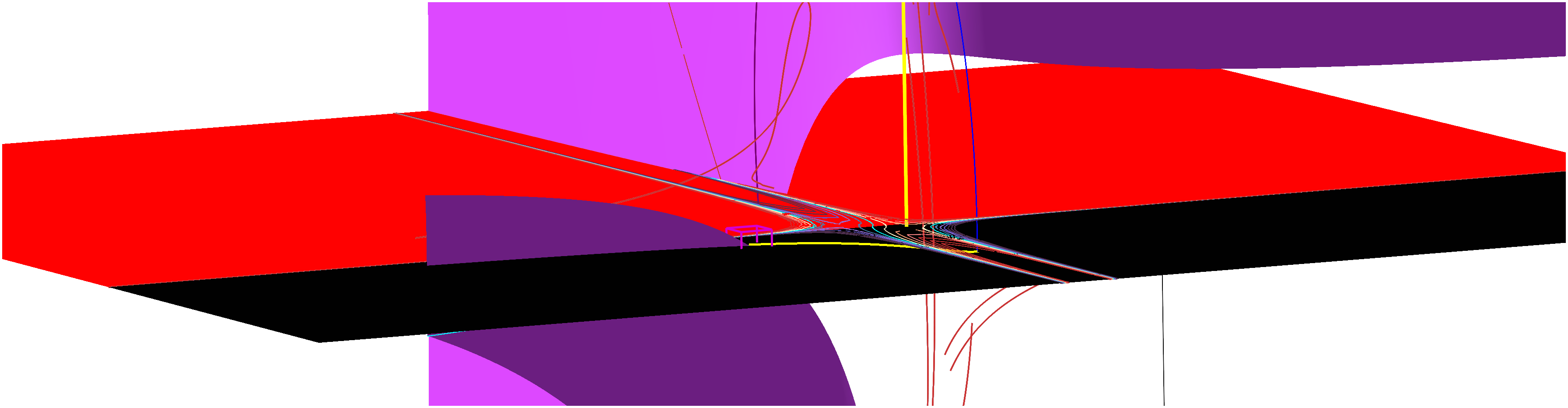}	
		\caption[]{\textit{ Left a:)} The saturated of rarefaction wave.   \textit{ Right b:) The saturated of shock wave.} \label{saturated1}} 
	\end{center}
\end{figure}

\begin{figure}[htpb]
	\begin{center}	\includegraphics[scale=0.5,width=0.45\linewidth]{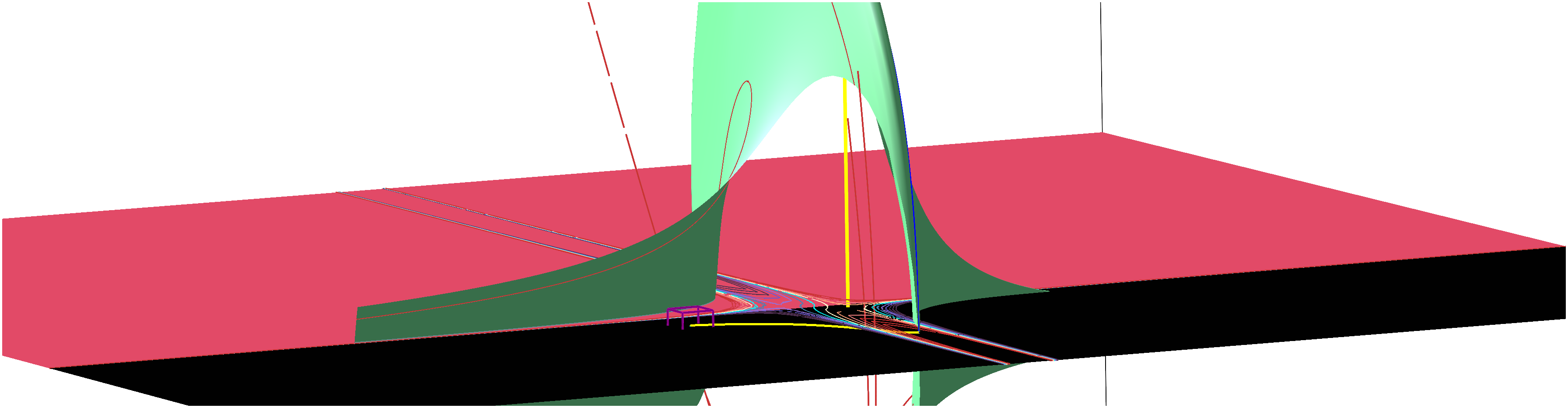}
\includegraphics[scale=0.5,width=0.45\linewidth]{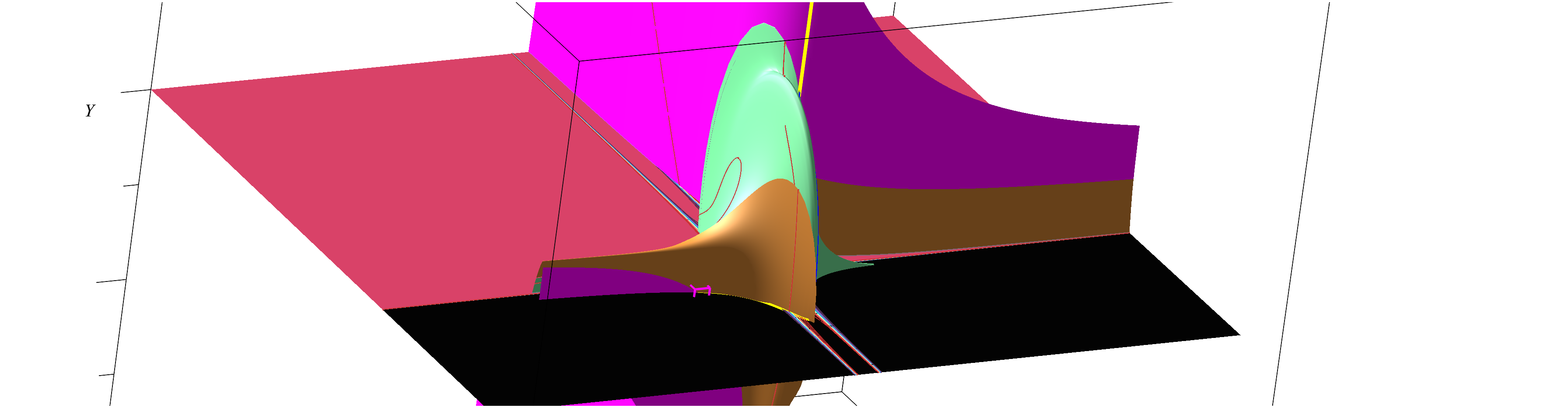}	
		\caption[]{\textit{ Left a:)}The saturated of composite wave.   \textit{ Right b:) All saturated put together.} \label{saturated2}} 
	\end{center}
\end{figure}

\subsection{2-reverse wave sequence}
\label{2reverse}

To obtain the Riemann solution we need to construct the wave in the backward direction. 
We define the backward direction through 2-reverse shock. We will show that this 2-reverse shock, in following the direction from left to right is a Lax's 2-shock, satisfying conditions (\ref{lax2}), so, we define the rarefaction (and composite waves) in the opposite direction.

For a state\ \ $U\in\mathcal{C}_{_{f}}$, we define the 2-reverse shock as the arc of Hugoniot for which the shock speed increases along of this arc. The rarefaction is constructed in the opposite direction. Following this construction and remembering that Fig. \ref{direct1}, we have that the reverse rarefaction is drawn in the direction of\ \ $\mathcal{I}_{_{f}}$\ \ and the Hugoniot is drawn in the opposite direction as in Fig. \ref{reverse1}.

\begin{figure}[htpb]
	\begin{center}
\includegraphics[scale=0.5,width=0.6\linewidth]{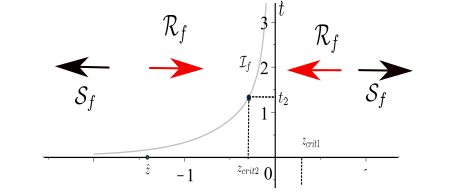}	
		\caption{Direction of reverse waves in\ \ $\mathcal{C}_{_{f}}$.} \label{reverse1}
	\end{center}
\end{figure}
\noindent
We prove now that the  reverse 2-shock satisfies the Lax's condition (\ref{lax2}).

First, we consider\ \ $\tilde{U}\in \mathcal{C}_{_{f}}$. From\ \ $\tilde{U}$\ \ we draw the Hugoniot, that we denote as\ \ $sh(\tilde{U})$\ \ and consider the arc for which shock speed increases as\ \ $sh^+(U)$. Let\ \ $U\in sh^+(\tilde{U})$\ \ and close to\ \ $\tilde{U}$. Notice that since we are constructing the reverse 2-shock, we have that for the direct wave\ \ $U\in sh^+(\tilde{U})$\ \ represents the ``left'' state and\ \ $\tilde{U}$\ \ the ``right''. Since the shock speed increasing\ \ $s$\ \ in the reverse 2-shock, follows that
\begin{equation}
s(U)>s(\tilde{U})=\lambda_{_{f}}(U),\label{ddsd}
\end{equation}
that is the first inequality in (\ref{lax2}).

To prove the other inequality of (\ref{lax2}), we use the following assumption proved in \cite{AEMP10}:

\begin{assump}\label{assump:01} Let\ \ $\mathcal{V}\in\mathcal{W}$\ \ be a neighborhood of a point\ \ $\mathcal{U}\in\mathcal{C}$\ \ away from the inflection and the coincidence loci. Let\ \ $\mathcal{V}_{_{1}}$\ \ be one of the two connected components of\ \ $\mathcal{V}-\mathcal{C}$. In this paper we assume that if\ \ $s$\ \ decreases along the Hugoniot curve through\ \ $\mathcal{U}$\ \ into\ \ $\mathcal{V}_{_{1}}$, then\ \ $s$\ \ increases along the Hugoniot$'$ curve through\ \ $\mathcal{U}$\ \ into\ \ $\mathcal{V}_{_{1}}$. A similar statement is assumed to hold if\ \ $s$\ \ increases along the Hugoniot curve into\ \ $\mathcal{V}_{_{1}}$.
\end{assump}

Now, to prove the second inequality in (\ref{lax2}), we need to obtain the projection of\ \ $U\in sh^+(\tilde{U})$\ \ into\ \ $\mathcal{C}_{_{s}}$\ \ to obtain the values of\ \ $\lambda_{_{s}}$\ \ evaluated in these projetions. To do, we define\ \ ${U}_{_{s}}$,\ \ ${U}_{_{f}}$,\ \ $ {U'}_{_{s}}$\ \ and\ \  ${U'}_{_{f}}$\ \ to be the points satisfying\ \ $(\ref{usuli})$. Notice that to reach\ \ $\tilde{U}$\ \ the shock speed decreases, thus, from the Assumption \ref{assump:01},  the shock speed along of Hugoniot$'$ increases to reach\ \ $U_{_{f}}$, so, we know that
\begin{equation}
s(U)<s(U'_{_{f}})=\lambda_{_{f}}(U'_{_{f}})\label{dds}
\end{equation}
Finally, to reaches the projection\ \ $U'_{_{s}}$, we follows the opposite direction of the arc of Hugoniot$'$ that reaches\ \ $U'_{_{f}}$, thus the shock speed decreases along of this curve, so, we have that:
\begin{equation}
\lambda_{_{s}}(U')=s(U')<s(U).\label{dds2}
\end{equation}
Using (\ref{ddsd}), (\ref{dds}) and\ \ $(\ref{dds2})$, we have that the shock satisfies the Lax's condition (\ref{lax2}) and it is a Lax's 2-shock and the proof is completed.

\begin{remark} In \cite{AEMP10}, authors proved that the shock satisfies the Lax's conditions, \emph{i.e.}, the shock is either  1-shock of Lax or a 2-shock of Lax. However, in the paper \cite{AEMP10},  to prove that the shock from for a state in\ \ $C_{_{s}}$\ \ is a 1-shock, authors consider\ \ $U\in\mathcal{C}_{_{s}}$, as a abuse in notation they consider another state that belongs to 1-shock curve that they also call as\ \ $\mathcal{U}$. Despite the abuse of notation, they correctly proved that the shock is a 1-shock.
A similar abuse of notation appears to consider that the 2-reverse shock is satisfies the Lax's condition (\ref{lax2}). In the present paper, we corrected these mistakes and we gave details of the proof that the 2-reverse shock satisfies the Lax's conditions (\ref{lax2}).
\end{remark}

In the following examples we show the Riemann solution in the wave manifold\ \ $\mathcal{W}$\ \ and the corresponding solution in\ \ $uv$-plane.
\subsection{Examples of solutions}

\noindent
\textbf{Example 1.}

First, we consider a state\ \ $W_{_{L}}=(u_{_{L}},v_{_{L}})$\ \ for which the state\ \ $\mathcal{U}_{_{L}}$ (obtained by Algorithm RS) belongs to\ \  $I$.

Applying the Algorithm RS, we construct the waves of family-1 through\ \ $W_{_{L}}$. For this example, we consider,  see Fig \ref{figura1sol1}.{\it a}. 

In this example, we utilize\ \ $z_{_{L}}=-5$\ \ and\ \ $t_{_{L}}=-0.065$ (the corresponding value in the\ \ $uv$\ \ plane is\ \ $u_{_{L}}=-0.2430769231$\ \ and\ \ $v_{_{L}}=-0.6365384615$). The sequence  through\ \  $\mathcal{U}_{_{L}}$ is\ \ $\mathcal{H}_{_{1}}\mathcal{U}_{_{L}}\mathcal{R}_{_{1}}$\ \ and the curve stops at the coincidence curve.
Notice that in this sequence, the rarefaction curve\ \ $(\mathcal{R}_{_{1}})$\ \ crosses\ \ $\mathcal{C}_{_{s}}$\ \ at\ \ $t=0$,   see Figs. \ref{figura1sol3} (in wave manifold\ \ $\mathcal{W}$) and \ref{figura1sol1}.{\it a} in the\ \ $uv$-plane. 

The saturated curves are described also in Fig. \ref{figura1sol3}. The magenta surface is the saturated 
of\ \  $\mathcal{H}_{_{1}}$; the brown surface is the saturated of\ \ $\mathcal{R}_{_{1}}$\ \ and\ \ $\mathcal{R}_{_{2}}$; the green surface is the saturated of\ \ $\mathcal{C}_{_{2}}$.  The curve\ \ $\mathcal{C}_{_{2}}$\ \ is drawn to reach the double contact in\ \ $z=1/3$.

\begin{figure}[htpb]
	\begin{center}	\includegraphics[scale=0.8,width=0.47\linewidth]{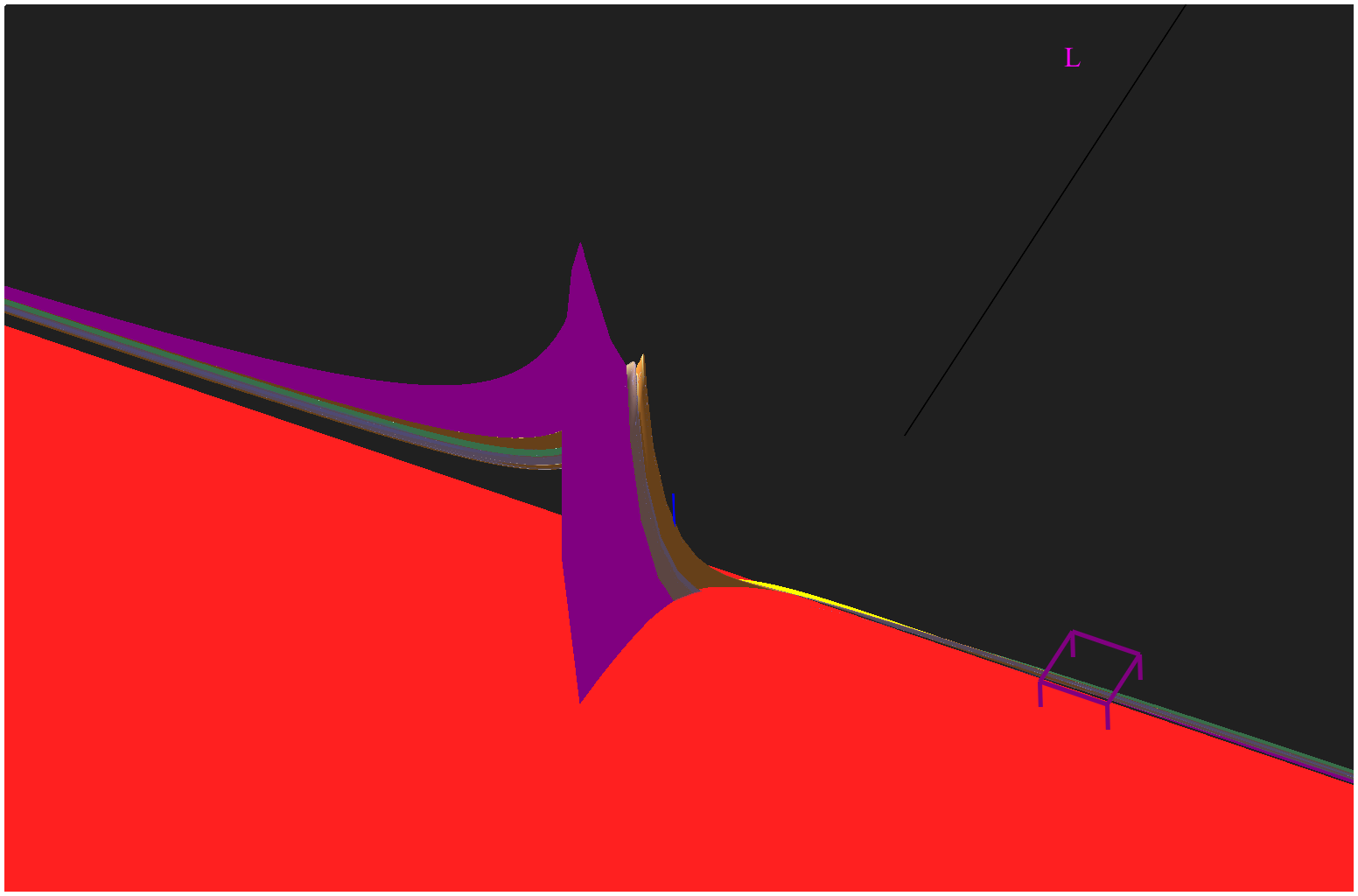}
\includegraphics[scale=0.8,width=0.47\linewidth]{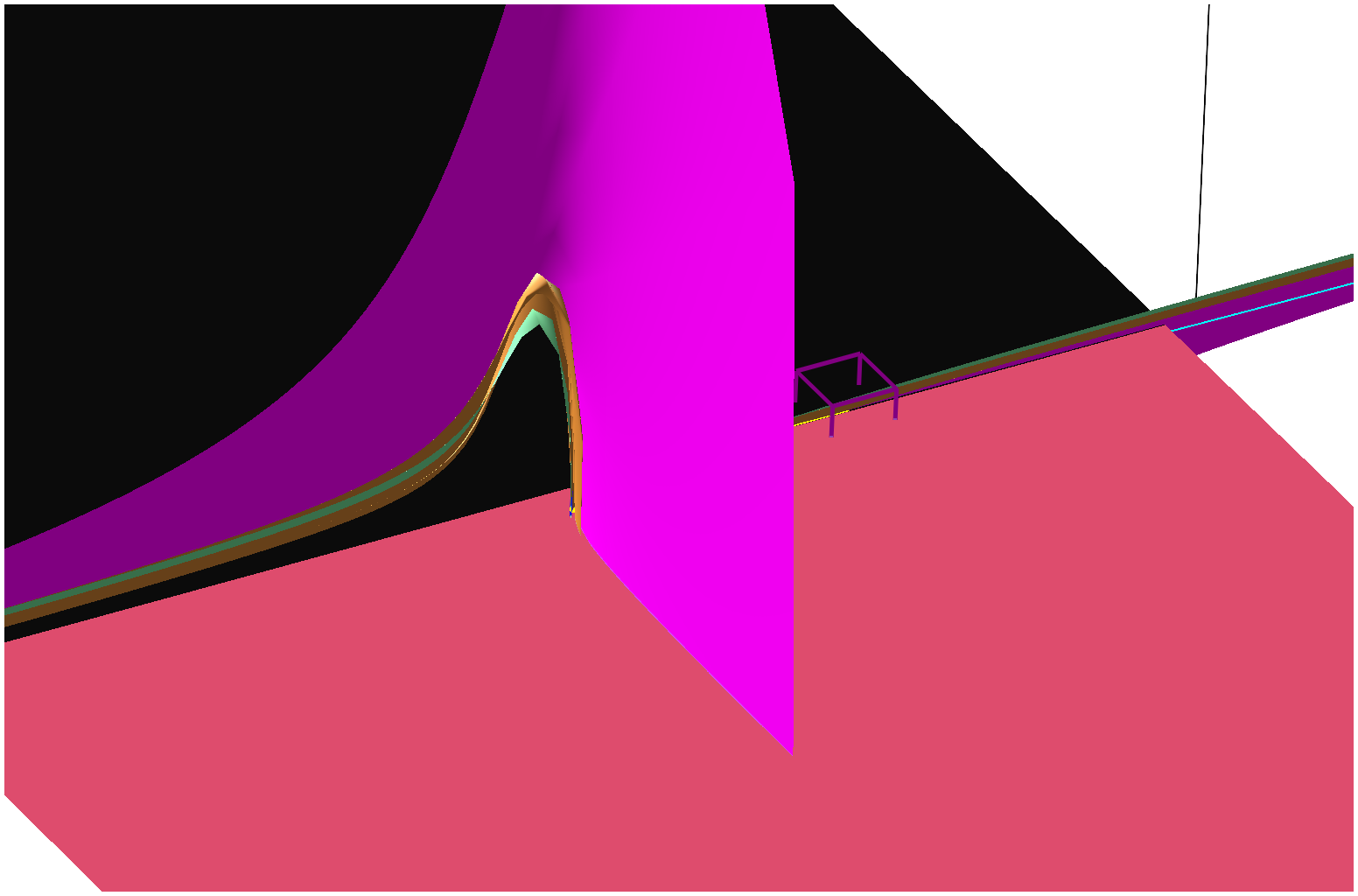}
		\caption[]{The wave sequence and the saturated surface from\ \ $\mathcal{U}_{_{L}}$\ \ in region\ \ $I$, here\ \ $z_{_{L}}=-5$\ \ and\ \ $t_{_{L}}=-0.065$. From\ \ $\mathcal{U}_{_{L}}$, denoted as a magenta box, we have the wave sequence\ \ $\mathcal{H}_{_{1}}\mathcal{U}_{_{L}}\mathcal{R}_{_{1}}$. The magenta surface is the saturated of\ \  $\mathcal{H}_{_{1}}$; the brown surface is the saturated of\ \ $\mathcal{R}_{_{1}}$\ \ and\ \ $\mathcal{R}_{_{2}}$; the green surface is the saturated of\ \ $\mathcal{C}_{_{2}}$.  The curve\ \ $\mathcal{C}_{_{2}}$\ \ is drawn to reach the double contact in\ \ $z=1/3$.
 \label{figura1sol3} } 
	\end{center}
\end{figure}

\begin{figure}[htpb]
	\begin{center}
		\includegraphics[scale=0.8,width=0.7\linewidth]{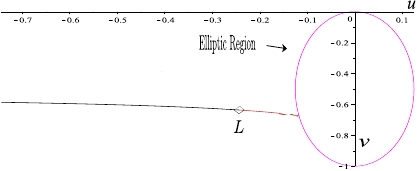}	
		\caption[]{ The wave sequence on the plane $uv$, here   $u_{_{L}}=-0.2430769231$ and $v_{_{L}}=-0.6365384615$. From $L$ there is a $\mathcal{H}_{_{1}}$ (black curve) and a $\mathcal{R}_{_{1}}$ (red curve). } \label{figura1sol1}
	\end{center}
\end{figure}

\newpage

\noindent
To construct the 2-reverse wave sequence, here we have two possibilities. 

The first possibility, we give\ \  $(z_{_{R}}=2,t_{_{R}}=2)\in\mathcal{C}_{_{f}}$ (the corresponding values on the plane\ \ $uv$\ \ are\ \ $(u_{_{R}}=0.8500000000, v_{_{R}}=3.200000000)$. From\ \
$\mathcal{U}_{_{r}}$ (obtained as described in Algorithm RS) we draw the 2-reverse wave sequence that is given by\ \ $\mathcal{H}_{_{2}}\mathcal{U}_{_{R}} \mathcal{R}_{_{2}}\mathcal{C}_{_{2}}$, as described in Fig. \ref{figura1sol2}. In the wave manifold\ \ $\mathcal{W}$,\  $\mathcal{R}_{_{2}}$\ \ is the green curve and\ \ $\mathcal{C}_{_{2}}$\ \ is the black curve, the\ \ $\mathcal{H}_{_{2}}$\ \ is below\ \ $\mathcal{C}_{_{f}}$. Notice that, from Fig. \ref{figura1sol2}.{\it a}, that\ \ $\mathcal{C}_{_{2}}$\ \ crosses the saturated of\ \ $\mathcal{H}_{_{1}}$ (the magenta surface) in a state in\ \ $\mathcal{W}$.

The Riemann solution from state\ \ $\mathcal{U}_{_{L}}$\ \ consists of a\ \ $\mathcal{H}_{_{1}}$\ \ from\ \ $\mathcal{U}_{_{L}}$\ \ to state\ \ $\mathcal{U}_{_{M}}$; from\ \ $\mathcal{U}_{_{M}}$\ \ there is a\ \
$\mathcal{C}_{_{2}}$ (a right characteristic shock) followed by a\ \ $\mathcal{R}_{_{2}}$\ \ to\ \ $\mathcal{U}_{_{R}}$. The solution in\ \ $\mathcal{W}$\ \ is described in Fig. \ref{figura1sol2}.{\it a} and in the\ \ $uv$\ \ plane is described in \ref{figura1sol2}.{\it b}.

\begin{figure}[htpb]
	\begin{center}
\includegraphics[scale=0.8,width=0.47\linewidth]{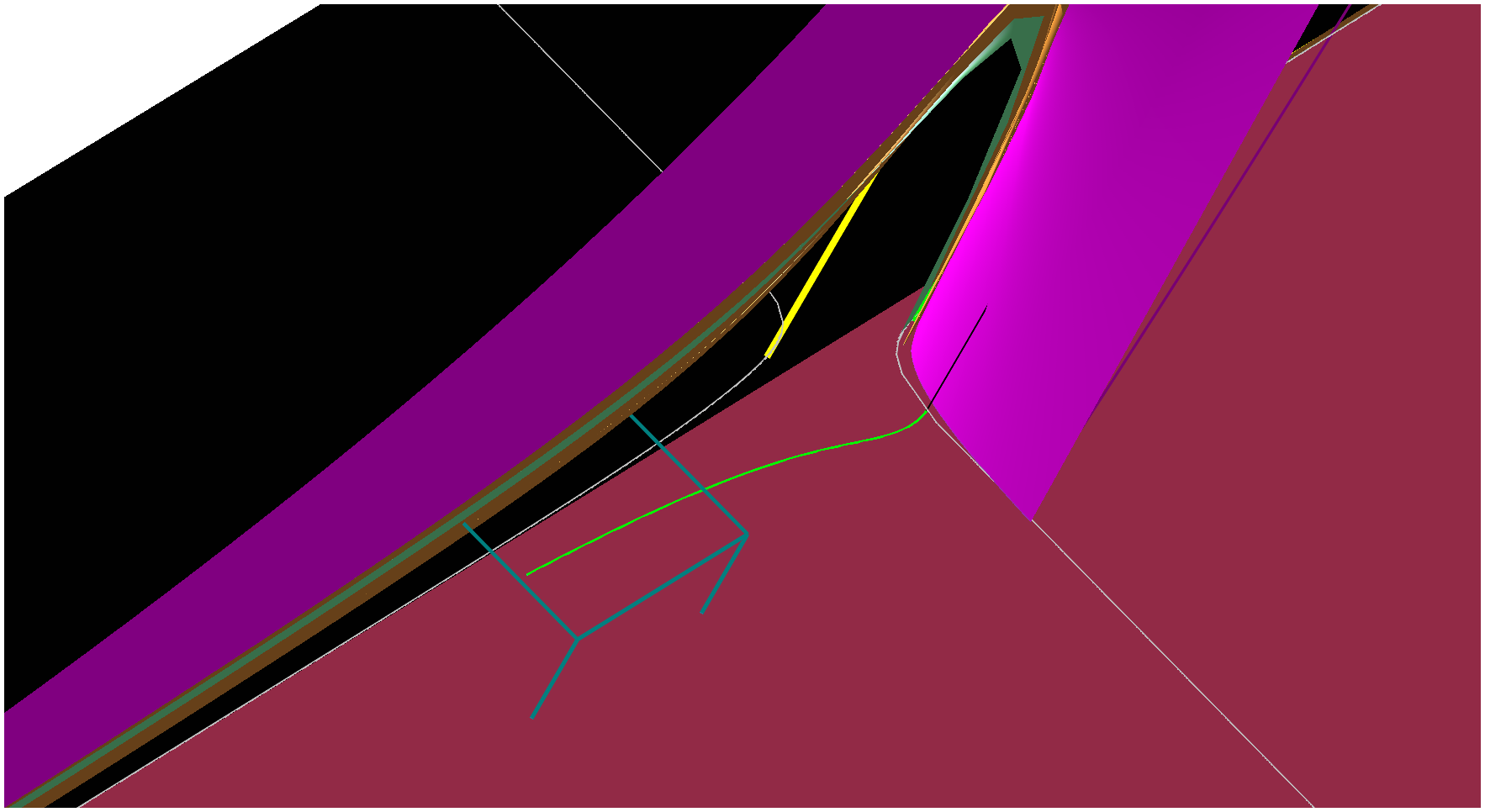}	\includegraphics[scale=0.8,width=0.47\linewidth]{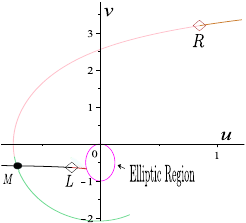}
		\caption[]{The Riemann solution from state\ \ $\mathcal{U}_{_{L}}$\ \ consists of a\ \ $\mathcal{H}_{_{1}}$\ \ from\ \ $\mathcal{U}_{_{L}}$\ \ to state\ \ $\mathcal{U}_{_{M}}$; from\ \ $\mathcal{U}_{_{M}}$\ \ there is a\ \ $\mathcal{C}_{_{2}}$ (a right characteristic shock) followed by a\ \ $\mathcal{R}_{_{2}}$\ \ to\ \ $\mathcal{U}_{_{R}}$. \textit{Left- a:)} The solution in the wave manifold\ \ $\mathcal{W}$. \textit{Right- b:)}The Riemann solution  in plane\ \ $uv$, for\ \ $L=(u_{_{L}}=-0.2430769231,v_{_{L}}=-0.6365384615)$\ \ and\ \ $R=(u_{_{R}}=-0.85, v_{_{R}}=3.2)$. \label{figura1sol2}} 
	\end{center}
\end{figure}

In the second  possibility for\ \ $\mathcal{U}_{_{R}}$, we give\ \  $(z_{_{R}}=0.385677655, t_{_{R}}=4.849940052)\in\mathcal{C}_{_{f}}$ (the corresponding values on the plane\ \ $uv$\ \ are\ \ $(u_{_{R}}=-0.6,v_{_{R}}=-1.1)$. From\ \ $\mathcal{U}_{_{r}}$\ \ (obtained as described in Algorithm RS) we draw the 2-reverse wave sequence that is given by\ \ $\mathcal{C}_{_{2}}\mathcal{R}_{_{2}}\mathcal{U}_{_{R}} \mathcal{H}_{_{2}}$, as described in Fig. \ref{figura1sol4}. In the wave manifold\ \ $\mathcal{W}$,\ $\mathcal{R}_{_{2}}$\ \ is the green curve, $\mathcal{H}_{_{2}}$\ \ is the blue curve and\ \ $\mathcal{C}_{_{2}}$\ \ is below\ \ $\mathcal{C}_{_{f}}$. Notice that, from Fig. \ref{figura1sol4}.{\it a}, that\ \
$\mathcal{R}_{_{2}}$\ \ crosses the saturated of\ \ $\mathcal{H}_{_{1}}$ (the magenta surface) in a state in\ \ $\mathcal{W}$.

The Riemann solution from state\ \ $\mathcal{U}_{_{L}}$\ \ consists of a\ \ $\mathcal{H}_{_{1}}$\ \ from\ \ $\mathcal{U}_{_{L}}$\ \ to state\ \ $\mathcal{U}_{_{M}}$; from\ \ $\mathcal{U}_{_{M}}$\ \ there is a\ \ $\mathcal{R}_{_{2}}$\ \ to\ \ $\mathcal{U}_{_{R}}$. The solution in\ \ $\mathcal{W}$\ \ is described in Fig. \ref{figura1sol4}.{\it a} and in the\ \ $uv$\ \ plane is described in \ref{figura1sol4}.{\it b}.

\begin{figure}[htpb]
	\begin{center}	
\includegraphics[scale=0.8,width=0.47\linewidth]{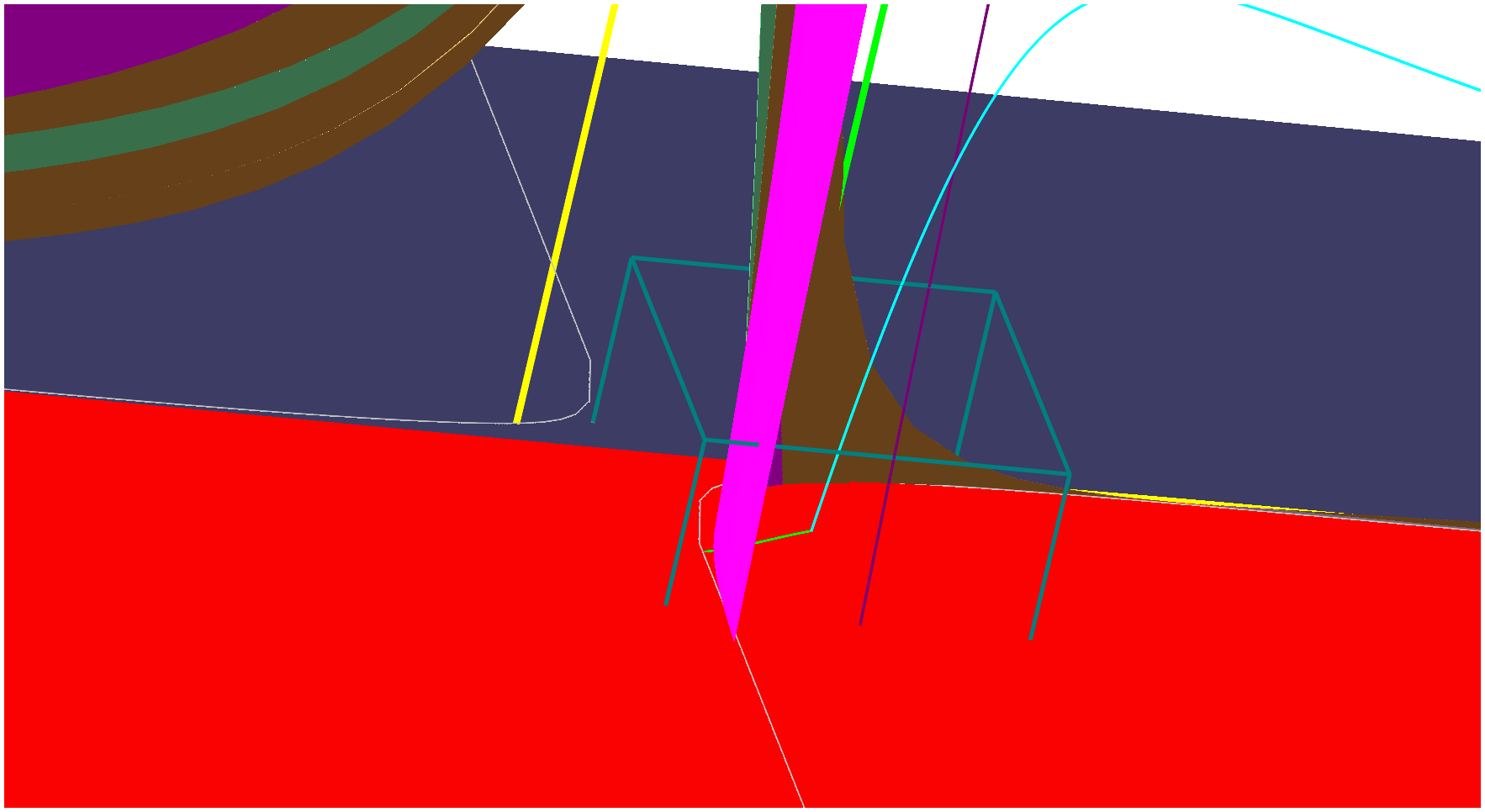}
\includegraphics[scale=0.8,width=0.47\linewidth]{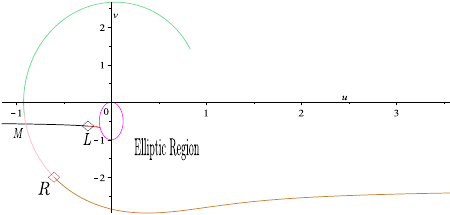}
		\caption[]{The 2-reverse wave sequence that is $ \mathcal{C}_{_{2}}\mathcal{R}_{_{2}}\mathcal{U}_{_{R}} \mathcal{H}_{_{2}}$. 
The Riemann solution from state $\mathcal{U}_{_{L}}$ consists of a $\mathcal{H}_{_{1}}$ from $\mathcal{U}_{_{L}}$ to state $\mathcal{U}_{_{M}}$; from $\mathcal{U}_{_{M}}$ there is a $\mathcal{R}_{_{2}}$ to $\mathcal{U}_{_{R}}$. \textit{Left- a:)} The solution in the wave manifold $\mathcal{W}$, here $\mathcal{R}_{_{2}}$ is the green curve, $\mathcal{H}_{_{2}}$ is the blue curve and $\mathcal{C}_{_{2}}$ is below $\mathcal{C}_{_{f}}$. \textit{Right- b:)}  The Riemann solution  in plane $uv$, for $L=(u_{_{L}}=-0.2430769231,v_{_{L}}=-0.6365384615)$ and $R=(u_{_{R}}=-0.6, v_{_{R}}=-1.1)$. \label{figura1sol4}}
	\end{center}
\end{figure}

\noindent
\textbf{Example 2.}

In this example we utilize\ \ $z_{_{L}}=1$\ \ and\ \ $t_{_{L}}=-1$ (the corresponding value in the\ \ $uv$ plane is\ \ $u_{_{L}}=0.125, v_{_{L}}=0.5 $), the state\ \ $(z_{_{L}},t_{_{L}})$\ \ belongs also to the region\ \ $I$. The sequence  through\ \  $\mathcal{U}_{_{L}}$ is\ \ 
$\mathcal{H}_{_{1}}\mathcal{U}_{_{L}}\mathcal{R}_{_{1}}$, there is no composite wave because the rarefaction does not reach the inflection curve. However, notice that, the rarefaction curve\ \
$(\mathcal{R}_{_{1}})$\ \ crosses\ \ $\mathcal{C}_{_{s}}$\ \  at\ \ $t=0$\ \ and the wave curve stops,   see Figs. \ref{figura1ex2-1} (in wave manifold\ \ $\mathcal{W}$) and \ref{figura1ex2-2} in the\ \ $uv$-plane. 

The saturated curves are described also in Fig. \ref{figura1ex2-2}. The magenta surface is the saturated of\ \  $\mathcal{H}_{_{1}}$; the brown surface is the saturated of\ \ $\mathcal{R}_{_{1}}$\ \ and\ \ $\mathcal{R}_{_{2}}$.

\begin{figure}[htpb]
	\begin{center}	\includegraphics[scale=0.8,width=0.47\linewidth]{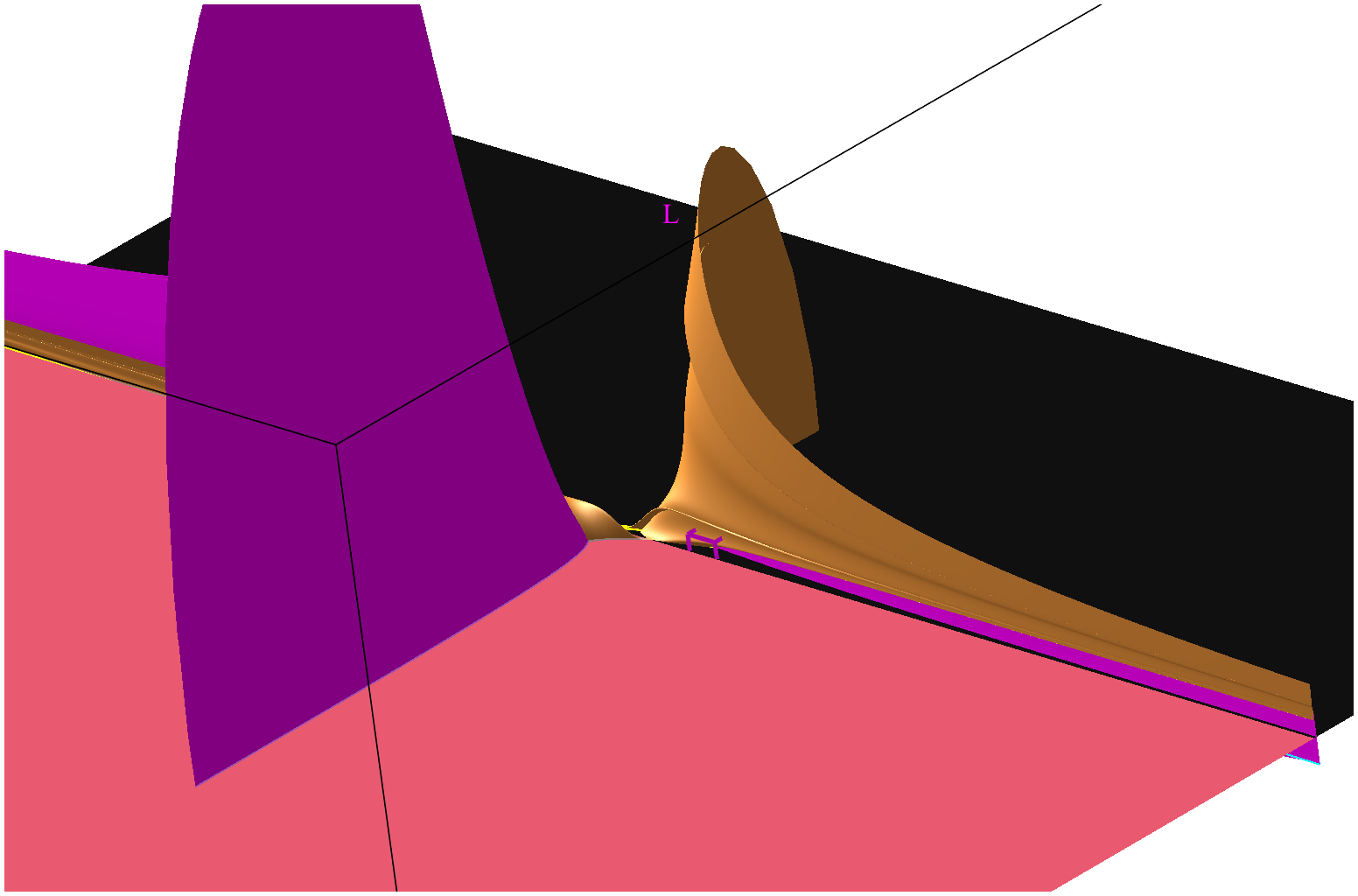}
\includegraphics[scale=0.8,width=0.47\linewidth]{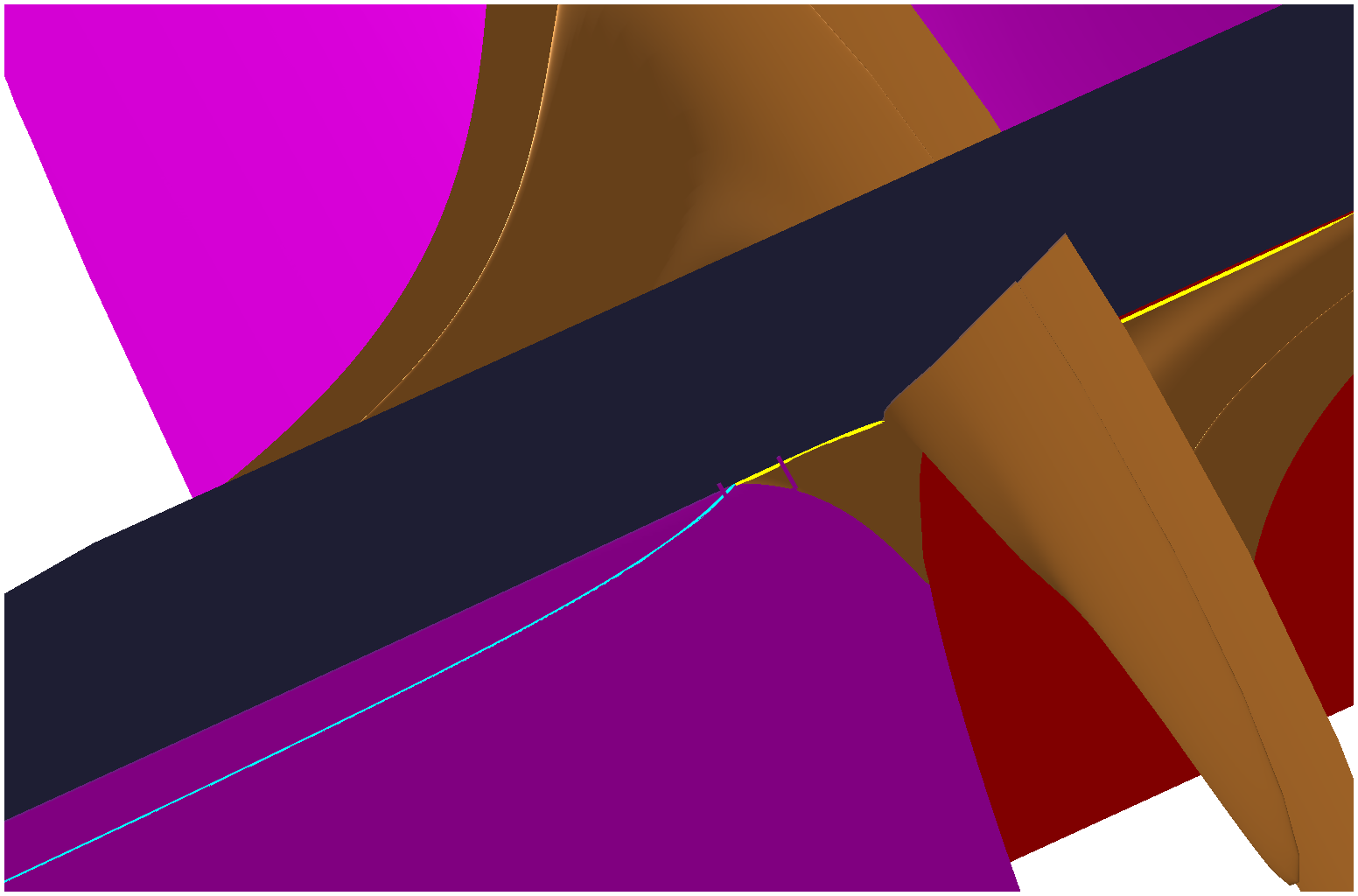}
		\caption[]{The wave sequence and the saturated surface from\ \ $\mathcal{U}_{_{L}}$\ \ in region\ \ $II$, here\ \ $z_{_{L}}=1$\ \ and\ \ $t_{_{L}}=-1$. From\ \ $\mathcal{U}_{_{L}}$, denoted as a magenta box, we have the wave sequence\ \ $\mathcal{H}_{_{1}}\mathcal{U}_{_{L}}\mathcal{R}_{_{1}}\mathcal{R}_{_{2}}$. The magenta surface is the saturated of\ \ $\mathcal{H}_{_{1}}$\ \ and the brown surface is the saturated of\ \ $\mathcal{R}_{_{1}}$\ \ and\ \ $\mathcal{R}_{_{2}}$; there is no composite wave, because the rarefaction does not reach the inflection curve. Here, we have different views of the\ \ $\mathcal{U}_{_{L}}$. \textit{Left- a:)} The view for\ \ $Y>0$. \textit{Right- b:)} The view with\ \ $Y<0$\ \ the blue curve represents\ \ $\mathcal{H}_{_{1}}$\ \ and yellow is\ \ $\mathcal{R}_{_{1}}$. \label{figura1ex2-1}} 
	\end{center}
\end{figure}

\begin{figure}[htpb]
	\begin{center}
		\includegraphics[scale=0.8,width=0.75\linewidth]{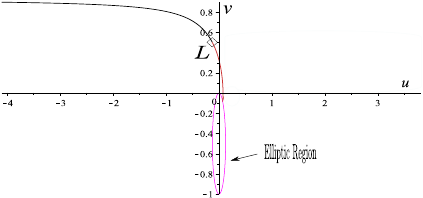}	
		\caption[]{The wave sequence on the plane\ \ $uv$, here\ \ $u_{_{L}}=0.125$\ \ and\ \ $v_{_{L}}=0.5$. From\ \ $L$\ \ there is a\ \ $\mathcal{H}_{_{1}}$ (black curve) and a\ \ $\mathcal{R}_{_{1}}$ (red curve). The magenta ellipse is the elliptic region.\label{figura1ex2-2}} 
	\end{center}
\end{figure}

To construct the 2-reverse wave sequence, we consider\ \ $\mathcal{U}_{_{r}}$ (obtained as described in Algorithm RS) and we draw the 2-reverse wave sequence that is given, in this example,  by\ \ $\mathcal{C}_{_{2}}\mathcal{R}_{_{2}}\mathcal{U}_{_{R}} \mathcal{H}_{_{2}}$, as described in Fig. \ref{figura1sol2}. In the wave manifold\ \ $\mathcal{W}$,\  $\mathcal{R}_{_{2}}$\ \ is the green curve on the plane\ \ $\mathcal{C}_{_{f}}$\ \ and\ \ $\mathcal{C}_{_{2}}$\ \ is the black curve for\ \ $Y<0$ (below\ \ $\mathcal{C}_{_{f}}$, the\ \ $\mathcal{H}_{_{2}}$\ \ is the blue curve for\ \ $Y<0$ (above\ \ $\mathcal{C}_{_{f}}$). Notice that, from Fig. \ref{figura1ex2-3}.{\it a}, that\ \ $\mathcal{C}_{_{2}}$\ \ crosses the saturated of\ \ $\mathcal{H}_{_{1}}$ (the magenta surface) in a state in\ \ $\mathcal{W}$\ \ for\ \ $Y<0$ (below\ \ $\mathcal{C}_{_{f}}$).

The Riemann solution from state\ \ $\mathcal{U}_{_{L}}$\ \ consists of a\ \ $\mathcal{H}_{_{1}}$\ \ from\ \ $\mathcal{U}_{_{L}}$\ \ to state\ \ $\mathcal{U}_{_{M}}$; from\ \ $\mathcal{U}_{_{M}}$\ \ there is a\ \
$\mathcal{C}_{_{2}}$ (a right characteristic shock) followed by a\ \ $\mathcal{R}_{_{2}}$\ \ to\ \
$\mathcal{U}_{_{R}}$. The solution in\ \ $\mathcal{W}$\ \ is described in Fig. \ref{figura1ex2-3} and in the\ \ $uv$\ \ plane is described in  Fig. \ref{figura1ex2-4}.

\begin{figure}[htpb]
	\begin{center}	\includegraphics[scale=0.8,width=0.47\linewidth]{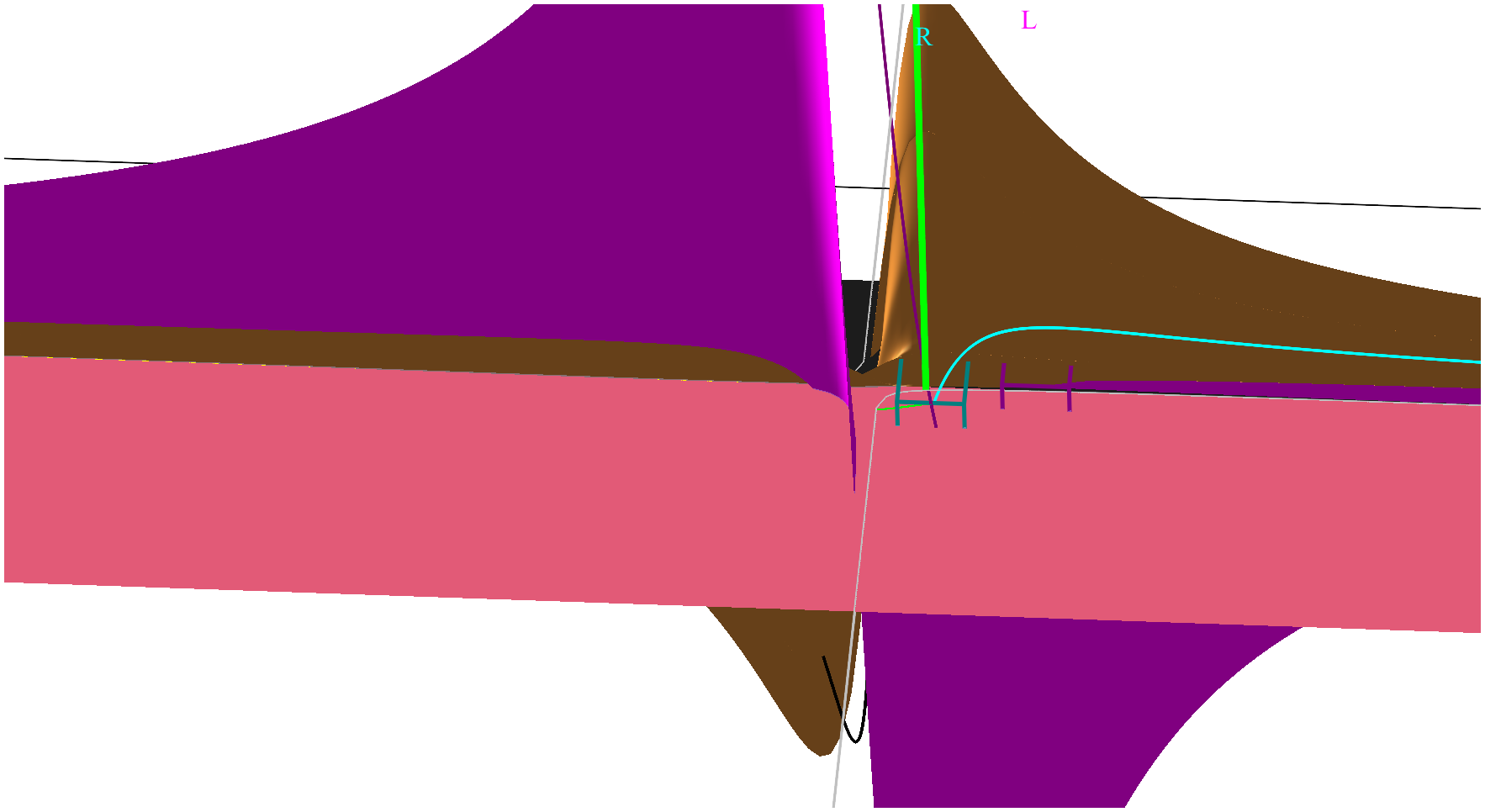}
\includegraphics[scale=0.8,width=0.47\linewidth]{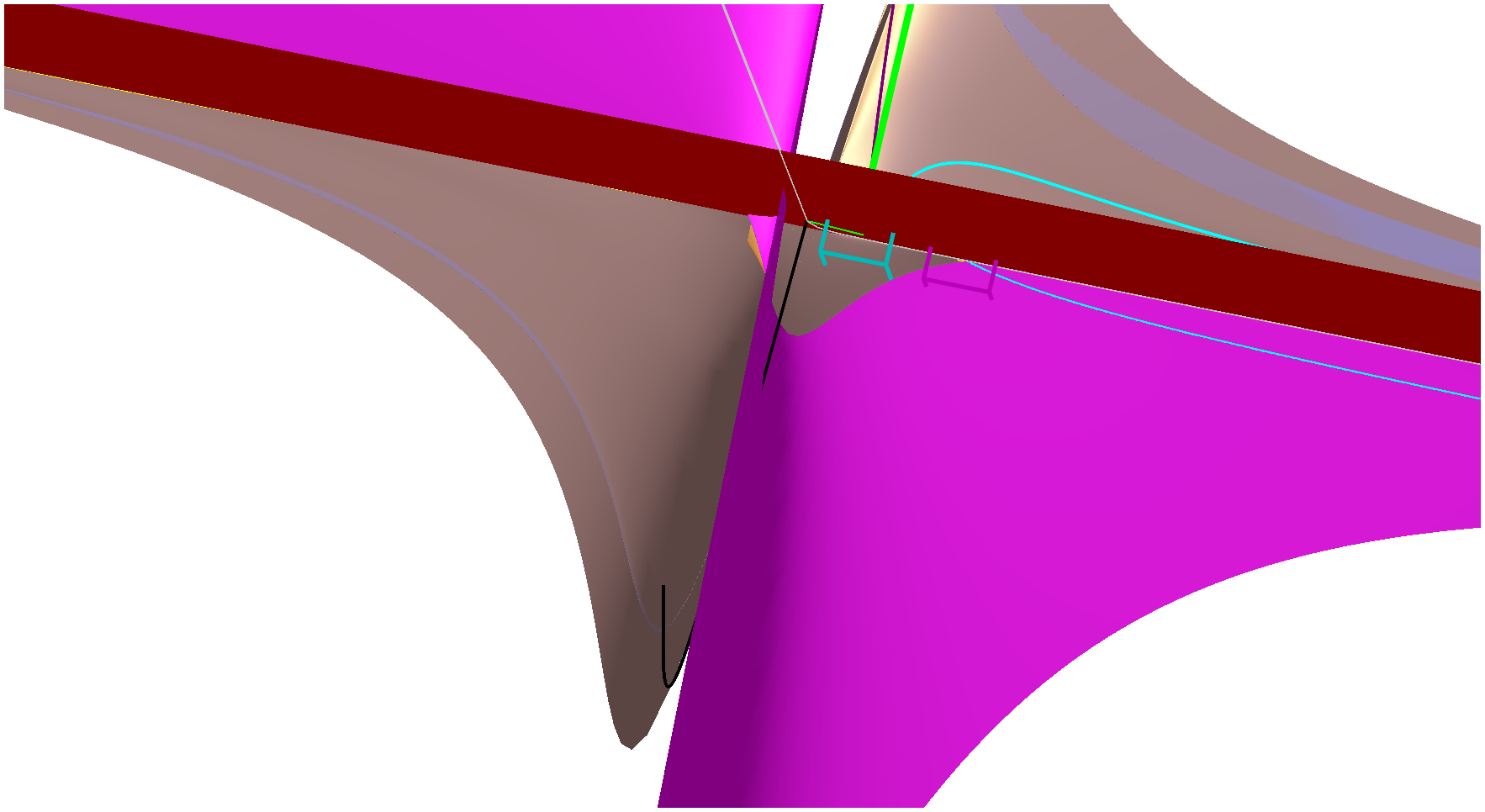}
		\caption[]{The 2-reverse wave sequence that is\ \  $\mathcal{C}_{_{2}}\mathcal{R}_{_{2}}\mathcal{U}_{_{R}} \mathcal{H}_{_{2}}$. The green curve on the plane\ \ $\mathcal{C}_{_{f}}$\ \ and\ \ $\mathcal{C}_{_{2}}$\ \ is the black curve for\ \ $Y<0$ (below\ \ $\mathcal{C}_{_{f}}$, the\ \ $\mathcal{H}_{_{2}}$\ \ is the blue curve for\ \ $Y<0$ (above\ \ $\mathcal{C}_{_{f}}$). Notice that\ \ $\mathcal{C}_{_{2}}$\ \ crosses the saturated of\ \ $\mathcal{H}_{_{1}}$ (the magenta surface) in a state in\ \ $\mathcal{W}$\ \ for\ \ $Y<0$ (below\ \ $\mathcal{C}_{_{f}}$).\label{figura1ex2-3}} 
	\end{center}
\end{figure}

\begin{figure}[htpb]
	\begin{center}
	\includegraphics[scale=0.8,width=0.75\linewidth]{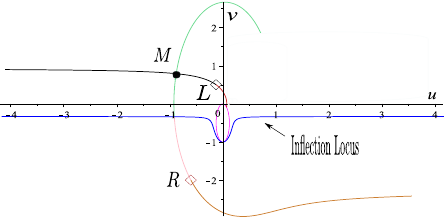}
		\caption[]{The Riemann solution  in plane\ \ $uv$, for\ \ $L=(u_{_{L}}=-0.125, v_{_{L}}=3.5)$\ \ and\ \ $R=(u_{_{R}}=-0.75, v_{_{R}}=-2.5)$.  From state\ \ $L$\ \ consists of a\ \ $\mathcal{H}_{_{1}}$\ \ from\ \ $L$\ \ to state\ \ $M$; from\ \ $M$\ \ there is a\ \ $\mathcal{C}_{_{2}}$ (a right characteristic shock) followed by a\ \ $\mathcal{R}_{_{2}}$\ \ to\ \ $R$. \label{figura1ex2-4}} 
	\end{center}
\end{figure}

\noindent
\textbf{Example 3.}

In this example, we utilize\ \ $z_{_{L}}=-1$\ \ and\ \ $t_{_{L}}=-4$ (the corresponding value in the\ \ $uv$\ \ plane is\ \ $u_{_{L}}=-0.125, v_{_{L}}=3.5 $), the state\ \ $(z_{_{L}},t_{_{L}})$\ \ belongs to the region\ \ $II$. The sequence  through\ \  $\mathcal{U}_{_{L}}$\ \ is\ \ 
$\mathcal{H}_{_{1}}\mathcal{U}_{_{L}}\mathcal{R}_{_{1}}\mathcal{C}_{_{1}}$, see Figs. \ref{figura1ex3-1} (in wave manifold\ \ $\mathcal{W}$) and \ref{figura1ex2-2} in the\ \ $uv$-plane. 

In Fig. \ref{figura1ex3-1}, the magenta surface is the saturated of\ \  $\mathcal{H}_{_{1}}$; the brown surface is the saturated of\ \ $\mathcal{R}_{_{1}}$\ \ and\ \ $\mathcal{R}_{_{2}}$; the green surface is the saturated of\ \ $\mathcal{C}_{_{2}}$. The curve\ \ $\mathcal{C}_{_{2}}$\ \ is drawn to reach the double contact in\ \ $z=-1/3$. In Fig. \ref{figura1ex3-2}, the double contact\ \ $z=-1/3$ is the yellow straight line.

\begin{figure}[htpb]
	\begin{center}	\includegraphics[scale=0.8,width=0.47\linewidth]{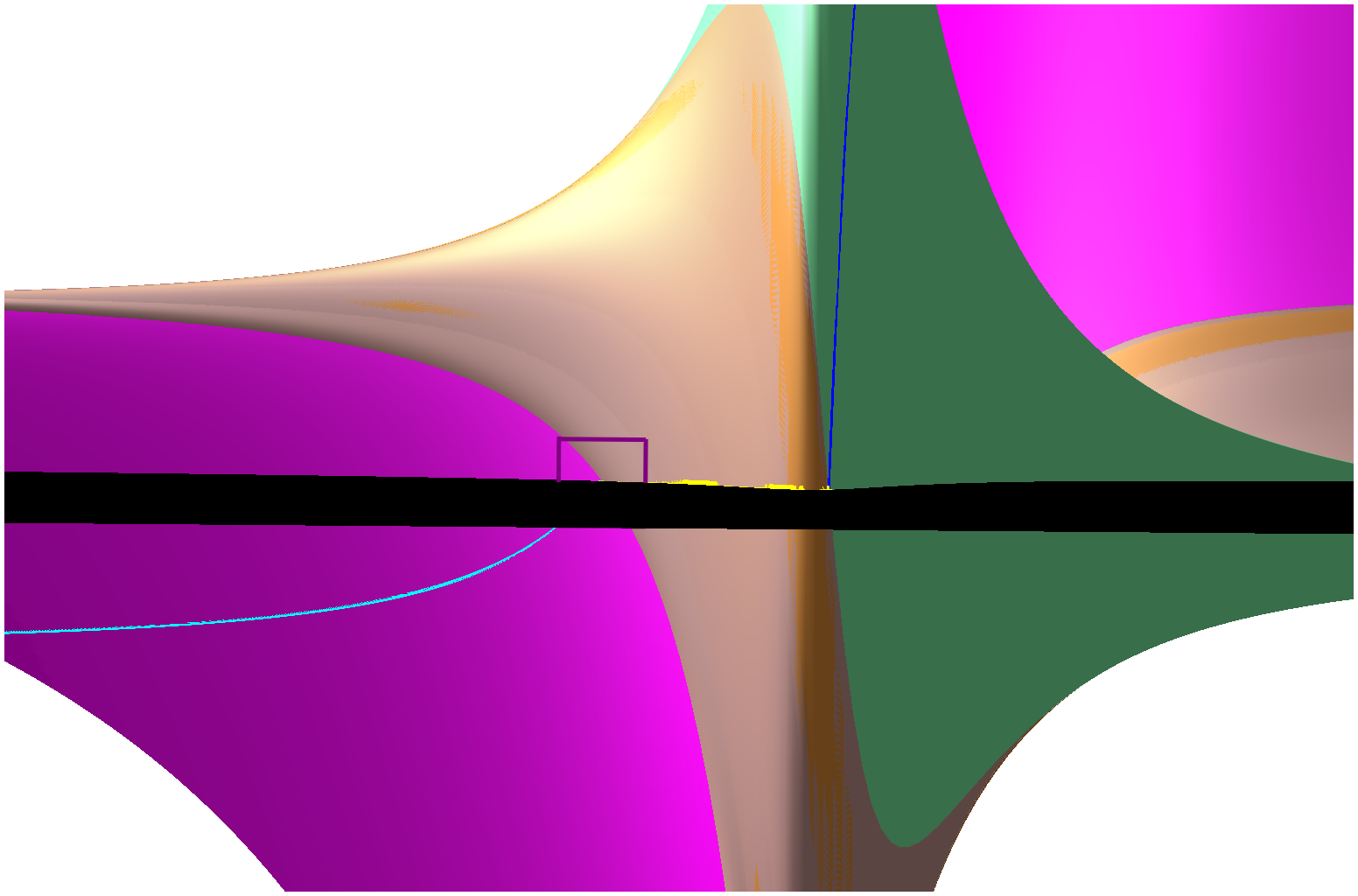}
\includegraphics[scale=0.8,width=0.47\linewidth]{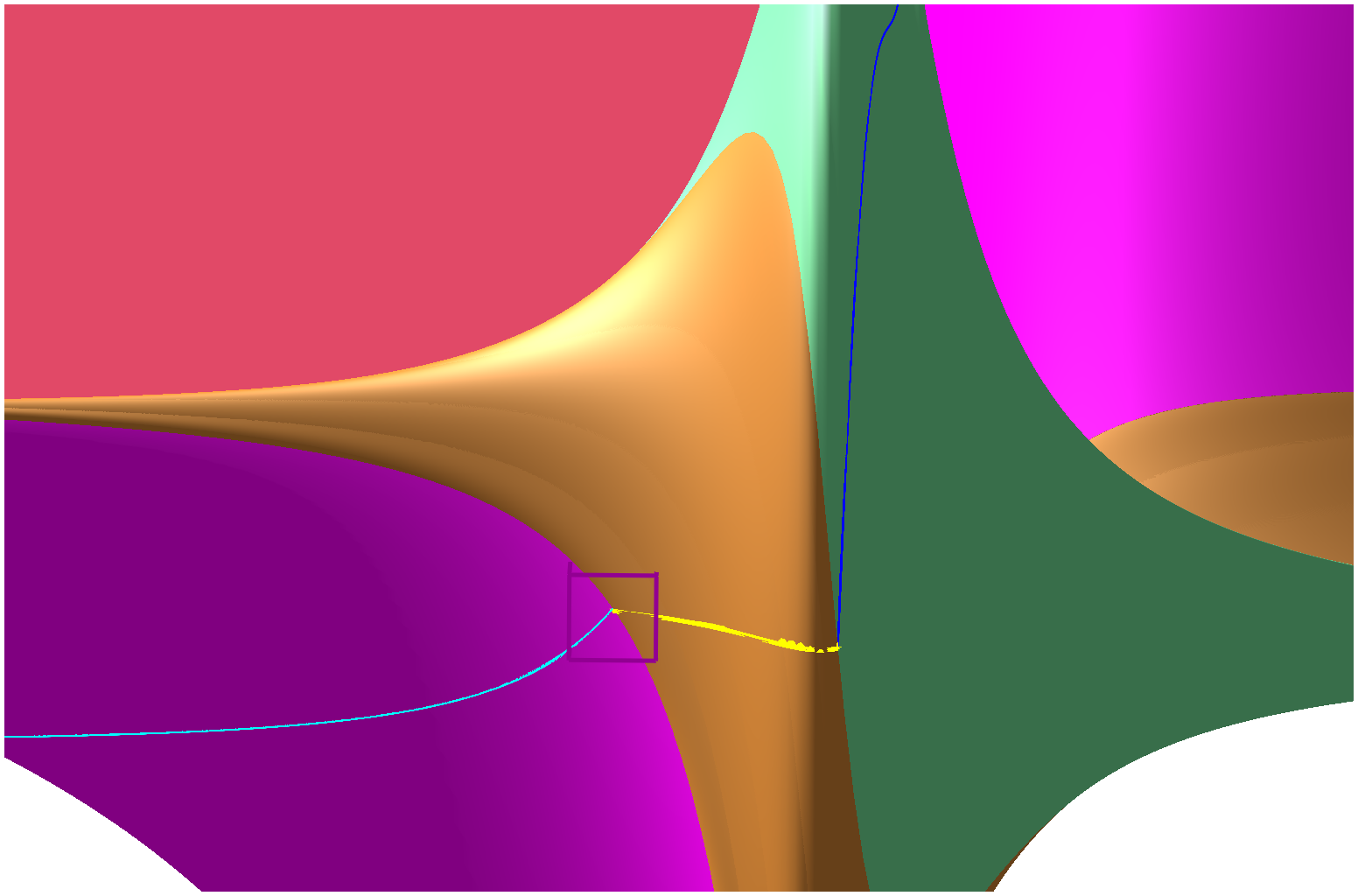}
		\caption[]{Here the values for\ \ $\mathcal{U}_{_{L}}$\ \ in\ \ $III$\ \ with\ \ $z_{_{L}}=-1$\ \ and\ \ $t_{_{L}}=-4$. The sequence  through\ \  $\mathcal{U}_{_{L}}$\ \ is\ \ 
$\mathcal{H}_{_{1}}\mathcal{U}_{_{L}}\mathcal{R}_{_{1}}\mathcal{C}_{_{1}}$. The magenta surface is the saturated of\ \  $\mathcal{H}_{_{1}}$ (blue curve); the brown surface is the saturated of\ \ $\mathcal{R}_{_{1}}$ (yellow curve); the green surface is the saturated of\ \ $\mathcal{C}_{_{2}}$ (cyan).  The curve\ \ $\mathcal{C}_{_{2}}$\ \ is drawn to reach the double contact in\ \ $z=-1/3$.\label{figura1ex3-1}} 
	\end{center}
\end{figure}

\begin{figure}[htpb]
	\begin{center}
		\includegraphics[scale=0.8,width=0.75\linewidth]{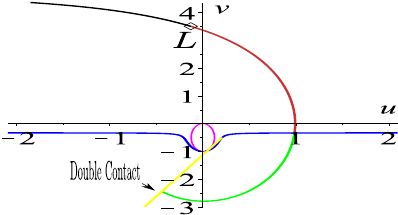}	
		\caption[]{The wave sequence on the plane\ \ $uv$, here\ \ $u_{_{L}}=-0.125$\ \ and\ \ $v_{_{L}}=3.5$. From\ \ $L$\ \ there is a\ \ $\mathcal{H}_{_{1}}$ (black curve) and a\ \ $\mathcal{R}_{_{1}}$ (red curve), since the rarefaction crosses the inflection (blue curve), there is a\ \ $\mathcal{C}_{_{2}}$\ \ to the double contact (yellow straight line)\ \ $z=-1/3$. The magenta ellipse is the elliptic region.\label{figura1ex3-2}} 
	\end{center}
\end{figure}

To construct the 2-reverse wave sequence, here we consider\ \  $(z_{_{R}}=5,t_{_{R}}=3)\in\mathcal{C}_{_{f}}$ (the corresponding values on the plane\ \ $uv$\ \ are\ \ $(u_{_{R}}=9.048076925, v_{_{R}}=14.03846154)$. From\ \ $\mathcal{U}_{_{r}}$ (obtained as described in Algorithm RS) we draw the 2-reverse wave sequence that is given, in this example,  by\ \ $\mathcal{H}_{_{2}}\mathcal{U}_{_{R}} \mathcal{R}_{_{2}}\mathcal{C}_{_{2}}$, as described in Fig. \ref{figura1ex3-3}. In the wave manifold\ \ $\mathcal{W}$,\  $\mathcal{R}_{_{2}}$\ \ is the green curve on\ \ $\mathcal{C}_{_{f}}$. The curve\ \ $H_{_{1}}$\ \ is drawn for\ \ $Y<0$, \emph{i.e.}, below\ \ $\mathcal{C}_{_{f}}$\ \ and\ \  $\mathcal{C}_{_{2}}$\ \ is the black curve behind the magenta surface (saturated of\ \ $\mathcal{H}_{_{1}}$), from\ \ $\mathcal{U}_{_{R}}$\ \ there is only a wave crossing the magenta surface, that is\ \ $\mathcal{R}_{_{2}}$.

In this case, the Riemann solution from state\ \ $\mathcal{U}_{_{L}}$\ \ consists of a\ \ $\mathcal{H}_{_{1}}$\ \ from\ \ $\mathcal{U}_{_{L}}$\ \ to state\ \ $\mathcal{U}_{_{M}}$; from\ \ $\mathcal{U}_{_{M}}$ there is a\ \ $\mathcal{R}_{_{2}}$\ \ to\ \ $\mathcal{U}_{_{R}}$. The solution in\ \ $\mathcal{W}$\ \ is described in Fig. \ref{figura1ex3-3} and in the\ \ $uv$\ \ plane is described in \ref{figura1ex3-4}.

\begin{figure}[htpb]
	\begin{center}	\includegraphics[scale=0.8,width=0.47\linewidth]{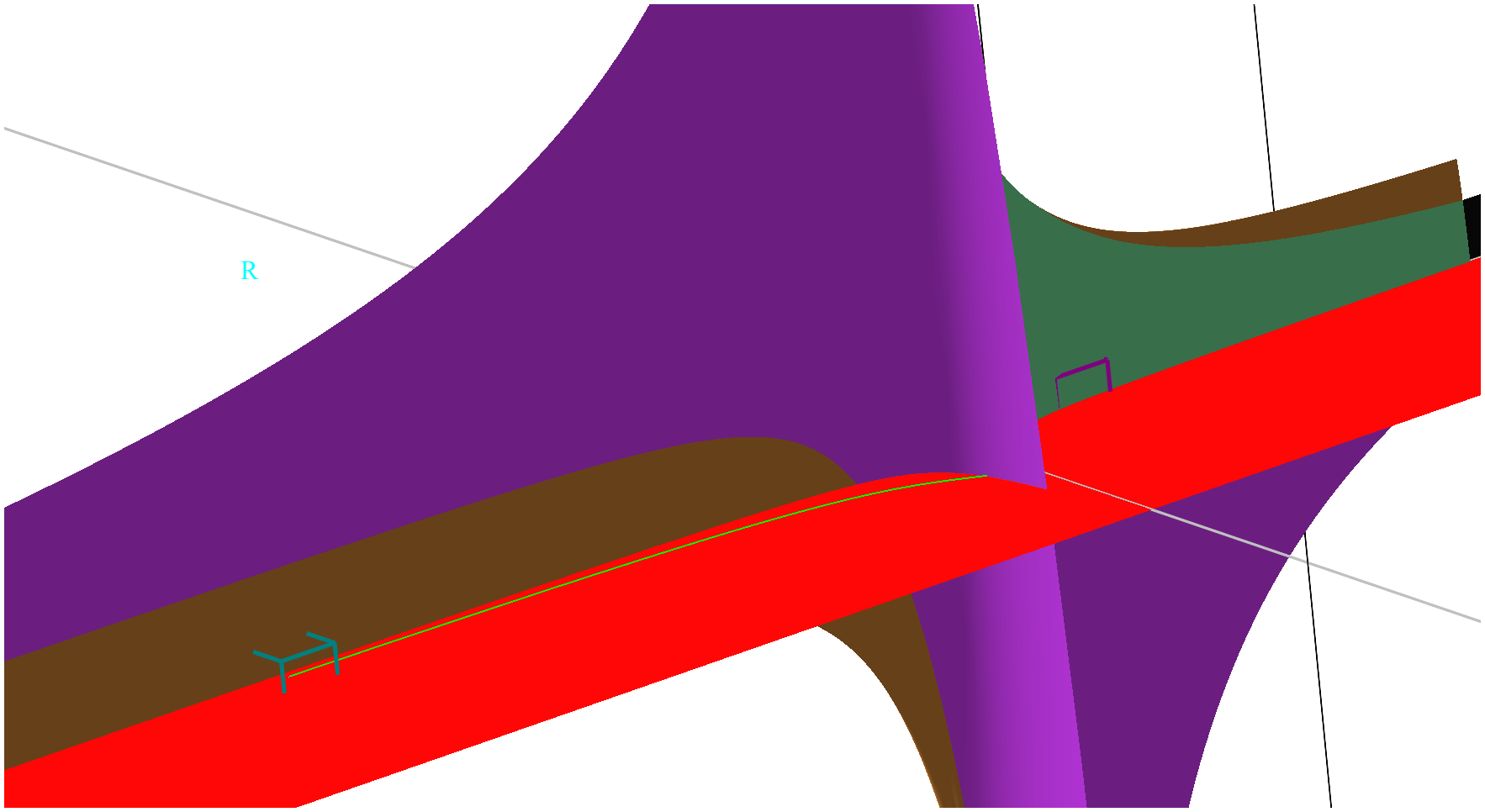}
\includegraphics[scale=0.8,width=0.47\linewidth]{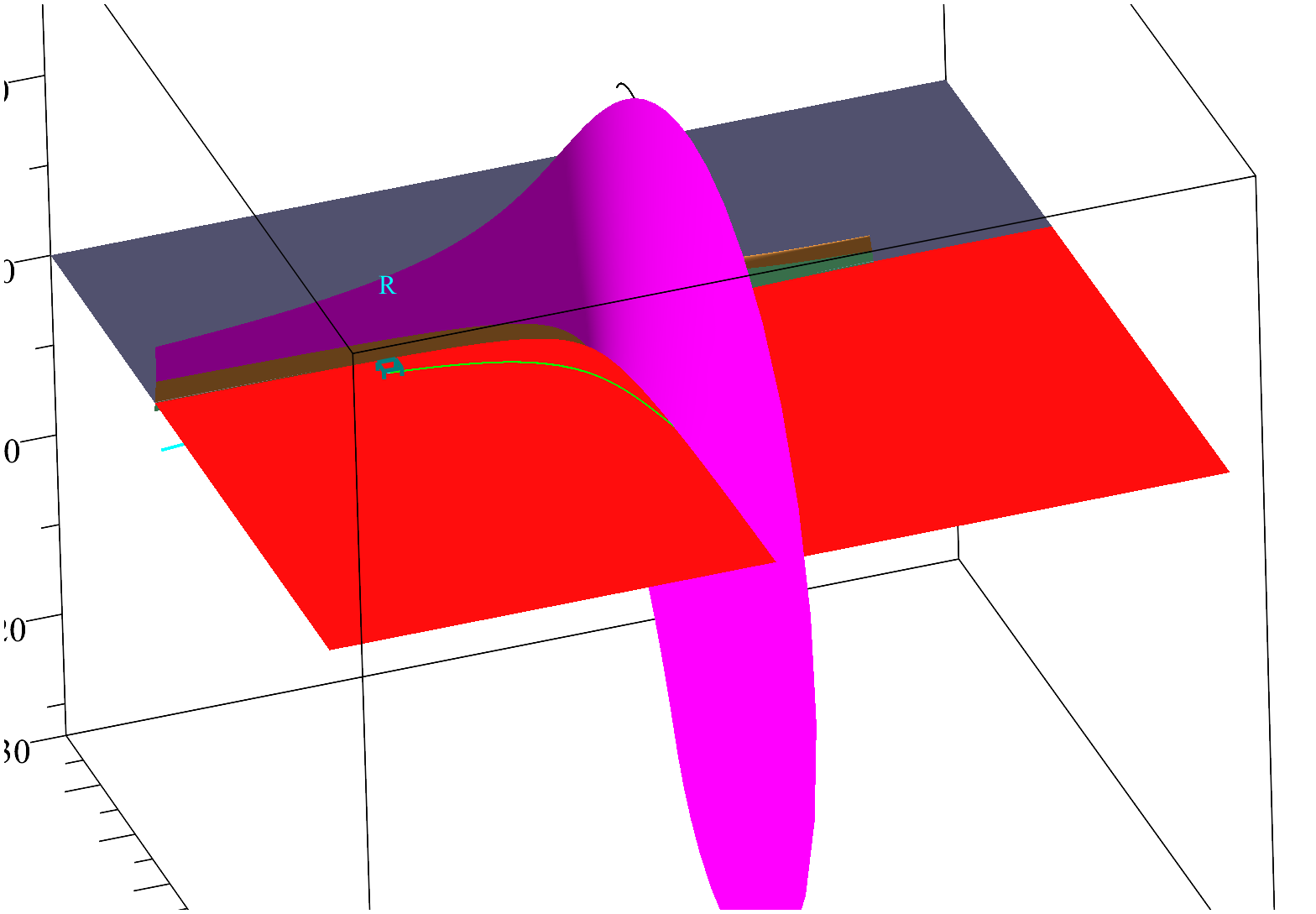}
		\caption[]{The 2-reverse wave sequence for\ \ $(z_{_{R}}=5,t_{_{R}}=3)\in\mathcal{C}_{_{f}}$\ \ is\ \ $\mathcal{H}_{_{2}}\mathcal{U}_{_{R}} \mathcal{R}_{_{2}}\mathcal{C}_{_{2}}$. The\ \ $\mathcal{R}_{_{2}}$\ \ is the green curve on\ \ $\mathcal{C}_{_{f}}$. The curve\ \ $H_{_{1}}$\ \ is drawn for\ \ $Y<0$ (below\ \ $\mathcal{C}_{_{f}}$) and\ \  $\mathcal{C}_{_{2}}$\ \ is the black curve behind the magenta surface. The Riemann solution from state\ \ $\mathcal{U}_{_{L}}$\ \ consists of a\ \ $\mathcal{H}_{_{1}}$\ \ from\ \ $\mathcal{U}_{_{L}}$\ \ to state\ \ $\mathcal{U}_{_{M}}$; from\ \ $\mathcal{U}_{_{M}}$\ \ there is a\ \ $\mathcal{R}_{_{2}}$\ \ to\ \ $\mathcal{U}_{_{R}}$.\label{figura1ex3-3}} 
	\end{center}
\end{figure}

\begin{figure}[htpb]
	\begin{center}
			\includegraphics[scale=0.8,width=0.75\linewidth]{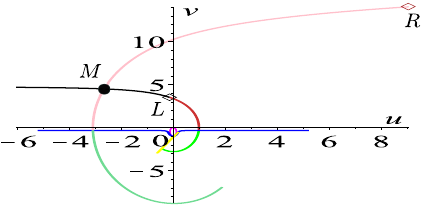}
		\caption[]{The Riemann solution  in plane\ \ $uv$, for\ \ $L=(u_{_{L}}=-0.125,v_{_{L}}=3.5)$\ \ and\ \ $R=(u_{_{R}}=9.048076925, v_{_{R}}=14.03846154)$. From\ \ $L$\ \ there is a\ \ $\mathcal{H}_{_{1}}$\ \  to a state\ \ $M$ (black curve); from\ \ $M$\ \ there is a\ \ $\mathcal{R}_{_{2}}$ (pink curve) to\ \ $R$.The blue curve is the inflection.\label{figura1ex3-4}} 
	\end{center}
\end{figure}

\newpage

\noindent
\textbf{Example 4.}

In this example, we utilize\ \ $z_{_{L}}=1$\ \ and\ \ $t_{_{L}}=-2$ (the corresponding value in the\ \ $uv$\ \ plane is\ \ $u_{_{L}}=0.125, v_{_{L}}=-2.5 $), the state\ \ $(z_{_{L}},t_{_{L}})$\ \ belongs to the region\ \ $III$. The sequence  through\ \  $\mathcal{U}_{_{L}}$\ \ is\ \ 
$\mathcal{C}_{_{1}}\mathcal{R}_{_{1}}\mathcal{U}_{_{L}}\mathcal{H}_{_{1}}$, see Figs. \ref{figura1ex4-1} (in wave manifold\ \ $\mathcal{W}$) and \ref{figura1ex4-2} (in the\ \ $uv$-plane). 

In Fig. \ref{figura1ex4-1}, the magenta surface is the saturated of\ \  $\mathcal{H}_{_{1}}$; the brown surface is the saturated of\ \ $\mathcal{R}_{_{1}}$\ \ and\ \ $\mathcal{R}_{_{2}}$; the green surface is the saturated of\ \ $\mathcal{C}_{_{2}}$.  The curve\ \ $\mathcal{C}_{_{2}}$\ \ is drawn to reach the double contact in\ \ $z=1/3$ (with\ \ $t<0$). In Fig. \ref{figura1ex4-2}, the double contact\ \ $z=1/3$\ \ is the purple straight line.

\begin{figure}[htpb]
	\begin{center}	\includegraphics[scale=0.8,width=0.47\linewidth]{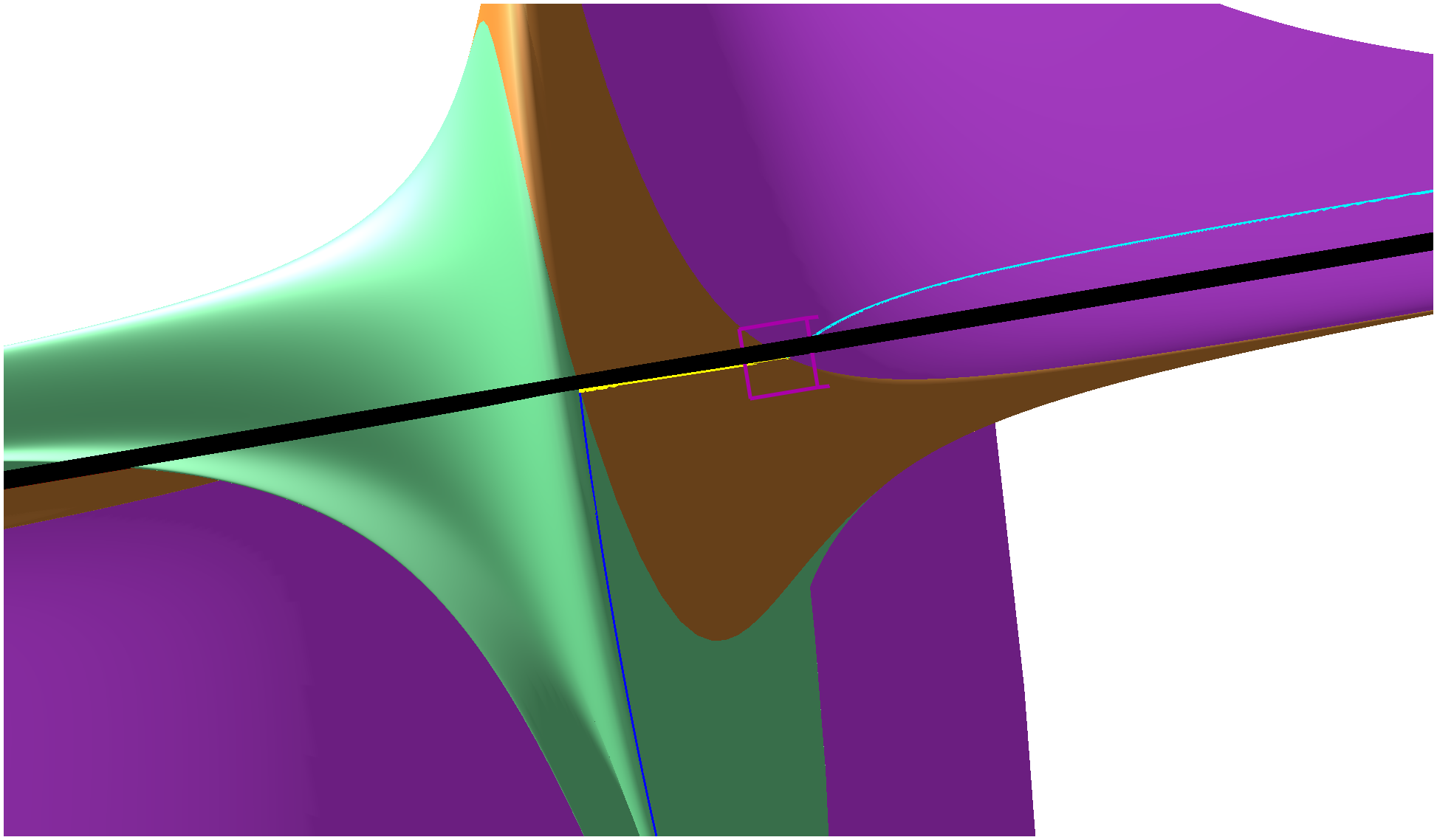}
\includegraphics[scale=0.8,width=0.47\linewidth]{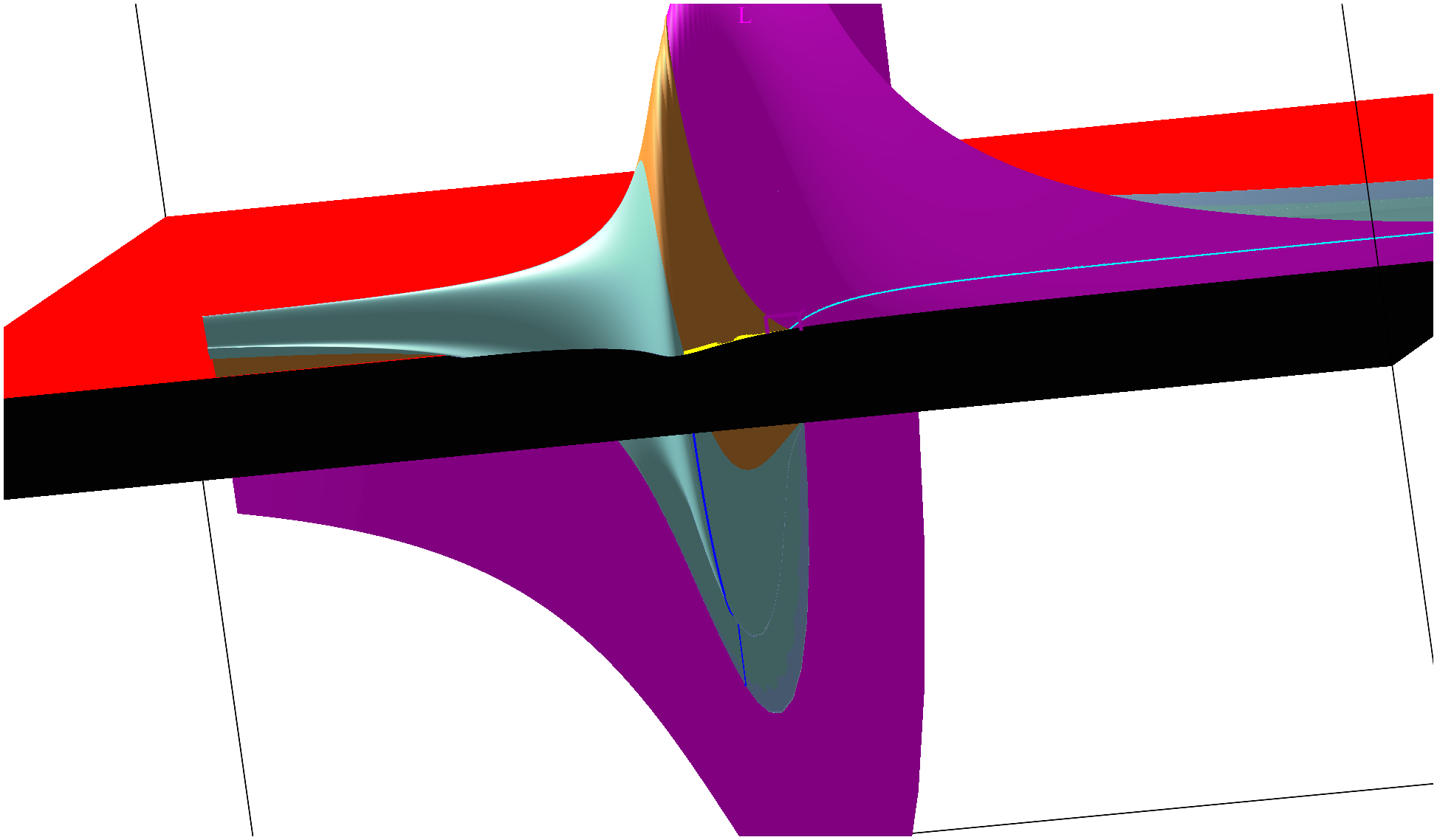}
		\caption[]{Here the values for\ \ $\mathcal{U}_{_{L}}$\ \ in\ \ $III$\ \ with\ \ $z_{_{L}}=1$\ \ and\ \ $t_{_{L}}=-2$. The sequence  through\ \  $\mathcal{U}_{_{L}}$\ \ is 
$\mathcal{C}_{{1}}\mathcal{R}_{_{1}}\mathcal{U}_{_{L}}\mathcal{H}_{_{1}}$. The magenta surface is the 
saturated of\ \  $\mathcal{H}_{_{1}}$ (blue curve); the brown surface is the saturated of\ \ $\mathcal{R}_{_{1}}$ (yellow curve); the green surface is the saturated of\ \ $\mathcal{C}_{_{2}}$ (cyan).  The curve\ \ $\mathcal{C}_{_{2}}$\ \ is drawn to reach the double contact in\ \ $z=1/3$.\label{figura1ex4-1}} 
	\end{center}
\end{figure}

\begin{figure}[htpb]
	\begin{center}
		\includegraphics[scale=0.8,width=0.6\linewidth]{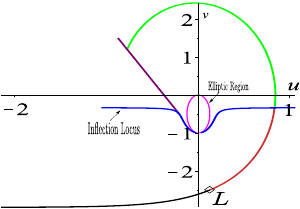}	
		\caption[]{The wave sequence on the plane\ \ $uv$, here\ \ $u_{_{L}}=0.125$\ \ and\ \ $v_{_{L}}=-2.5$. From\ \ $L$\ \ there is a\ \ $\mathcal{H}_{_{1}}$ (black curve) and a\ \ $\mathcal{R}_{_{1}}$ (red curve), since the rarefaction crosses the inflection (blue curve), there is a\ \
		$\mathcal{C}_{_{2}}$\ \ to the double contact (purple straight line)\ \ $z=1/3$. The magenta ellipse is the elliptic region.\label{figura1ex4-2}} 
	\end{center}
\end{figure}

To construct the 2-reverse wave sequence, here we consider\ \  $(z_{_{R}}=4.961249694,t_{_{R}}=1)\in\mathcal{C}_{_{f}}$ (the corresponding values on the plane\ \ $uv$\ \ are\ \ $(u_{_{R}}=3, v_{_{R}}=4)$. From\ \ $\mathcal{U}_{_{r}}$ (obtained as described in Algorithm RS) we draw the 2-reverse wave sequence that is given, in this example,  by\ \ $\mathcal{C}_{_{2}}\mathcal{R}_{_{2}}\mathcal{U}_{_{R}} \mathcal{H}_{_{2}}$, as described in Fig. \ref{figura1ex4-3}. In the wave manifold\ \ $\mathcal{W}$, $\mathcal{R}_{_{2}}$\ \ is the green curve on\ \  $\mathcal{C}_{_{f}}$. The curve\ \ $H_{_{1}}$\ \ is the cyan curve drawn for\ \ $Y>0$\ \ and increasing\ \ $z$. The curve\ \  $\mathcal{C}_{_{2}}$\ \ is the black curve below\ \ $\mathcal{C}_{_{f}}$, from\ \ $\mathcal{U}_{_{R}}$\ \ the curve\ \ $\mathcal{R}_{_{2}}$\ \ crosses the saturated surface of\ \ $\mathcal{C}_{_{1}}$ (green surface).

In this case, the Riemann solution from state\ \ $\mathcal{U}_{_{L}}$\ \ consists of a\ \ $\mathcal{R}_{_{1}}$\ \ followed by a\ \ $\mathcal{C}_{_{1}}$\ \ to a state\ \ $\mathcal{U}_{_{M}}$; from\ \ $\mathcal{U}_{_{M}}$\ \ there is a\ \ $\mathcal{R}_{_{2}}$\ \ to\ \ $\mathcal{U}_{_{R}}$. The solution in\ \ $\mathcal{W}$\ \ is described in Fig. \ref{figura1ex4-3} and in the\ \ $uv$\ \ plane is described in Fig. \ref{figura1ex4-4}.

\begin{figure}[htpb]
	\begin{center}	\includegraphics[scale=0.8,width=0.47\linewidth]{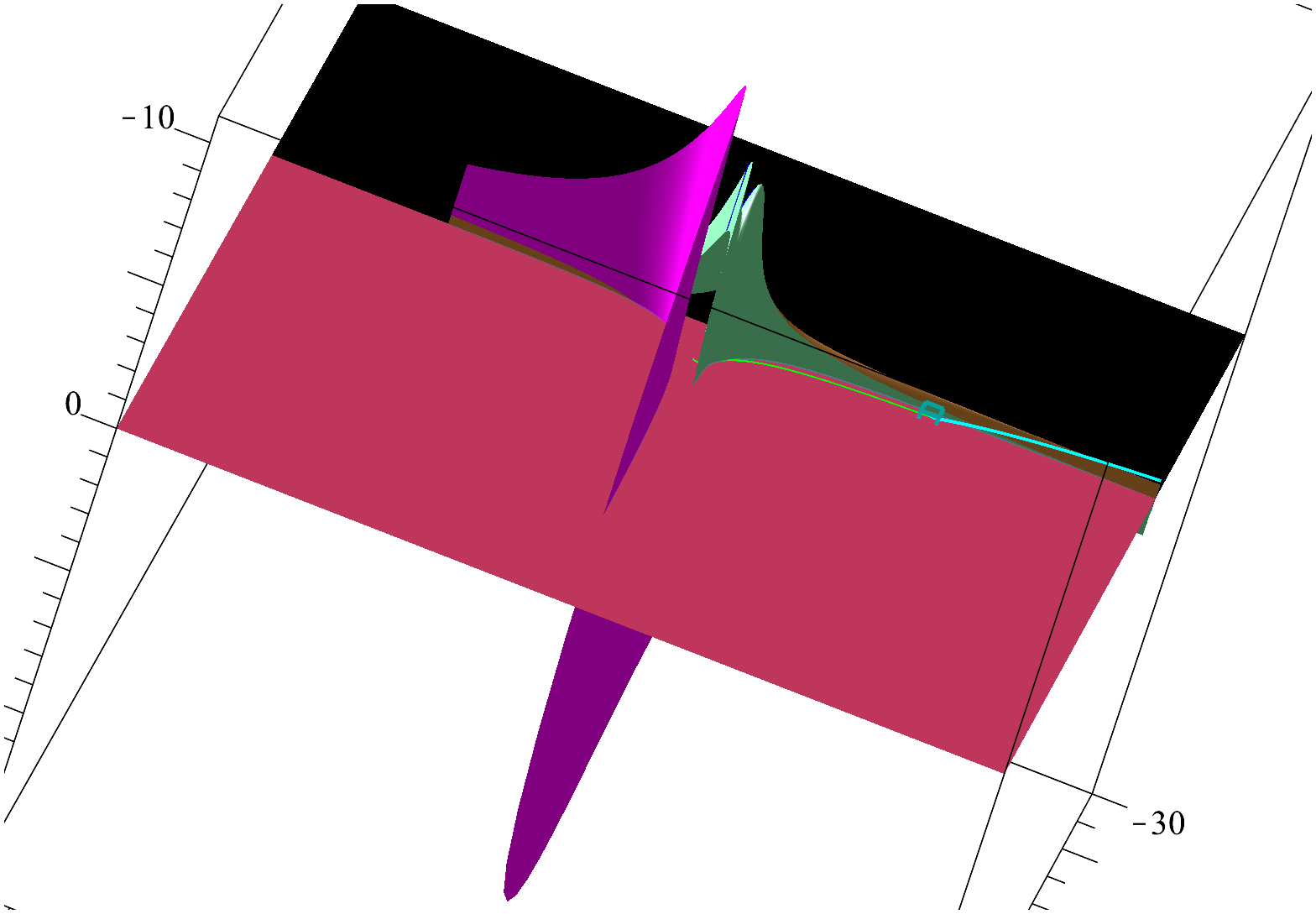}
\includegraphics[scale=0.8,width=0.47\linewidth]{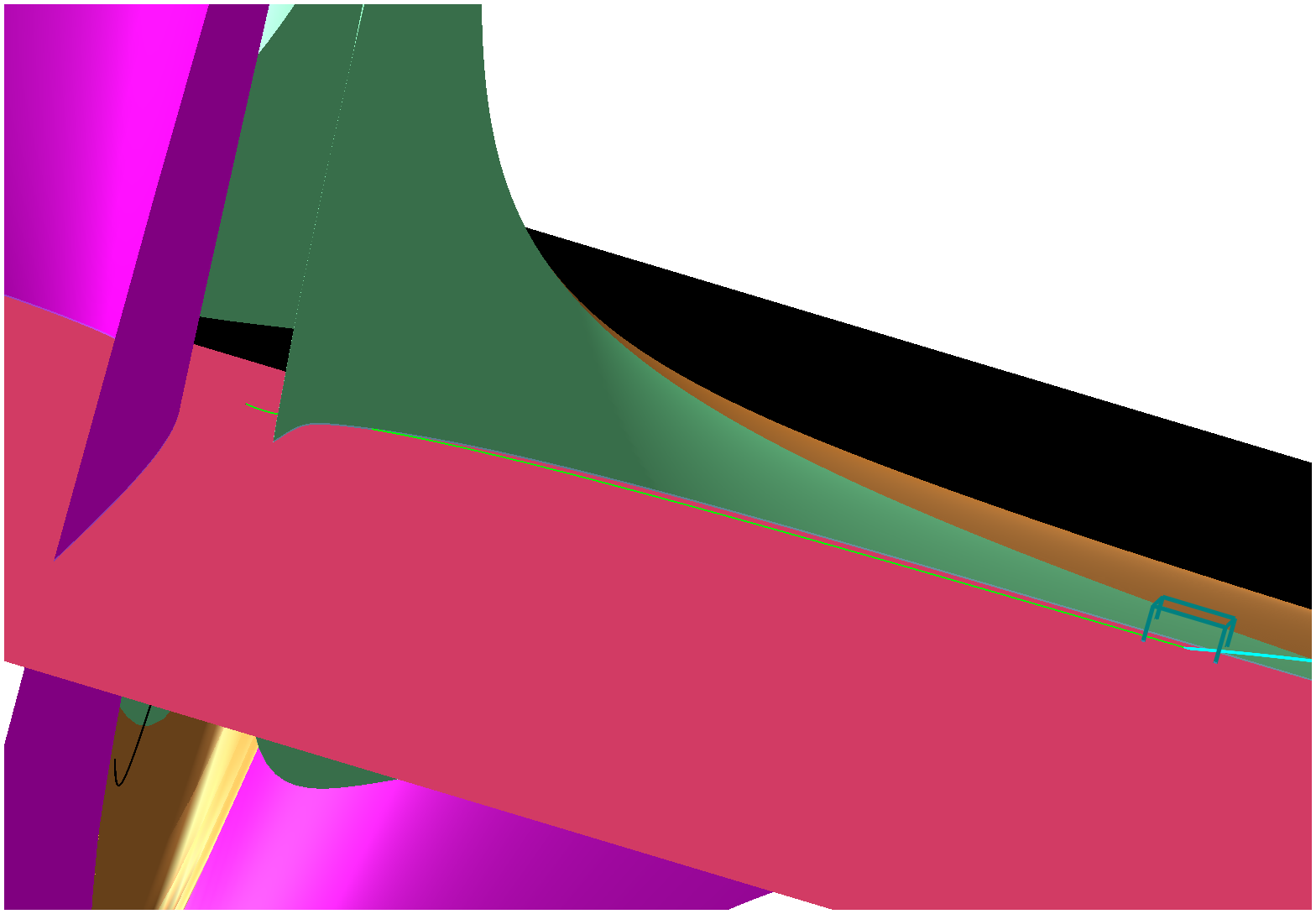}
		\caption[]{The 2-reverse wave sequence for\ \ $(z_{_{R}}=4.961249694,t_{_{R}}=1)\in\mathcal{C}_{_{f}}$\ \  is\ \  $\mathcal{C}_{_{2}}\mathcal{R}_{_{2}}\mathcal{U}_{_{R}} \mathcal{H}_{_{2}}$. Here\ \ $\mathcal{R}_{_{2}}$\ \ is the green curve on\ \  $\mathcal{C}_{_{f}}$. The curve\ \ $H_{_{1}}$\ \ is the cyan curve drawn for\ \ $Y>0$\ \ and increasing\ \ $z$. The curve\ \  $\mathcal{C}_{_{2}}$\ \ is the black curve below\ \ $\mathcal{C}_{_{f}}$, from\ \ $\mathcal{U}_{_{R}}$\ \ the curve\ \ $\mathcal{R}_{_{2}}$\ \ crosses the saturated surface of\ \ $\mathcal{C}_{_{1}}$ (green surface). The Riemann solution from state\ \ $\mathcal{U}_{_{L}}$\ \ consists of a\ \ $\mathcal{R}_{_{1}}$ followed by a\ \ $\mathcal{C}_{_{1}}$\ \ to a state\ \ $\mathcal{U}_{_{M}}$; from\ \ $\mathcal{U}_{_{M}}$\ \ there is  a\ \ $\mathcal{R}_{_{2}}$\ \ to\ \ $\mathcal{U}_{_{R}}$.} \label{figura1ex4-3}
	\end{center}
\end{figure}
\begin{figure}[htpb]
	\begin{center}
			\includegraphics[scale=0.8,width=0.7\linewidth]{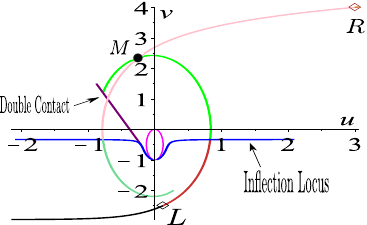}
		\caption[]{The Riemann solution  in plane\ \ $uv$, for\ \ $L=(u_{_{L}}=0.125,v_{_{L}}=-2.5)$\ \ and\ \ $R=(u_{_{R}}=3, v_{_{R}}=4)$. From\ \ $L$\ \ there is a\ \ $\mathcal{R}_{_{1}}$ (red curve) followed by a\ \ $\mathcal{C}_{_{1}}$ (green curve) to state\ \ $M$ (black curve); from\ \ $M$\ \ there is a\ \ $\mathcal{R}_{_{2}}$ (pink curve) to\ \ $R$.The blue curve is the inflection. The purple is the double contact that corresponds to\ \ $z=1/3$. \label{figura1ex4-4}} 
	\end{center}
\end{figure}

\subsection{The Regions in Phase Space}

For the flux function\ \  $f(u,v)$\ \ and\ \ $g(u,v)$\ \ given in equation $(\ref{fgeq})$, we can obtain the eigenpairs (eigenvalues,\ \  $\lambda$,\ \ and eigenvectors,\ \ $\vec{\boldsymbol{r}}$)\ \  in the phase space\ \ $(u,v)$. These eigenpairs are solutions of:
\begin{equation}
\mathbf{A}\vec{\boldsymbol{r}}=\lambda \vec{\boldsymbol{r}}, \quad \text{ where }\  \lambda
\ \text{ is obtained solving }\ \operatorname{det}(\mathbf{A}-\lambda \mathbf{I})=0.\label{eqts}
\end{equation}
Here\ \ $\mathbf{A}$\ \ is the jacobian of\ \ $(f,g)^{^{\mathrm{T}}}$\ \ and\ \ $\mathbf{I}$\ \ is the\ \ $2\times 2$\ \ identity matrix.  Notice that from equation $(\ref{eqts}.b)$ we have a degree two polynomial, thus for each state $(u,v)$ in the phase space we have associated two eigenvalues\ \ $\lambda$. The phase space is classified using\ \ $\lambda$\ \ and\ \ $\vec{\boldsymbol{r}}$. The region in the phase space\ \ $(u,v)$\ \ for which each state admits two real eigenvalues\ \ $\lambda$\ \ and a basis of eigenvectors (local) is called \emph{hyperbolic region}, which we denote as\ \ $\mathcal{R}_{_{H}}$. The region for which we have complex (and conjugated) eigenvalues\ \ $\lambda$\ \ is called \emph{elliptic region}, denoted as\ \ $\mathcal{R}_{_{E}}$. The boundary (if this boundary exits) between\ \ $\mathcal{R}_{_{H}}$\ \ and\ \ $\mathcal{R}_{_{E}}$\ \ is denoted as\ \ $\partial \mathcal{R}_{_{E}}$,\ \ here the two eigenvalues are equals.

For\ \  $f(u,v)$\ \ and\ \ $g(u,v)$\ \ given in equation (\ref{fgeq}) the eigenvalues are:
\begin{align}
&\lambda_{_{1}}=\frac{\alpha_{_{1}}-\sqrt{\alpha_{_{2}}}}{2}\quad \text{ and } \quad\lambda_{_{2}}=\frac{\alpha_{_{1}}+\sqrt{\alpha_{_{2}}}}{2}, \quad \text{ where }\alpha_{_{1}}=u(b_{_{1}}+2)+a_{_{1}}+a_{_{4}}\quad \text{and}\\
&\alpha_{_{2}}=a_{_{1}}^{^{2}}+(2b_{_{1}}u-2a_{_{4}})a_{_{1}}+a_{_{4}}^{^{2}}-2a_{_{4}}b_{_{1}}u+b_{_{1}}^{^{2}}u^{^{2}}+4(v+a_{_{2}})(a_{_{3}}+v)
\end{align}
To obtain\ \ $\mathcal{R}_{_{H}}$,\ \ $\mathcal{R}_{_{E}}$\ \ and\ \ $\partial \mathcal{R}_{_{E}}$,\ \ we need to study the signal of\ \ $\alpha_{_{2}}$.  After some algebraic, we can write\ \  $\alpha_{_{2}}$\ \  and define\ \  $\widetilde{\alpha}_{_{2}}$\ \  as:
\begin{equation}
\alpha_{_{2}}=(2v+a_{_{2}}+a_{_{3}})^{^{2}}+(b_{_{1}}u+a_{_{1}}-a_{_{4}})^{^{2}}-c^{^{2}},\quad 
 \widetilde{\alpha}_{_{2}}=\frac{[v+\frac{a_{_{2}}+a_{_{3}}}{2}]^{^{2}}}{(\frac{c}{2})^{^{2}}}+\frac{[u+\frac{a_{_{1}}-a_{_{4}}}{b_{_{1}}}]^{^{2}}}{(\frac{c}{b_{_{1}}})^{^{2}}}-1.
 \label{alfa2}
\end{equation}
 Here, we use that\ \  $c=a_{_{3}}-a_{_{2}}>0$. From\ \  $\widetilde{\alpha}_{_{2}}$, we can define:
\begin{align}
\mathcal{R}_{_{H}}=\{(u,v)\;/ \; \widetilde{\alpha}_{_{2}}>0\},\;
\mathcal{R}_{_{E}}=\{(u,v)\;/ \; \widetilde{\alpha}_{_{2}}<0\}  \text{ and } \partial \mathcal{R}_{_{E}}=\{(u,v)\; / \; \widetilde{\alpha}_{_{2}}=0\}.
\end{align}
 Notice that\ \ $\partial \mathcal{R}_{_{E}}$\ \ represents an ellipse. In  Fig. \ref{figelipse}.{\it a}, we illustrate\ \ $\mathcal{R}_{_{H}}$,\  $\mathcal{R}_{_{E}}$\ \ and\ \ $\partial \mathcal{R}_{_{E}}$\ \ for  particular values of\ \ $a_{_{1}}$,\  $a_{_{2}}$,\  $a_{_{3}}$,\   $a_{_{4}}$,\  $c$\ \ and\ \ $b_{_{1}}$\ \ described in caption of figure. 

 
 \begin{figure}[h!]
\includegraphics[scale=0.8,width=0.5\linewidth]{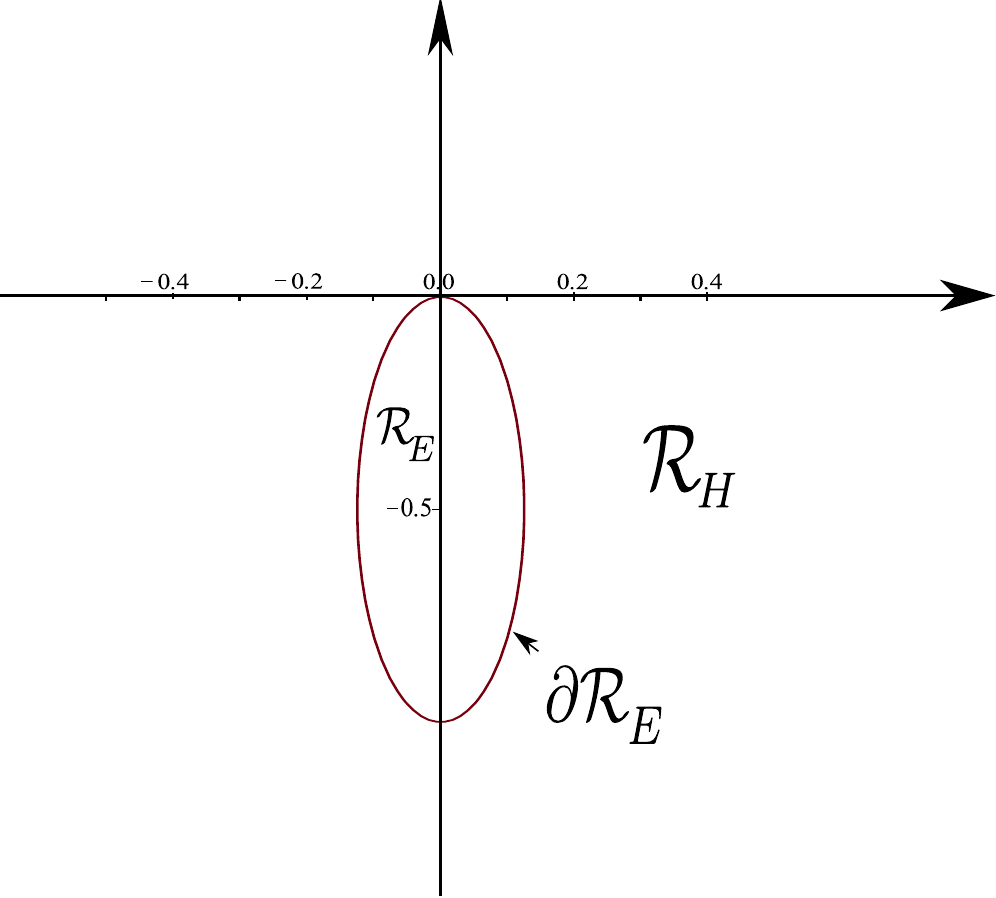}
\includegraphics[scale=0.8,width=0.5\linewidth]{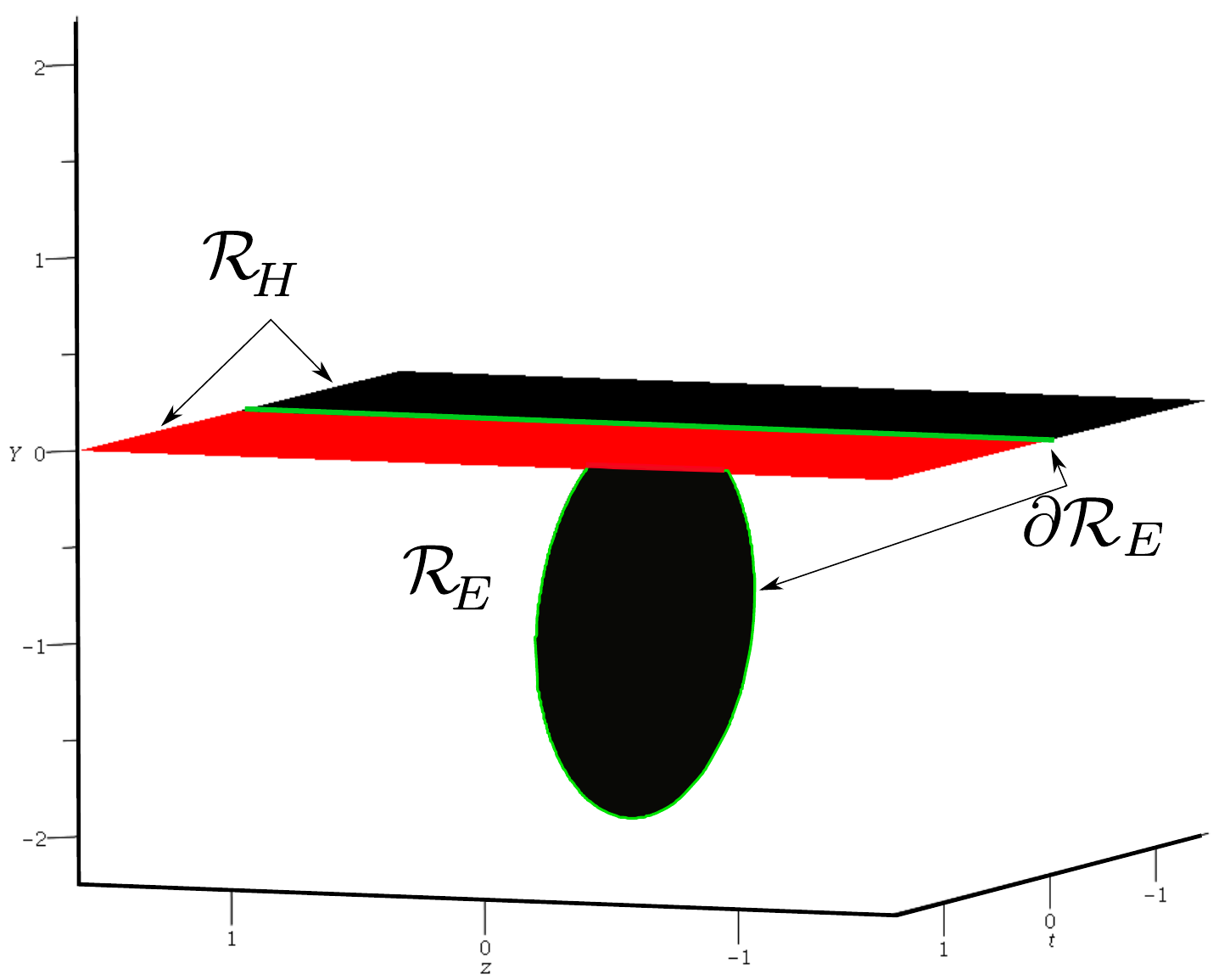}
\caption[]{\textit{Left - a:).} Regions\ \ $\mathcal{R}_{_{E}}$,\ \ $\mathcal{R}_{_{H}}$\ \ and the boundary\ \ $\partial \mathcal{R}_{_{E}}$. Here we use\ \ $a_{_{1}}=0$,\ \ $a_{_{2}}=0$,\ \ $a_{_{3}}=1$,\ \  $a_{_{4}}=0$,\ \ $c=1$,\ \ and\ \ $b_{_{1}}=8$.}
\label{figelipse}
\end{figure}

From the variables change defined Section \ref{sec:wmhc}, we  obtain\ \
$u$,\ \ $v$,\ \  $u^\prime$\ \ and\ \ $v^\prime$\ \ as function of\ \ $Y$,\ \ $t$\ \ and\ \ $z$. The variables\ \ $u$\ \  and\ \ $v$\ \ are:
\begin{align}
&u=\frac{2t(z^{^{4}}-1)+Yb_{_{1}}z^{^{3}}+(-2a_{_{1}}+2a_{_{4}})(z^{^{2}}+1)+(Yb_{_{1}}+4c)z}{2(z^{^{2}}+1)b_{_{1}}},\label{u}\\\\
&v=\frac{2tz^{^{3}}+(Y-2a_{_{3}})z^{^{2}}+2tz+Y-2a_{_{2}}}{2z^{^{2}}+2}\label{v}
\end{align}
to obtain\ \ $u^\prime$\ \ and\ \ $v^\prime$\ \ we use\ \ $u$\ \ and\ \ $v$\ \ interchanging\ \ $Y$\ \ with\ \ $-Y$.

Substituting in\ \ $\alpha_{_{2}}$,\ \  $u$\ \ and\ \ $v$\ \  given equations\ \ (\ref{u}),\ \ (\ref{v}),\ \  we obtain an expression that we denote as\ \ $\alpha_{_{M}}(Y,t,z)$:  
\begin{eqnarray}
\alpha_{_{M}}(Y,t,z)&=&\frac{(Y^{^{2}}b_{_{1}}^{^{2}}+12t^{^{2}})z^{^{4}}+16Ytz^{^{3}}+[(b_{_{1}}^{^{2}}+4)Y^{^{2}}+8c(b_{_{1}}-1)Y+4c^{^{2}}+12t^{^{2}}-4c]z^{^{2}}}{4z^{^{2}}+4}\nonumber\\\nonumber\\
&&+\frac{4t^{^{2}}z^{^{6}}+4Yb_{_{1}}tz^{^{5}}-4t(b_{_{1}}-4)Yz+4Y^{^{2}}+8Yc+4c^{^{2}}+4t^{^{2}}-4c}{4z^{^{2}}+4}.
\end{eqnarray}
If we set\ \  $\alpha_{_{M}}=0$, we obtain the extension of\ \ $\mathcal{R}_{_{E}}$\ \ in\ \ $\mathcal{W}$,\ \ which we will call\ \ \emph{coincidence surface},\ \ \emph{i.e.}, 
\begin{equation}
coincidence=\{(z,t,Y)\;/\;\alpha_{_{M}}=0\}.
\end{equation}
Similarly, we define  \emph{coincidence$'$} that is obtained substituting\ \ $u^\prime$\ \ and\ \ $v^\prime$\ \ in\ \ $\alpha_{_{2}}$. This expression is denoted as\ \ $\alpha_{_{M}}^\prime=\alpha_{_{M}}(z,t,-Y)$,\ \ thus\ \ $coincidence^\prime=\{(z,t,Y)\;/\;\alpha_{_{M}}^\prime=0\}$. In Fig. \ref{suptan}, we show \emph{coincidence surface} and the characteristic\ \ $\mathcal{C}$. The\ \ \emph{ coincidence$'$ surface}\ \ is the symmetric projection of the coincidence surface in\ \ $\mathcal{C}$.  

 \begin{figure}[h!]
 \centering
\includegraphics[scale=0.28]{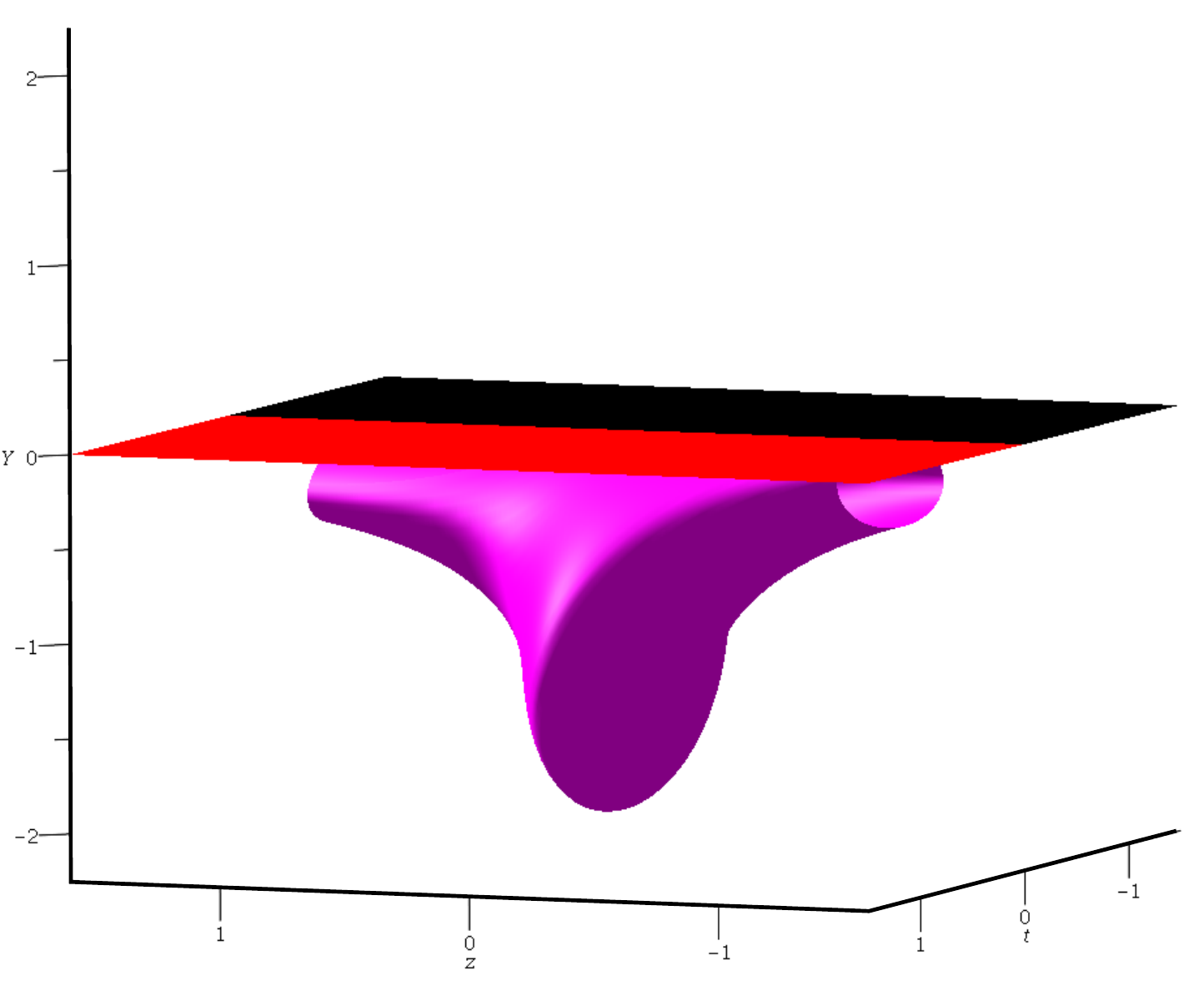}
\caption[]{\textit{Left - a:).} Regions\ \ $\mathcal{R}_{_{E}}$,\ \ $\mathcal{R}_{_{H}}$\ \ and the boundary\ \ $\partial \mathcal{R}_{_{E}}$. Here, we use\ \ $a_{_{1}}=0$,\ \ $a_{_{2}}=0$,\ \ $a_{_{3}}=1$,\ \  $a_{_{4}}=0$,\ \ $c=1$,\ \ and\ \ $b_{_{1}}=8$.}
\label{suptan}
\end{figure}
\newpage 
\section*{Appendix\label{sec:Appendices}}

\setcounter{section}{0}

\setcounter{equation}{0}
\section{Characterizing the L2 Condition \label{app:L2}}
In this appendix section, we will prove that the necessary and sufficient condition to verify the condition \textbf{L2}  is that\ \ $z^{^{2}}<1/(b_{_{1}}+1).$

As in Section \ref{sec:Yzt-subsection}, parametric equations for the Hugoniot$'$ curve through a point\ \ $(t_{_{0}},z_{_{0}},Y_{_{0}})$\ \ is given by 
\begin{equation}
\left\{
\begin{array}{l}
Y=\displaystyle \frac{A'z^{^{2}}+B'z +C'}{(z_{_{0}}^{^{2}}+1)[(b_{_{1}}-1)z^{^{2}}+1]}\\\\
t=\displaystyle \frac{D'z^{^{3}}+E'z^{^{2}}+F'z+G'}{c(z_{_{0}}^{^{2}}+1)[(b_{_{1}}-1)z^{^{4}}+b_{_{1}}z^{^{2}}+1]},
\end{array}
\right.
\label{eq:hugpontoli}
\end{equation}

\noindent where\ \ $A'$,\ \ $B'$,\ \ $C'$,\ \ $D'$,\ \ $E'$,\ \ $F'$,\ \ $G'$\ \ are obtained from equations \eqref{eq:hugponto} by changing\ \ $Y$\ \ into\ \ $-Y$\ \ and\ \ $Y_{_{0}}$\ \ into\ \ $-Y_{_{0}}$.

Parametric equations for Hugoniot$'$ curve through a point of $Son'$ are obtained from the equations \eqref{eq:hugpontoli} substituting\ \ $t_{_{0}}$\ \ by\ \ $t'_{_{0}}$,\ \ $t'_{_{0}}$\ \ as in Section \ref{sec:sonlifs}, getting\ \ $(Y_{_{hug'}}(z),t_{_{hug'}}(z))$.



We need to calculate the speed\ \ $s$\ \ at the intersections points of the Hugoniot$'$ curve with\ \ $\mathcal{C}$\ \ and then write the condition \textbf{L2}. As the inequalities involved change only in\ \ $Son$,\ \ it is enough to verify them at a point of\ \ $Son'$.
The  speed\ \ $s$\ \ at the point\ \ $(t'_{_{0}},z_{_{0}},Y_{_{0}})$\ \ is given by the equation \eqref{eq:speedson'}.


Let\ \ $z_{_{1}}$\ \ and\ \ $z_{_{2}}$\ \ be,\ \ $z_{_{1}}< z_{_{2}}$,\ \ the solutions of  the equation $Y_{_{hug'}}=0$. The roots\ \ $z_{_{1}}$\ \ and\ \ $z_{_{2}}$\ \ are the\ \ $z$\ \ coordinates of the intersection points of the Hugoniot$'$ curve with\ \ $\mathcal{C}$. The speed,\ \ $s_{_{hug'}}(z)$,\ \ along the Hugoniot$'$ curve through\ \ $(t'_{_{0}},z_{_{0}},Y_{_{0}})$\ \ is obtained by changing\ \ $t$\ \ into\ \ $t_{_{hug'}}(z)$\ \ in equation \eqref{eq:speed}. We must compare the values\ \ $s_{_{hug'}}(z_{_{1}})$\ \ and\ \ $s_{_{hug'}}(z_{_{2}})$\ \ with\ \ $s_{_{son'}}$. Denoting by\ \  $s_{_{hug'n}}(z)$\ \  the numerator of\ \ $s_{_{hug'}}(z)$,\ \ $s_{_{hug'd}}(z)$\ \  the denominator of\ \ $s_{_{hug'}}(z)$\ \  and\ \ $Y_{_{hug'n}}(z)$\ \  the numerator of\ \  $Y_{_{hug'}}(z)$,\ \ we can write

\begin{equation*}
s_{_{hug'n}}(z)=p_{_{1}}(z)Y_{_{hug'n}}(z)+r_{_{1}}(z)
\end{equation*}
and 
\begin{equation*} 
s_{_{hug'd}}(z)=p_{_{2}}(z)Y_{_{hug'n}}(z)+r_{_{2}}(z).
\end{equation*}
It follows that,\ \ $s_{_{hug'\mathcal{C}}}(z)$,\ \ given by
$$s_{_{hug'\mathcal{C}}}(z)= \frac{r_{_{1}}(z)}{r_{_{2}}(z)},$$
has the same value of\ \ $s_{_{hug'}}(z)$\ \ at the intersection points of the Hugoniot$'$ curve through\ \ $(t'_{_{0}},z_{_{0}},Y_{_{0}})$\ \ with\ \ $\mathcal{C}$.\ \ The expression of\ \ $s_{_{hug'\mathcal{C}}}(z)$\ \ is of the form\ \ $(az+b)/(cz+d)$. Our goal is to get a condition such that\ \ $s_{_{son'}}(t'_{_{0}},z_{_{0}}.Y_{_{0}})$\ \  satisfies\ \ $s_{_{hug'\mathcal{C}}}(z_{_{1}})< s_{_{son'}}(t'_{_{0}},z_{_{0}}.Y_{_{0}})<s_{_{hug'\mathcal{C}}}(z_{_{2}})$.

In a more general way, we have the following problem: given a number\ \ $s$\ \ we want a condition such that 
$$\frac{az_{_{1}}+b}{cz_{_{1}}+d}<s<\frac{az_{_{2}}+b}{cz_{_{2}}+d},$$ 
where\ \ $z_{_{1}}$\ \ and\ \ $z_{_{2}}$\ \ are the roots of the polynomial\ \ $fz^{^{2}}+gz+h$. Straightforward computations give that the condition is
$$\frac{(c^{^{2}}h-dcg+d^{^{2}}f)s^{^{2}}+[agd+bcg-2(ach+bdf)]s+a^{^{2}}h-abg+b^{^{2}}f}{c^{^{2}}h-dcg+d^{^{2}}f}<0.$$

Changing\ \ $a$,\ \ $b$\ \ into the coefficients of the numerator of\ \ $s_{_{hug'\mathcal{C}}}(z)$,\ \ $c$\ \ and\ \  $d$ into the coefficients of the denominator of\ \ $s_{_{hug'\mathcal{C}}}(z)$\ \ and\ \ $f$,\ \ $g$,\ \ $h$\ \ into the coefficients of the numerator of\ \  $Y_{_{hug'}}(z)$,\ \ we get that the condition is\ \ $$cond= \frac{Y_{_{0}}^{^{2}}((b_{_{1}}+1)z_{_{0}}^{^{2}}-1}{2}<0,$$
it follows that the condition \textbf{L2} is satisfied if and only if $-1 / \sqrt{b_{_{1}}+1}<z< 1 / \sqrt{b_{_{1}}+1}.$

\section{Lax's inequalities in the wave manifold}
\label{laxinequalities}
The Lax's inequalities for shocks are used to select shock (entropy condition) and used to prove uniqueness in the solution. 

For a\ \ $2\time2$\ \ system of equations (\ref{eq:cons-law}), we have a pair of inequalities associated to a 1-shock and another pair of inequalities associated to a 2-shock.

In the Lax's theory, the system is strictly hyperbolic with two different eigenvalues\ \ $\lambda_{_{1}}(W)<\lambda_{_{2}}(W)$. For regions for which the solution is continuous, we can define the $i$-characteristic on the plane\ \ $xt$, for\ \ $i=1,2$,  solving:
\begin{equation}
	\frac{dx}{dt}=\lambda_{_{i}}(W).
\end{equation}

For a shock between a state\ \ $W^-=(u^-,v^-)$\ \ to\ \ $W^+=(u^+,v^+)$, with shock speed\ \ $s(W^-,W^+)$, the 1-shock satisfies the inequalities given by:
\begin{equation}
	s(W^-,W^+)<\lambda_{_{1}}(W^-)\quad\text{ and }\quad \lambda_{_{1}}(W^+)<s(W^-,W^+)<\lambda_{_{2}}(W^+).\label{lax1new}
\end{equation}
The condition (\ref{lax1new}) states that the 1-characteristic wave, on the plane\ \ $xt$, (of each side of the shock) enter in the shock wave (with slope\ \ $s$).

The 2-shock satisfies the inequalities given by:
\begin{equation}
	s(W^-,W^+)>\lambda_{_{2}}(W^+)\quad \text{ and }\quad \lambda_{_{1}}(W^-)<s(W^-,W^+)<\lambda_{_{2}}(W^-).\label{lax2new}
\end{equation}
The condition (\ref{lax2new}) states that the 2-characteristic wave, on the plane\ \ $xt$, (of each side of the shock) enter in the shock wave (with slope\ \ $s$).

Here, we describe the Lax inequalities in the wave manifold.

For a state\ \ $\mathcal{U}\in\mathcal{W}$, using the Hugoniot$'$, we obtain the projections\ \ $\mathcal{U}'_{_{s}}\in\mathcal{C}_{_{s}}$\ \ and\ \ $\mathcal{U}'_{_{f}}\in\mathcal{C}_{_{f}}$, using Eq. (\ref{usuli}). So, we have that the eigenvalue associated to\ \ $\lambda_{_{1}}$\ \  in the wave manifold is $\lambda_{_{s}}$ and the eigenvalue associated to $\lambda_{_{2}}$ in the wave manifold is $\lambda_{_{f}}$. We consider another\ \ $\widetilde{\mathcal{U}}\in sh(\mathcal{U})$, whose projections in\ \ $\mathcal{C}_{_{s}}$\ \ and\ \ $\mathcal{C}_{_{f}}$\ \ are, respectivelly, $\widetilde{\mathcal{U}}'_{_{s}}$\ \ and\ \ $\widetilde{\mathcal{U}}'_{_{f}}$.    

Using the previous analysis, the 1-shock in the wave manifold between the states\ \ $\mathcal{U}\in\mathcal{W}$\ \ and\ \  $\widetilde{\mathcal{U}}\in sh(\mathcal{U})$\ \ becomes:
\begin{equation}
	s(\widetilde{\mathcal{U}})<=s(\mathcal{U}_{_{s}}')=\lambda_{_{s}}(\mathcal{U}'_{_{s}})\quad\text{ and }\quad \lambda_{_{s}}(\widetilde{\mathcal{U}}')=
	s(\widetilde{\mathcal{U}}'_{_{s}})<s(\widetilde{\mathcal{U}})<
	s(\widetilde{\mathcal{U}}'_{_{f}})=\lambda_{_{f}}(\widetilde{\mathcal{U}}_{_{f}}').\label{lax1neww}
\end{equation}
Here, we use that the shock speed\ \ $s$\ \ equals to the\ \ $\lambda$\ \ in\ \ $\mathcal{C}$.

The 2-shock in the wave manifold between the states\ \ $\mathcal{U}\in\mathcal{W}$\ \ and\ \  $\widetilde{\mathcal{U}}\in sh(\mathcal{U})$\ \ becomes:
\begin{equation}
	s(\widetilde{\mathcal{U}})<s(\mathcal{U}_{_{f}}')=\lambda_{_{f}}(\widetilde{\mathcal{U}}'_{_{f}})\quad\text{ and }\quad \lambda_{_{s}}({\mathcal{U}}')=
	s({\mathcal{U}}'_{_{s}})<s(\widetilde{\mathcal{U}})<
	s({\mathcal{U}}'_{_{f}})=\lambda_{_{f}}({\mathcal{U}}_{_{f}}').\label{lax2neww}
\end{equation}

 \section{Choices\label{app:cho}}
 
 \subsection{Choice 1 \label{app:cho1}}
 
 This paper, as many previous ones, studies systems of  two conservation laws \ref{eq:cons-law}, \emph{ i.e.}, system
 $$W_{_{t}}+F(W)_{_{x}}=0,$$
 with\ \ $F=(f,g)$\ \ where\ \ $f(u,v)=v^{^{2}}/2+(b_{_{1}}+1)u^{^{2}}/2+a_{_{1}}u+a_{_{2}}v $\ \ and\ \ $g(u,v)=uv+a_{_{3}}u+a_{_{4}}v$. Actually, here we are considering the so called symmetric case($b_{_{2}}=0$).
 
 Why these particular\ \ $f$\ \ and\ \ $g$?
 
 In \cite{Schaeffer87}, the above equation was considered with the hypothesis that\ \ $DF$\ \ was hyperbolic (distinct eigenvalues) everywhere except at\ \ $(0,0)$\ \ where\ \ $DF(0,0)=I_{_{2\times 2}}$, \emph{i.e.},\ \ $(0,0)$\ \ is an umbilic point. This led to the study of\ \ $F$\ \ quadratic and it was shown that it is sufficient to consider\ \ $F= grad(C)$\ \ with\ \ 
 $C= au^{^{3}}/3+buv+uv^{^{2}}$.
 
 In \cite{Palmeira88}, it was started the study of\ \ $F=grad(C)\ +$ linear terms, in such a way that\ \ $F$\ \ is no longer a gradient.The addition of linear terms replaced the umbilic point by an elliptic region. Before adding linear terms, a simple change of coordinates was done and\ \ $C$\ \ became\ \ $uv^{^{2}}/2+[(b_{_{1}}+1)u^{^{3}}]/6-b_{_{2}}v^{^{3}}/6$. The reason for this change of coordinates was to simplify the differential equation of the eigenspaces of\ \ $DF$. Subsequent papers kept this choice of\ \ $F$.
 
\subsection{Choice 2 \label{app:cho2}}

Where did the coordinates\ \ $t$\ \ and\ \ $z$\ \ came from?

In \cite{Marchesin94b} were introduced coordinates\ \  $\widetilde U$,\  $\widetilde V$,\  $X$,\  $Y$,\  $Z$. In these coordinates, the wave manifold was given by the pair of equations\ \ $(1-Z^{^{2}})\widetilde V-Z\widetilde U+c=0$; $Y=ZX$. Both natural choices($\widetilde U$,\  $Z$,\  $X$\  and\  $\widetilde V$,\  $Z$,\  $X$) of coordinates for the wave manifold were used.

In this paper, we are treating case IV in the classification of \cite{Schaeffer87}. In this case the secondary bifurcation is contained in plane\ \ $Z=0$, in order to stay away from it we use\ \ $z=\frac{1}{Z}$. For technical reasons, we replace\ \  $\widetilde V$\ \ by\ \ $V_{_{1}}= \widetilde V +a_{_{2}}$. So, the equations of the wave manifold become\ \ $(z^{^{2}}-1)V_{_{1}}-z\widetilde U+c=0$;\ $X=zY$.

We would like to use coordinates which would be valid for the whole wave manifold, except the plane\ \ $z=\infty$\ \ which is formed by Hugoniot curve of points in the secondary bifurcation. 

We begin writing the equation of characteristic surface\ \ $\mathcal{C}$,\ \ $(z^{^{2}}-1)V_{_{1}}-z\widetilde U+c=0$, in\ \ $(\widetilde U$,\  $V_{_{1}}$, $z)$-space and add the coordinate\ \ $Y$. 

We recall that\ \ $\mathcal{C}$\ \ is a ruled surface, \emph{i.e.}, for fixed\ \ $z$\ \ we have a horizontal line in\ \ $(\widetilde U, V_{_{1}},z)$-space. In the\ \ $(\widetilde U, V_{_{1}})$-plane, the line is in the direction of vector\ \ $(z^{^{2}}-1,z)$. We also know that coincidence curve is the singular set of the projection\ \ $(\widetilde U, V_{_{1}},z) \longmapsto (\widetilde U, V_{_{1}},0)$ restricted to\ \ $\mathcal{C}$. Putting\ \ $h(\widetilde U, V_{_{1}},z)= (z^{^{2}}-1)V_{_{1}}-z\widetilde U+c$, a parametrization of the coincidence curve is obtained solving the linear system 
\begin{equation}
	\left\{
	\begin{array}{l}
	h(\widetilde U, V_{_{1}},z)=0	\\\\
		\displaystyle \frac{dh(\widetilde U, V_{_{1}},z)}{dz}=0,
	\end{array}
	\right.
\end{equation}
as a system in\ \ $\widetilde U$,\  $V_{_{1}}$, obtaining\ \ $\widetilde U=(2cz)/(z^{^{2}}+1)$\ \ and\ \ $V_{_{1}}=c/(z^{^{2}}+1)$. In order to introduce the coordinate\ \ $t$, we take the rules of surface\ \ $\mathcal{C}$\ \ starting at the coincidence curve. For fixed\ \ $z$\ \ the coordinate\ \ $t$\ \ measures how much the point moves away from the coincidence curve in the direction of the rules, \emph{ i.e.}, $\widetilde U = (2cz)/(z^{^{2}}+1) +t(z^{^{2}}-1)$\ \ and\ \ $V_{_{1}}=c/(z^{^{2}}+1) +tz$.
 
\subsection*{Acknowledgements}
The authors are grateful to Bradley Plohr for many enlightening discussions, mainly concerning the use of lax conditions and the use of the ELI software to check our results.
 
\bibliographystyle{amsplain} 
\def\Canic{\v{C}ani\'c}\def\Canic{\v{C}ani\'c}\def\Zoladek{\.Zo\l{}adek}\def\Freistuhler{Freist\"uhler}\def\Freistuhler{Freist\"uhler}\def\Freistuhler{Freist\"uhler}\def\Freistuhler{Freist\"uhler}\def\Freistuhler{Freist\"uhler}\def\Zoladek{\.Zo\l{}adek}
\providecommand{\bysame}{\leavevmode\hbox to3em{\hrulefill}\thinspace}
\providecommand{\MR}{\relax\ifhmode\unskip\space\fi MR }
\providecommand{\MRhref}[2]{%
  \href{http://www.ams.org/mathscinet-getitem?mr=#1}{#2}
}
\providecommand{\href}[2]{#2}

\end{document}